\documentclass[12pt]{amsart}
\usepackage[usenames]{color}
\usepackage{a4wide,amssymb,amsmath,url}

\begin{document}

\newcommand{\iii}{{\mathrm i}}
\newcommand{\Ht}{H}
\newcommand{\data}{\mathcal{D}}
\newcommand{\Irr}{\operatorname{Irr}}
\newcommand{\fctr}[1]{G_{#1}}                
\newcommand{\PW}{\mathcal{PW}}
\newcommand{\PB}{polynomially bounded}
\newcommand{\Casimir}{\Omega}
\newcommand{\der}{\operatorname{der}}
\newcommand{\gen}[1]{\left<{#1}\right>}
\newcommand{\primi}{\mathfrak{p}}
\newcommand{\unip}{\operatorname{unip}}
\newcommand{\F}{\mathbb{F}}
\newcommand{\FFF}{{\mathcal F}}
\newcommand{\Hecke}{\mathcal{H}}
\newcommand{\types}{\mathcal{F}}
\newcommand{\prpr}[1]{\hat #1}
\newcommand{\dmin}{{d_{\min}}}
\newcommand{\TWN}{(TWN)}
\newcommand{\BD}{(BD)}
\newcommand{\ucirc}{\mathbf{S^1}}
\newcommand{\param}{\Lambda}
\newcommand{\C}{{\mathbb C}}
\newcommand{\N}{{\mathbb N}}
\newcommand{\R}{{\mathbb R}}
\newcommand{\Z}{{\mathbb Z}}
\newcommand{\Q}{{\mathbb Q}}
\newcommand{\Id}{\operatorname{Id}}
\newcommand{\A}{{\mathbb A}}
\newcommand{\srts}{\Delta}
\newcommand{\levis}{{\mathcal L}}
\newcommand{\AF}{{\mathcal A}}
\newcommand{\K}{\mathbf{K}}
\newcommand{\plnch}{\operatorname{pl}}
\newcommand{\bs}{\backslash}
\newcommand{\temp}{\operatorname{temp}}
\newcommand{\disc}{\operatorname{disc}}
\newcommand{\cusp}{\operatorname{cusp}}
\newcommand{\spec}{\operatorname{spec}}
\newcommand{\sprod}[2]{\left\langle#1,#2\right\rangle}
\renewcommand{\Im}{\operatorname{Im}}
\renewcommand{\Re}{\operatorname{Re}}
\newcommand{\Ind}{\operatorname{Ind}}
\newcommand{\tr}{\operatorname{tr}}
\newcommand{\Ad}{\operatorname{Ad}}
\newcommand{\Lie}{\operatorname{Lie}}
\newcommand{\Hom}{\operatorname{Hom}}
\newcommand{\Ker}{\operatorname{Ker}}
\newcommand{\vol}{\operatorname{vol}}
\newcommand{\Area}{\operatorname{Area}}
\newcommand{\SL}{\operatorname{SL}}
\newcommand{\GL}{\operatorname{GL}}
\newcommand{\card}[1]{\lvert#1\rvert}
\newcommand{\abs}[1]{\left|#1\right|}
\newcommand{\sobnorm}[1]{\lVert#1\rVert}
\newcommand{\norm}[1]{\lVert#1\rVert}
\newcommand{\one}{\mathbf 1}
\newcommand{\aaa}{\mathfrak{a}}
\newcommand{\eps}{\epsilon}
\newcommand{\ad}{\operatorname{ad}}
\newcommand{\proj}{\operatorname{proj}}
\newcommand{\rest}{\big|}
\newcommand{\dsum}{\oplus}
\newcommand{\dtup}{\mathcal{X}}
\newcommand{\bases}{\mathfrak{B}}
\newcommand{\PPP}{\mathcal{P}}
\newcommand{\rts}{\Sigma}
\newcommand{\bss}{\underline{\beta}}
\newcommand{\univ}{\mathcal{U}}
\newcommand{\modulus}{\delta}
\newcommand{\AAA}{A}
\newcommand{\LieG}{\mathfrak{g}}
\newcommand{\zzz}{\mathfrak{z}}
\newcommand{\swrz}{\mathcal{C}}
\newcommand{\inorm}{\operatorname{N}}
\newcommand{\nnn}{\mathfrak{n}}
\newcommand{\fin}{\operatorname{fin}}
\newcommand{\uuu}{\mathfrak{u}}
\newcommand{\level}{\operatorname{level}}
\newcommand{\SC}{\operatorname{sc}}
\newcommand{\scprj}{p^{\SC}}
\newcommand{\spltrs}{\mathbf{S}}

\newcommand{\sm}[4]{\left(\begin{smallmatrix}{#1}&{#2}\\{#3}&{#4}\end{smallmatrix}\right)}

\newtheorem{theorem}{Theorem}[section]
\newtheorem{lemma}[theorem]{Lemma}
\newtheorem{proposition}[theorem]{Proposition}
\newtheorem{remark}[theorem]{Remark}
\newtheorem{conjecture}[theorem]{Conjecture}
\newtheorem{definition}[theorem]{Definition}
\newtheorem{corollary}[theorem]{Corollary}
\newtheorem{example}[theorem]{Example}

\title[Limit multiplicities]{Limit multiplicities for principal congruence subgroups of $\GL (n)$ and $\SL(n)$}
\author{Tobias Finis}
\address{Freie Universit\"at Berlin, Institut f\"ur Mathematik, Arnimallee 3, D-14195 Berlin,
Germany}
\thanks{Authors partially sponsored by grant \# 964-107.6/2007 from the German-Israeli Foundation for Scientific Research and Development.
First named author supported by DFG Heisenberg grant \# FI 1795/1-1.}
\email{finis@math.fu-berlin.de}
\author{Erez Lapid}
\address{Einstein Institute of Mathematics, The Hebrew University of Jerusalem, Jerusalem, 91904, Israel,
\and Department of Mathematics, The Weizmann Institute of Science, Rehovot 76100, Israel}
\email{erez.m.lapid@gmail.com}
\author{Werner M\"{u}ller}
\address{Mathematisches Institut, Rheinische Friedrich-Wilhelms-Universit\"{a}t Bonn, Endenicher Allee 60, D-53115 Bonn, Germany}
\email{mueller@math.uni-bonn.de}
\date{\today}


\begin{abstract}
We study the limiting behavior of the discrete spectra associated to the principal
congruence subgroups of a reductive group over a number field.
While this problem is well understood in the cocompact case (i.e., when the
group is anisotropic modulo the center), we treat groups of unbounded rank.
For the groups $\GL(n)$ and $\SL(n)$ we show that the
suitably normalized spectra converge to the Plancherel measure
(the limit multiplicity property). For general reductive groups we obtain
a substantial reduction of the problem.
Our main tool is the recent refinement of the spectral side of Arthur's
trace formula obtained in \cite{MR2811597, MR2811598}, which allows us to show
that for $\GL(n)$ and $\SL(n)$ the contribution of the continuous spectrum is negligible in the limit.
\end{abstract}

\maketitle

\setcounter{tocdepth}{1}
\tableofcontents

\section{Introduction} \label{SectionIntro}
Let (for now) $G$ be a connected linear semisimple Lie group with a fixed choice of a Haar measure.
Since the group $G$ is of type I, we can write
unitary representations of $G$ on separable Hilbert spaces as direct integrals (with multiplicities) over
the unitary dual $\Pi (G)$, the set of isomorphism classes
of irreducible unitary representations of $G$ with the Fell topology (cf.~\cite{MR0246136}).
An important case is the regular representation of $G \times G$ on $L^2(G)$, which
can be decomposed as the direct integral of the tensor products $\pi \otimes \pi^\ast$ against the
\emph{Plancherel measure} $\mu_{\plnch}$ on $\Pi(G)$.
The support of the Plancherel measure is called the \emph{tempered dual} $\Pi(G)_{\temp} \subset \Pi(G)$.

Other basic objects of interest are the regular representations $R_\Gamma$ of $G$ on $L^2(\Gamma\bs G)$ for lattices $\Gamma$ in $G$.
We will focus on the discrete part $L^2_{\disc}(\Gamma\bs G)$ of $L^2(\Gamma\bs G)$, namely the sum of all irreducible
subrepresentations, and we denote by $R_{\Gamma,\disc}$ the corresponding restriction of $R_\Gamma$.
For any $\pi\in\Pi(G)$ let $m_\Gamma(\pi)$ be the multiplicity of $\pi$ in $L^2(\Gamma\bs G)$. Thus,
\[
m_\Gamma(\pi)=\dim\Hom_G(\pi,R_\Gamma)=\dim\Hom_G(\pi,R_{\Gamma,\disc}).
\]
These multiplicities are known to be finite \cite[Theorem 3.3]{MR643242}, at least under a weak reduction-theoretic assumption on $G$ and $\Gamma$
[ibid., p. 62], which is satisfied if $G$ has no compact factors or if $\Gamma$ is arithmetic.
We define the discrete spectral measure on $\Pi(G)$ with respect to $\Gamma$ by
\[
\mu_\Gamma=\frac1{\vol(\Gamma\bs G)}\sum_{\pi\in\Pi(G)}m_\Gamma(\pi)\delta_\pi,
\]
where $\delta_\pi$ is the Dirac measure at $\pi$.
While one cannot hope to describe the multiplicity functions $m_\Gamma$ on $\Pi (G)$ explicitly
(apart from certain special cases, for example when $\pi$ belongs to the discrete series),
it is feasible and interesting to study
asymptotic questions. The limit multiplicity problem concerns the asymptotic behavior of $\mu_\Gamma$ as
$\vol (\Gamma \bs G) \to \infty$.

To make this more explicit, we recall that up to a closed subset of Plancherel measure zero, the topological space
$\Pi(G)_{\temp}$ is homeomorphic to a countable union of
Euclidean spaces of bounded dimensions, and that under this homeomorphism the Plancherel density is given by a continuous function.
(The same is true in the case of $p$-adic reductive groups considered below.)\footnote{See \cite[\S 2.3]{MR860667} for the archimedean case
and \cite{MR1989693} for the $p$-adic case.}
We call the relatively quasi-compact subsets of $\Pi(G)$ \emph{bounded} (see \S \ref{sec: reduction} below for a more explicit description).
Note that $\mu_\Gamma (A) < \infty$ for bounded sets $A \subset \Pi (G)$
under the reduction-theoretic assumption on $G$ and $\Gamma$ mentioned above \cite{MR701563}.
By definition, a Jordan measurable subset $A$ of $\Pi(G)_{\temp}$ is a bounded set such that $\mu_{\plnch}(\partial A)=0$,
where $\partial A=\bar A-A^\circ$ is the boundary of $A$ in $\Pi(G)_{\temp}$.
A Riemann integrable function on $\Pi(G)_{\temp}$ is a bounded, compactly supported function
which is continuous almost everywhere with respect to the Plancherel measure.

Let $\Gamma_1,\Gamma_2,\dots$ be a sequence of lattices in $G$.
We say that the sequence $(\Gamma_n)$ has the limit multiplicity property if the following two conditions are satisfied:
\begin{enumerate}
\item For any Jordan measurable set $A\subset\Pi(G)_{\temp}$ we have
\[
\mu_{\Gamma_n}(A)\rightarrow\mu_{\plnch}(A)\ \ \text{as }n\rightarrow\infty.
\]
\item For any bounded set $A\subset\Pi(G)\setminus\Pi(G)_{\temp}$ we have
\[
\mu_{\Gamma_n}(A)\rightarrow0\ \ \text{as }n\rightarrow\infty.
\]
\end{enumerate}
Note that we can rephrase the first condition by requiring that
\[
\lim_{n\rightarrow\infty}\mu_{\Gamma_n}(f) = \mu_{\plnch}(f)
\]
for any Riemann integrable function (or alternatively, for any continuous compactly supported
function) $f$ on $\Pi(G)_{\temp}$.



%

A great deal is known about the limit multiplicity problem for uniform lattices, where $R_\Gamma$ decomposes discretely.
The first results in this direction were proved by DeGeorge--Wallach \cite{MR0492077, MR534759, MR562677}
for normal towers, i.e., descending sequences of finite index normal subgroups of a given uniform lattice
with trivial intersection. Subsequently, Delorme \cite{MR860667} completely resolved the limit multiplicity problem
for this case in the affirmative.
Recently, there has been a big progress in proving limit multiplicity for much more general sequences of uniform lattices \cite{MR2835886, ABBGNRS}.
In particular, families of non-commensurable lattices were considered for the first time.

In the case of non-compact quotients $\Gamma \bs G$, where the spectrum also contains a continuous part, much less is known.
Here, the limit multiplicity problem has been solved for normal towers of arithmetic lattices and
discrete series $L$-packets $A \subset\Pi(G)$ (with regular parameters) by Rohlfs--Speh \cite{MR883670}. Building on this
work, the case of singleton sets $A$ and
normal towers of congruence subgroups has been solved by Savin (\cite{MR969416}, cf.~also \cite{MR1082961}). Earlier results on the discrete series
had been obtained by DeGeorge \cite{MR671316} and Barbasch--Moscovici \cite{MR722507} for groups of
real rank one, and by Clozel \cite{MR818353} for general groups (but with a weaker statement).
The limit multiplicity problem for the entire unitary dual has been solved for the principal
congruence subgroups of $\SL_2 (\Z)$ by Sarnak \cite{Saroldnote} (cf.~\cite[p. 173]{MR803367}, \cite[\S 5]{MR1718672}).
Also, a refined quantitative version
of the limit multiplicity property for the non-tempered spectrum of the
subgroups $\Gamma_0 (N)$ has been proven by Iwaniec \cite{MR1067982}.\footnote{Recall that by Selberg's eigenvalue conjecture
the non-tempered spectrum should consist only of the trivial representation in this case.}
A partial result for certain normal towers of congruence arithmetic lattices
defined by groups of $\Q$-rank one has been shown by Deitmar and Hoffmann in \cite{MR1718672}.
Finally, generalizations to the distribution of Hecke eigenvalues have been obtained by Sauvageot \cite{MR1468833},
Shin \cite{MR3004076} and Shin--Templier \cite{Shin-Temp}.

In this paper we embark upon a general analysis of the case of non-compact quotients.
We consider the entire unitary dual and groups of unbounded rank.
The main problem is to show that the contribution
of the continuous spectrum is negligible in the limit.
This was known up to now only in the case of $\GL (2)$ (or implicitly in the very special
situation considered in \cite{MR883670} and \cite{MR969416}).
Our approach is based on a careful study of the spectral side of
Arthur's trace formula in the recent form given in \cite{MR2811597, MR2811598}.
As we shall see, this form is crucial for the analysis.
Our results are unconditional only for the groups $\GL(n)$ and $\SL(n)$,
but we obtain a substantial reduction of the problem in the general case.

Before stating our main result we shift to an adelic setting which allows one to incorporate Hecke operators into the picture
(i.e., to consider the equidistribution of Hecke eigenvalues).
Thus, let now $G$ be a reductive group defined over a number field $F$ and $S$ a finite set of places of $F$ containing the
set $S_\infty$ of all archimedean places. Let $F_S$ be the product of the completions $F_v$ for $v \in S$, $\A^S$ the
restricted product of the $F_v$ for $v \notin S$, and $\A = F_S\times\A^S$ the ring of adeles of $F$.
In the special case $S=S_\infty$ we write $F_\infty = F_{S_\infty}$ and $\A_{\fin} = \A^{S_\infty}$.
As usual, $G(F_S)^1$ denotes the intersection of the kernels of the homomorphisms
$\abs{\chi}: G (F_S) \to \R^{>0}$ where $\chi$ ranges over the $F$-rational characters of $G$ and $\abs{\cdot}$
denotes the normalized absolute value on $F^*_S$. Similarly, we define the normal subgroup $G (\A)^1$ of $G(\A)$.
Fix a Haar measure on $G(\A)$.
For any open compact subgroup $K$ of $G(\A^S)$ let $\mu_K = \mu^{G,S}_K$ be the measure on $\Pi(G(F_S)^1)$
given by
\begin{align*}
\mu_K & = \frac{1}{\vol(G(F)\bs G(\A)^1/K)} \sum_{\pi\in\Pi(G(F_S)^1)}\dim\Hom_{G(F_S)^1}(\pi,L^2(G(F)\bs G(\A)^1/K)) \, \delta_\pi \\
& = \frac{\vol(K)}{\vol(G(F)\bs G(\A)^1)} \sum_{\pi\in\Pi (G(\A)^1)} \dim\Hom_{G(\A)^1}(\pi,L^2(G(F)\bs G(\A)^1)) \, \dim (\pi^S)^K \,  \delta_{\pi_S}.
\end{align*}

We say that a collection $\mathcal{K}$ of open compact subgroups of $G (\A^S)$
has the limit multiplicity property if $\mu_K\rightarrow\mu_{\plnch}$ for $K \in \mathcal{K}$ in the sense that
\begin{enumerate}
\item for any Jordan measurable subset $A\subset\Pi(G(F_S)^1)_{\temp}$ we have
$\mu_K(A)\rightarrow\mu_{\plnch}(A)$, $K \in \mathcal{K}$, and,
\item for any bounded subset $A\subset\Pi(G(F_S)^1)\setminus\Pi(G(F_S)^1)_{\temp}$ we have
$\mu_K(A)\rightarrow0$, $K \in \mathcal{K}$.
\end{enumerate}
Here, we write for example $\mu_K(A)\rightarrow\mu_{\plnch}(A)$ to mean that for every $\epsilon>0$
there are only finitely many subgroups $K \in \mathcal{K}$ such that $\abs{\mu_K(A) - \mu_{\plnch}(A)} \ge \epsilon$.
Once again, we can rephrase the first condition by saying that for any Riemann integrable function $f$
on $\Pi(G(F_S)^1)_{\temp}$ we have
\[
\mu_K(f)\rightarrow\mu_{\plnch}(f), \quad K \in \mathcal{K}.
\]

\begin{remark}
By a well-known result of Wallach, for any reductive group $G$, any
local factor of an irreducible representation which occurs in the residual spectrum (i.e., the non-cuspidal discrete spectrum)
is necessarily non-tempered (see \cite[Theorem 4.3]{MR733320} for the archimedean and \cite[Proposition 4.10]{MR1253624} for the $p$-adic case).
Therefore, once the limit multiplicity property is established, it automatically holds for the cuspidal spectrum as well.
\end{remark}

We note that when $G$ satisfies the strong approximation property with
respect to $S_\infty$ (which is tantamount to saying that $G$ is semisimple,
simply connected and without any $F$-simple factors $H$ for which
$H(F_\infty)$ is compact \cite[Theorem 7.12]{MR1278263}) and $K$ is an open compact subgroup
of $G(\A_{\fin})$, we have
\[
G (F) \bs G (\A) / K \simeq \Gamma_K \bs G (F_\infty)
\]
for the lattice $\Gamma_K = G (F) \cap K$ in the connected semisimple Lie group $G(F_\infty)$.
Thus, these lattices are incorporated in the adelic setting.
A similar connection can be made for general $G$,
where however a single subgroup $K$ will correspond to a finite set of lattices in $G(F_\infty)$.

An important step in the analysis of the limit multiplicity problem is to reduce it to a question about the trace formula.
This is non-trivial not the least because of the complicated nature of the unitary dual.
This step was carried out by Delorme in the case where $S$ consists of the archimedean
places \cite{MR860667}. His argument was subsequently extended by Sauvageot to the general case \cite{MR1468833},
where he also axiomatized the essential property as a ``density principle'' (see \S\ref{sec: reduction} below).
Using the result of Sauvageot, we can recast the limit multiplicity problem as follows.
Let $\Hecke(G(F_S)^1)$ be the algebra of smooth, compactly supported bi-$K_S$-finite functions on $G(F_S)^1$.
For any $h\in \Hecke (G(F_S)^1)$ let $\hat h$ be the function on $\Pi(G(F_S)^1)$ given by
$\hat h(\pi)=\tr\pi(h)$.
Note that we have
\[
\mu_K (\hat{h}) = \frac{1}{\vol(G(F)\bs G(\A)^1)} \tr R_{\disc}(h\otimes\one_{K})
\]
and
\[
\mu_{\plnch} (\hat{h}) = h (1).
\]
Then we have the following theorem.

\begin{theorem}[Sauvageot] \label{thm: reduction}
Suppose that the collection $\mathcal{K}$ has the property that for any
function $h\in\Hecke(G(F_S)^1)$ we have
\begin{equation} \label{eq: main1}
\mu_K (\hat{h}) \rightarrow h(1), \quad K \in \mathcal{K}.
\end{equation}
Then limit multiplicity holds for $\mathcal{K}$.
\end{theorem}
We will recall how to obtain this result from Sauvageot's density principle in \S\ref{sec: reduction}.

Given this reduction, it is natural to attack assertion \eqref{eq: main1} via the trace formula.
In the cocompact case (i.e., when $G/Z(G)$ is anisotropic over $F$) one can use the Selberg trace formula.
In the general case we use Arthur's (non-invariant) trace formula which expresses a certain distribution $h\mapsto J(h)$
on $C_c^\infty(G(\A)^1)$ geometrically and spectrally \cite{MR518111, MR558260, MR681737, MR681738, MR828844, MR835041}.
The distribution $J$ depends on the choice of a maximal $F$-split torus $\spltrs_0$ of $G$ and
a suitable maximal compact subgroup $\K = \K_S \K^S$ of $G(\A)$ (cf.~\S\ref{SectionTraceFormula} below).
The main terms on the geometric side are the elliptic orbital integrals, most notably the contribution $\vol(G(F)\bs G(\A)^1) h(1)$
of the identity element.
The main term on the spectral side is $\tr R_{\disc}(h)$.

The relation \eqref{eq: main1} can now be broken down into the following two statements:
\begin{equation} \label{eq: mainspectral}
\text{For any }h\in\Hecke(G(F_S)^1)\text{ we have } J (h \otimes \one_K) - \tr R_{\disc} (h \otimes \one_K) \rightarrow0,
\end{equation}
and,
\begin{equation} \label{eq: maingeometric}
\text{for any }h\in\Hecke(G(F_S)^1)\text{ we have } J (h \otimes \one_K) \rightarrow\vol(G(F)\bs G(\A)^1) h(1).
\end{equation}
We call these assertions the \emph{spectral} and \emph{geometric limit properties}, respectively.

In the cocompact case the spectral limit property is trivial since $J (h) = \tr R_{\disc} (h)$.
Also, in this case it is easy to see that for any tower $\mathcal{K}$ of normal subgroups $K$ of $\K^S$
and for every $h\in\Hecke(G(F_S)^1)$ we have in fact
$J (h \otimes \one_K) = \vol(G(F)\bs G(\A)^1) h(1)$ for almost all $K \in \mathcal{K}$. This is Sauvageot's proof of the
limit multiplicity property in this case.

In general both properties are nontrivial. In this paper we consider
only the simplest collection of normal subgroups of $\K^S$, namely the
\emph{principal} congruence subgroups $\K^S (\mathfrak{n})$ of $\K^S$
for non-zero ideals $\mathfrak{n}$ of $\mathfrak{o}_F$ prime to $S$ (see \S \ref{SectionGeom}).
In this case, the geometric limit property is a consequence of Arthur's analysis of the unipotent contribution
to the trace formula in \cite{MR828844} (see \S \ref{SectionGeom}, in particular Corollary \ref{corgeometriclimit}).
The main task is to prove the spectral limit property for this collection of subgroups. We are able to do this unconditionally for the
groups $\GL(n)$ and $\SL(n)$, and consequently obtain the following as our main result.

\begin{theorem} \label{thm: main}
Let $G$ be either $\GL(n)$ or $\SL(n)$ over a number field $F$.
Then limit multiplicity holds for the collection of all principal congruence subgroups $\K^S (\mathfrak{n})$ of $\K^S$.
\end{theorem}

As explained above, for $G=\SL(n)$ over $F$ and (for simplicity) $S = S_\infty$ the strong approximation theorem \cite[Theorem 7.12]{MR1278263} allows for the
following reformulation of this result in terms of lattices in the semisimple Lie groups $\SL (n,F_\infty)$.

\begin{corollary}
Limit multiplicity holds for the collection of the principal congruence subgroups
\[
\Gamma (\mathfrak{n}) = \{ \gamma \in \SL (n,\mathfrak{o}_F) : \gamma \equiv 1 \pmod{\mathfrak{n}} \}
\]
of the lattice $\SL (n,\mathfrak{o}_F)$ in the semisimple Lie group $\SL (n,F_\infty)$.
\end{corollary}

The key input for our approach to the spectral limit property
is the refinement of the spectral expansion of Arthur's trace formula established in \cite{MR2811597} (cf.~Theorem \ref{thm: specexpand} below).
This result enables us to set up an inductive argument which relies on two conjectural properties, one global and one local,
which we call \TWN\ (tempered winding numbers) and \BD\ (bounded degree), respectively.
They are stated in \S \ref{SectionBounds} and are expected to hold for any reductive group $G$ over a number field.
Theorem \ref{thm: main} is proved for any group $G$ satisfying these properties (see Theorem \ref{MainTheorem}).

The global property \TWN\ is a uniform estimate on the winding number of the normalizing scalars
of the intertwining operators in the co-rank one case.
For $\GL(n)$ and $\SL(n)$ this property follows from known, but delicate, properties of the Rankin--Selberg $L$-functions (Proposition \ref{prop: mainglobal}).
In order to describe the local property \BD, recall that in the non-archimedean case the matrix coefficients of the local intertwining operators
are rational functions of $q^{-s}$, where $q$ is the cardinality of the residue field,
and that the degrees of the denominators are bounded in terms of $G$ only.
Property \BD\ gives an upper bound on the degree of the numerator in terms of the level.
This property was studied in \cite{MR3001800}, where among other things it was proved for the groups $G=\GL (n)$
(and implicitly also for $\SL(n)$).
The import of property \BD\ is that it yields a good bound for integrals of logarithmic derivatives of normalized intertwining operators
(Proposition \ref{prop: mainlocal}).
The archimedean analog of property \BD\ (for a general real reductive group) had been established in \cite[Appendix]{MR2053600}.

The analysis is carried out by induction on the semisimple rank of $G$.
Actually, for the induction step it is necessary to verify that the collection of measures
$\{ \mu^{G, S_\infty}_{\K (\nnn)} \}$ is \emph{\PB} in the sense
of Definition \ref{DefinitionBounded}, a property that already shows up in Delorme's work \cite{MR860667}.
This property is analyzed in \S \ref{SectionPB}, where we prove Proposition \ref{Delormeprop},
a result on real reductive Lie groups which generalizes a part of Delorme's argument,
and is (like Delorme's work) based on the Paley--Wiener theorem of Clozel--Delorme \cite{MR1046496}.
Once we have that the collections $\{ \mu^{M, S_\infty}_{\K_M (\nnn)} \}$ are {\PB}
for all proper Levi subgroups $M$ of $G$, we can deduce the spectral limit property for $G$ (Corollary \ref{corspectrallimit}).

We end this introduction with a few remarks on possible extensions of Theorems \ref{thm: main} and \ref{MainTheorem}.
For general sequences $(\Gamma_n)$ of distinct irreducible lattices in a semisimple Lie group $G$,
there is an obvious obstruction to the limit multiplicity property, namely the possibility that
the lattices $\Gamma_n$ (or an infinite subsequence thereof) all contain a non-trivial subgroup $\Delta$ of the
center of $G$, which forces the corresponding representations $R_{\Gamma_n}$ to be $\Delta$-invariant.
By passing to the quotient $G / \Delta$, we can assume that this is not the case.
A less obvious obstruction is that the members of an infinite subsequence of $(\Gamma_n)$ all contain a non-central normal subgroup of $\Gamma_1$
(necessarily of infinite index).
In such a case the analog of the geometric limit property \eqref{eq: maingeometric} fails.
Indeed, for $G = \SL_2 (\R)$ we can find a descending sequence of finite index normal subgroups $\Gamma_n$ of $\Gamma=\SL_2(\Z)$
such that for all $n$ the multiplicity in $L^2(\Gamma_n\bs G)$ of either one of the two lowest discrete series representations of $G$
(or equivalently, the genus of the corresponding Riemann surface) is equal to one \cite{MR0163966}.
Similarly, one can find a descending sequence of finite index normal subgroups $\Gamma_n$ of $\SL_2(\Z)$ such that the limiting measure of the sequence $(\mu_{\Gamma_n})$
has a strictly positive density on the entire complementary spectrum $\Pi (G) \setminus \Pi (G)_{\temp}$ \cite{MR1091611}.\footnote{
It follows from \cite{MR883670} (or alternatively by direct calculation)
that the limit multiplicity property holds for the discrete series of $\SL_2 (\R)$ and arbitrary normal towers of subgroups of $\SL_2 (\Z)$,
i.e., when the intersection of the normal subgroups is trivial.}
By Margulis's normal subgroup theorem, non-central normal subgroups of infinite index do not exist for irreducible lattices $\Gamma$ in
semisimple Lie groups $G$ of real rank at least two and without compact factors (\cite[p. 4, Theorem 4']{MR1090825}, cf.~also [ibid., IX.6.14]).
(The paper \cite{ABBGNRS} is a major outgrowth of the Margulis normal subgroup theorem.)
One expects that for irreducible arithmetic lattices, the limit multiplicity property holds at least for any sequence of distinct \emph{congruence} subgroups not
containing non-trivial central elements.
In the adelic setting, let $G$ be a reductive group defined over a number field $F$ such that
the derived group $G^{\der}$ of $G$ is $F$-simple and simply connected.
Then we expect the limit multiplicity property to be true for a collection $\mathcal{K}$ of open compact subgroups
of $G (\A^S)$, if $\vol (K \cap G^{\der} (\A^S)) \to 0$ for $K \in \mathcal{K}$ and every non-trivial
element of the center of $G(F)$ is contained in only finitely many members of $\mathcal{K}$.
For this, a good understanding of the structure
of these subgroups seems to be necessary both to deal with the geometric and with the spectral side.
We hope to return to this problem in a future paper.

We thank the Centre Interfacultaire Bernoulli, Lausanne,
and the Max Planck Institute for Mathematics, Bonn, where a part of this paper was worked out.
We thank Joseph Bernstein, Laurent Clozel, Wee Teck Gan and Peter Sarnak for useful discussions.

\section{Sauvageot's density principle} \label{sec: reduction}

In this section we recall the results of Sauvageot \cite{MR1468833} and the proof of Theorem \ref{thm: reduction}, providing a close link between
the limit multiplicity problem and the trace formula. We continue to use the notation introduced before Theorem \ref{thm: reduction}.
Recall that a bounded subset $A \subset \Pi (G(F_S)^1)$ is a relatively quasi-compact subset. Equivalently,
$A \subset \Pi (G(F_S)^1)$ is bounded if the archimedean
infinitesimal characters $\chi_{\pi_\infty}$ of the elements $\pi \in A$ are bounded and
there exists an open compact subgroup $K \subset G(F_{S - S_\infty})$ such that every $\pi \in A$ contains
a non-trivial $K$-fixed vector.

The main result of \cite{MR1468833} (Corollaire 6.2 and Th\'{e}or\`{e}me 7.3) is the following.\footnote{See the appendix of \cite{MR3004076} for important corrections.}

\begin{theorem}[Sauvageot] \label{thm: density principle}
Let $\epsilon>0$ be arbitrary. Then
\begin{enumerate}
\item For any bounded set $A\subset\Pi(G(F_S)^1)\setminus\Pi_{\temp}(G(F_S)^1)$
there exists $h\in\Hecke(G(F_S)^1)$ such that
\begin{enumerate}
\item $\hat h(\pi)\ge0$ for all $\pi\in\Pi(G(F_S)^1)$,
\item $\hat h(\pi)\ge1$ for all $\pi\in A$,
\item $h(1)=\mu_{\plnch}(\hat h)<\epsilon$.
\end{enumerate}
\item For any Riemann integrable function $f$ on $\Pi_{\temp}(G(F_S)^1)$
there exist $h_1,h_2\in\Hecke(G(F_S)^1)$ such that
\begin{enumerate}
\item $\abs{f(\pi)-\hat h_1(\pi)}\le\hat h_2(\pi)$ for all $\pi\in\Pi(G(F_S)^1)$,
where we extend $f$ by zero to the entire unitary dual $\Pi (G(F_S)^1)$.
\item $h_2(1) = \mu_{\plnch}(\hat{h}_2) <\epsilon$.
\end{enumerate}
\end{enumerate}
\end{theorem}

As in \cite{MR1468833}, this result easily implies Theorem \ref{thm: reduction}.
We recall the argument.
Let $A\subset\Pi(G(F_S)^1)\setminus\Pi_{\temp}(G(F_S)^1)$ be a bounded set.
For any $\epsilon>0$ let $h\in\Hecke(G(F_S)^1)$ be as in the first part of Theorem \ref{thm: density principle}.
By assumption we have $\abs{\mu_K(\hat{h})-h(1)}<\epsilon$ for all but finitely many $K \in \mathcal{K}$.
For all such $K$ we have
\[
\mu_K (A)\le\mu_K(\hat h)\le\abs{\mu_K(\hat h)-h(1)}+h(1)<2\epsilon.
\]

Similarly, let $f$ be a Riemann integrable function on $\Pi_{\temp}(G(F_S)^1)$. For any $\epsilon>0$ let
$h_1$ and $h_2$ be as in the second part of Theorem \ref{thm: density principle}.
By assumption, for all but finitely many $K \in \mathcal{K}$ we have $\abs{\mu_K(\hat{h}_i)-h_i(1)}<\epsilon$, $i = 1$, $2$.
Using $\mu_{\plnch} (\hat{h}_i) = h_i (1)$, we obtain
\begin{align*}
\abs{\mu_K(f)-\mu_{\plnch}(f)} & \le
\abs{\mu_K(f)-\mu_K(\hat h_1)}+\abs{\mu_K(\hat h_1)-h_1(1)}+\abs{h_1(1)-\mu_{\plnch}(f)}\\
&\le \abs{\mu_K(\hat h_1)-h_1(1)}+\mu_K(\hat h_2)+h_2(1)\\
&\le\abs{\mu_K(\hat h_1)-h_1(1)}+\abs{\mu_K(\hat h_2)-h_2(1)}+2h_2(1)<4\epsilon.
\end{align*}
Theorem \ref{thm: reduction} follows.

\section{The geometric limit property} \label{SectionGeom}

In this section we prove the geometric limit property for the principal congruence subgroups $\K^S (\nnn)$,
where $S$ is a finite set of places of $F$ containing $S_\infty$.
In fact, we obtain a somewhat more precise estimate (cf.~Proposition \ref{unipproposition} below), which
will be useful in the inductive argument of \S \ref{sec: spectralside}.

\subsection{Notation} \label{subsecnotationgeom}
We will mostly use the notation of \cite{MR2811597}.
As before, $G$ is a reductive group defined over a number field $F$ and $\A$ is the ring of adeles of $F$. Denote the adele norm
on $\A^\times$ by $\abs{\cdot}_{\A^\times}$.
For a finite place $v$ of $F$ let $q_v$ be the cardinality of the residue field of $v$.
We write $F_\infty=F\otimes\R$ and $\A_{\fin}$ for the ring of finite adeles.
As above, we fix a maximal compact subgroup $\K=\prod_v \K_v = \K_\infty \K_{\fin}$ of
$G(\A)=G(F_\infty)G(\A_{\fin})$. Let $G (\A)^1$ be the intersection
of the subgroups $\ker \abs{\chi}_{\A^\times}$ of $G (\A)$ as
$\chi$ ranges over the $F$-rational characters of $G$.

Fix once and for all a faithful $F$-rational representation $\rho: G \to \GL (V)$ and an $\mathfrak{o}_F$-lattice $\Lambda$ in the representation space
$V$ such that the stabilizer of $\hat{\Lambda} = \hat{\mathfrak{o}}_F \otimes \Lambda \subset \A_{\fin} \otimes V$ in $G(\A_{\fin})$ is the group $\K_{\fin}$.
(Since the maximal compact subgroups of $\GL (\A_{\fin} \otimes V)$ are precisely the stabilizers of lattices, it is easy to see that such a lattice exists.)
For any non-zero ideal $\nnn$ of $\mathfrak{o}_F$ let
\[
\K (\nnn) = \K_G(\nnn)=\{ g \in G (\A_{\fin}) \, : \, \rho (g) v \equiv v \pmod{\nnn \hat{\Lambda}}, \quad v \in \hat{\Lambda} \}
\]
be the principal congruence subgroup of level $\nnn$, a factorizable normal open subgroup of $\K_{\fin}$.
The groups $\K (\nnn)$ form a neighborhood base of the identity element in $G(\A_{\fin})$.
We denote by $\inorm (\nnn) = [ \mathfrak{o}_F : \nnn]$ the ideal norm of $\nnn$.
Similarly, for a finite set $S \supset S_\infty$ of places of $F$ and an ideal $\nnn$ prime to $S$ let $\K^S (\nnn)=\K_G^S(\nnn)$
be the corresponding open normal subgroup of $\K^S = \prod_{v \notin S} \K_v$.

Throughout, unless otherwise mentioned, all algebraic subgroups of $G$ that we will consider are implicitly assumed to be defined over $F$.

We fix a maximal $F$-split torus $\spltrs_0$ of $G$ and let $M_0$ be its centralizer, which is a minimal Levi subgroup.
We assume that the maximal compact subgroup $\K \subset G (\A)$ is admissible with respect to $M_0$ \cite[\S 1]{MR625344}.
Denote by $\AAA_0$ the identity component of $\spltrs_0(\R)$, which is viewed as a subgroup of
$\spltrs_0(\A)$ via the diagonal embedding of $\R$ into $F_\infty$.

Denote by $\aaa_0^*$ the $\R$-vector space spanned by the lattice $X^*(M_0)$ of $F$-rational characters of $M_0$
(or equivalently by the lattice $X^*(\spltrs_0)$).
We write $\aaa_0$ for the dual space of $\aaa_0^*$, which is spanned by the co-characters of $\spltrs_0$.
More generally, for a Levi subgroup $M\supset M_0$ we write $\spltrs_M$ for the split part of the identity component of the center of $M$
and set $\AAA_M = \AAA_0 \cap Z (M) =\AAA_0\cap\spltrs_M(\R)$.

We will use the notation $A\ll B$ to mean that there exists a constant $c$ (independent of the parameters under consideration) such that $\abs{A}\le cB$.
The implied constant may depend on $G$ and $\rho$, as well as on the field $F$.
If it depends on additional parameters (e.g., $\epsilon$), we write $A\ll_\epsilon B$.

\subsection{The geometric side of the trace formula}

Arthur's trace formula provides two alternative expressions for a certain distribution $J$ on $G (\A)^1$ which
depends on the choice of $M_0$ and $\K$. Let $d_0$ be the semisimple $F$-rank of $G$.
For $h \in C^\infty_c (G (\A)^1)$, Arthur defines $J(h)$ to be the
value at the point $T=T_0$ specified in \cite[Lemma 1.1]{MR625344}
of a certain polynomial $J^T (h)$ on $\aaa_0$ of degree at most $d_0$.
The polynomial $J^T(h)$ depends on the additional choice of a parabolic subgroup $P_0$ of $G$ with Levi part $M_0$, which we fix throughout.
Consider the equivalence relation on $G(F)$ defined by $\gamma \sim \gamma'$ whenever the semisimple parts of $\gamma$ and
$\gamma'$ are $G(F)$-conjugate, and denote by $\mathcal{O}$ the set of all resulting equivalence classes.
They are indexed by (but are not identical with) the conjugacy classes of semisimple elements
of $G(F)$.
The coarse geometric expansion \cite{MR518111} is
\begin{equation} \label{geometricside}
J^T (h) = \sum_{\mathfrak{o} \in \mathcal{O}} J^T_\mathfrak{o} (h),
\end{equation}
where the summands $J^T_\mathfrak{o} (h)$ are again polynomials in $T$ of degree at most $d_0$.
Write $J_\mathfrak{o} (h) = J^{T_0}_\mathfrak{o} (h)$, which depends only on $M_0$ and $\K$.
Then $J_\mathfrak{o} (h) = 0$ if the support of $h$ is disjoint from all conjugacy classes of $G(\A)$
intersecting $\mathfrak{o}$ (cf.~\cite[Theorem 8.1]{MR835041}). Let $\Omega\subset G(\A)^1$ be a compact set and denote by $C_\Omega^\infty(G(\A)^1)$
the space of smooth functions on $G(\A)^1$ supported in $\Omega$.
By [ibid., Lemma 9.1] (together with the descent formula of \cite[\S 2]{MR625344}),
there exists a finite subset $\mathcal{O} (\Omega) \subset \mathcal{O}$ such that
for $h \in C_{\Omega}^\infty (G(\A)^1)$ we may restrict summation in \eqref{geometricside} to $\mathfrak{o} \in \mathcal{O} (\Omega)$.
In particular, the sum is always finite.
When $\mathfrak{o}$ consists of the unipotent elements of $G(F)$, we write $J^T_{\unip} (h)$ for $J^T_\mathfrak{o} (h)$.

For each $k \ge 0$ fix a basis $\mathcal{B}_k$ of $\univ(\Lie G_\infty \otimes \C)_{\le k}$,
equipped with the usual filtration, and set
\[
\norm{h}_k =\sum_{X\in\mathcal{B}_k}\norm{X\star h}_{L^1(G(\A)^1)}
\]
for functions $h\in C_c^\infty(G(\A)^1)$,
where we view $X$ as a left-invariant differential operator on $G (F_\infty)$.
For a compact subset $\Omega \subset G (\A)^1$
the norms $\norm{\cdot}_k$ give $C^\infty_\Omega (G(\A)^1)$ the structure of a Fr\'echet space. (Note here that it is equivalent to use
the seminorms $\sup_{x \in \Omega} \abs{(X \star h)(x)}$ for $X \in \univ(\Lie G_\infty \otimes \C)$ instead of the norms
$\norm{h}_k$, $k \ge 0$.)

Analogously, we set
\[
\norm{h}_k =\sum_{X\in\mathcal{B}_k}\norm{X\star h}_{L^1(G(F_S)^1)}
\]
for $k \ge 0$ and $h\in C_c^\infty(G(F_S)^1)$. As above, for a compact subset $\Omega_S \subset G(F_S)^1$ these norms give $C_{\Omega_S}^\infty(G(F_S)^1)$
the structure of a Fr\'echet space.


\subsection{An estimate for the unipotent contribution}
By \cite[Theorem 4.2]{MR828844}, the unipotent contribution $J^T_{\unip}$ can be split into the contributions of the finitely many
$G(\bar F)$-conjugacy classes of unipotent elements of $G(F)$.
By [ibid., Corollary 4.4], the contribution of the unit element is simply the constant polynomial $\vol (G (F) \bs G (\A)^1) h (1)$.
Write
\[
J^T_{\unip - \{ 1 \}} (h) = J^T_{\unip} (h) - \vol (G (F) \bs G (\A)^1) h (1), \quad h \in C^\infty_c (G (\A)^1).
\]
We want to estimate $J_{\unip - \{ 1 \}}(h) = J^{T_0}_{\unip - \{ 1 \}}(h)$ for $h = h_S \otimes \one_{K^S (\nnn)}$.

\begin{proposition} \label{unipproposition}
There exists an integer $k\ge0$ such that for any compact subset $\Omega_S\subset G(F_S)^1$
and any integral ideals $\nnn_S$ and $\nnn$ of $\mathfrak{o}_F$, where $\nnn_S$ is a product of prime ideals of places in $S$ and $\nnn$ is prime to $S$, we have
\[
\abs{J_{\unip - \{ 1 \}} (h_S \otimes \one_{\K^S (\nnn)})} \ll_{\Omega_S} \frac{(1 + \log \inorm (\nnn_S \nnn))^{d_0}}{\inorm (\nnn)} \norm{h_S}_k
\]
for any bi-$\K_{S-S_\infty} (\nnn_S)$-invariant function $h_S \in C^\infty_{\Omega_S} (G (F_S)^1)$.
\end{proposition}

\begin{remark}
Let $G = \GL (2)$, $\K (\nnn)$ the standard principal congruence subgroups,
and assume for simplicity that $S = S_\infty$. Then we have the explicit formula
\begin{align*}
J_{\unip - \{ 1 \}} (h_\infty \otimes \one_{\K (\nnn)}) & = \frac{\vol (M_0(F)\bs M_0(\A)^1)}{\inorm (\nnn)}
\left( \int_{F_\infty} \int_{\K_\infty} h_\infty (k^{-1}\sm 1x01k) \log \abs{x}_\infty \ dk \, dx \right. \\
& + \left. (\gamma_F - \log \inorm (\nnn)) \int_{F_\infty} \int_{\K_\infty} h_\infty (k^{-1}\sm 1x01k) \ dk \, dx \right),
\end{align*}
where $\gamma_F = c_{0,F} / c_{-1,F}$ is the quotient of the two leading
coefficients in the Laurent expansion $\zeta_F (s) = c_{-1,F} (s-1)^{-1} +
c_{0,F} + \cdots$ of the Dedekind zeta function of $F$ at $s=1$ (cf.~\cite{MR546600, MR0401654}). This shows that
(regarding the dependency on $\inorm (\nnn)$) the estimate of Proposition \ref{unipproposition} is best possible in this case.
For general groups we will give an improved estimate in Proposition \ref{unippropositionrefined} below.
\end{remark}

Proposition \ref{unipproposition} will be proved below.
It has the following consequence.

\begin{corollary}[Geometric limit property] \label{corgeometriclimit}
For any $h_S \in C^\infty_c (G (F_S)^1)$ we have
\[
\lim_{\nnn} J (h_S \otimes \one_{\K^S (\nnn)}) = \vol (G (F) \bs G (\A)^1) h_S (1).
\]
\end{corollary}

\begin{proof}
Fix $h_S \in C^\infty_c (G (F_S)^1)$ and
let $\Omega_S \subset G (F_S)^1$ be the support of $h_S$.
Then the support of the test function $h_S \otimes \one_{\K^S (\nnn)}$ is $\Omega_S \K^S (\nnn)\subset\Omega_S \K^S$
and therefore there are only finitely many classes $\mathfrak{o} \in \mathcal{O}$ that contribute to the geometric
side of the trace formula \eqref{geometricside} for the functions $h_S \otimes \one_{\K^S (\nnn)}$. Moreover,
the only class $\mathfrak{o} \in \mathcal{O}$
for which the union of the $G(\A)$-conjugacy classes of elements of $\mathfrak{o}$ meets
$G (F_S) \K^S (\nnn)$ for infinitely many ideals $\nnn$ is the unipotent class. For assume that $\mathfrak{o}$ has this property and
let $f \in F[X]$ be the characteristic polynomial
of the linear map $\rho (\gamma) - 1 \in \operatorname{End} (V)$ for arbitrary $\gamma \in \mathfrak{o}$. The assumption
on $\mathfrak{o}$ implies that
every coefficient of $f$ (except the leading coefficient $1$)
is either arbitrarily close to $0$ at some place $v \notin S$
or has absolute value $<1$ at infinitely many places. Therefore, necessarily $f = X^{\dim V}$ and $\gamma$ is unipotent.

As a result, the geometric side reduces to $J_{\unip} (h_S \otimes \one_{\K^S (\nnn)})$ for all but finitely many ideals $\nnn$,
and the assertion follows from Proposition \ref{unipproposition}.
\end{proof}

The proof of Proposition \ref{unipproposition} consists of a slight extension of Arthur's arguments in \cite{MR828844}.
The case where $F = \Q$ and $\nnn$ is a power of a fixed prime is in fact already covered by Arthur's arguments.
We also remark that when we restrict the prime divisors of $\nnn$ to a fixed finite set, we can appeal directly to Arthur's fine
geometric expansion \cite{MR835041} to obtain the geometric limit property (cf.~\cite[Proposition 1.7]{MR1718672}).

We first quote Arthur's asymptotic formula for $J^T_{\unip - \{ 1 \}}$ \cite{MR828844} in a form suitable for our purposes.
Let $\mathcal{U} \subset G$ be the unipotent variety of $G$, so that $\mathcal{U}(F)$ consists of the unipotent elements of $G(F)$.
Fix a Euclidean norm $\norm{\cdot}$ on $\aaa_0$ which is invariant under the Weyl group and
let $d (T) = \min_{\alpha \in \srts_0} \sprod{\alpha}{T}$ for $T \in \aaa_0$.
Here $\srts_0$ is the set of simple roots of $\spltrs_0$ with respect to $P_0$.
For a parabolic subgroup $P \supset P_0$ with Levi subgroup $M\supset M_0$ write $A_P = A_M$ and set
$A_P (T_1) = \{ a \in A_M \, : \, \log\alpha(a)>\sprod{\alpha}{T_1} \quad \forall \alpha \in \srts_P \}$ for $T_1 \in \aaa_0$,
where $\srts_P$ are the simple roots of $\spltrs_M$ with respect to $P$ (viewed as elements of $\aaa_0^*$).
As in \cite[p. 941]{MR518111}, we fix a suitable vector $T_1$,
which depends only on $G$, $P_0$ and $\K$, such that $G (\A) = G(F)U_0 (\A) M_0 (\A)^1 A_{P_0} (T_1) \K$.
Finally, recall the truncation function $F (\cdot, T) = F^G (\cdot, T)$ for $T \in \aaa_0$, which is the characteristic function of a
certain compact subset of $G(F) \bs G (\A)^1$ depending on $T$ (\cite[p. 941]{MR518111}, \cite[p. 1242]{MR828844}).

\begin{theorem}(\cite[Theorem 4.2]{MR828844}) \label{TheoremArthur}
There exist $k,m\ge0$ and $\epsilon>0$ with the following property.
For any compact set $\Omega \subset G (\A)^1$ there exists a constant $d_\Omega \ge 0$
such that for all non-zero ideals $\nnn$ of $\mathfrak{o}_F$ with
$\K (\nnn) \Omega \K (\nnn) = \Omega$ and any bi-$\K(\nnn)$-invariant
function $h \in C_\Omega^\infty (G(\A)^1)$ we have
\begin{multline*}
\abs{ J^T_{\unip - \{ 1 \}} (h) - \int_{G(F) \bs G (\A)^1} F (x,T) \sum_{\gamma \in \mathcal{U} (F), \, \gamma \ne 1}
h(x^{-1} \gamma x) \, dx } \\\ll_{\Omega} \norm{h}_k \inorm (\nnn)^m (1+\norm{T})^{d_0}e^{-\epsilon d(T)}
\end{multline*}
for all $T \in \aaa_0$ with $d (T) \ge d_\Omega$.
\end{theorem}

Note that Theorem \ref{TheoremArthur} differs slightly from the formulation in \cite{MR828844}.
Namely, we introduced a different sequence of norms on $C_\Omega^\infty (G(\A)^1)$ (defining the same topology), we combined all the (finitely many) non-trivial geometric unipotent orbits,
made the dependence on $\nnn$ explicit, and included the factor $(1+\norm{T})^{d_0}$ instead of assuming that $\norm{T}\ge\epsilon_0d(T)$
for a suitable $\epsilon_0 > 0$.
The last change is allowed, since [ibid., Theorem 4.2] is based on [ibid., Theorem 3.1], where one can clearly make the corresponding change
(see [ibid., p. 1249] for the relevant part of Arthur's proof).

Next, we bound the truncated integral
\[
\int_{G(F) \bs G (\A)^1} F (x,T) \sum_{\gamma \in \mathcal{U} (F), \, \gamma \ne 1}
(h_S \otimes \one_{\K^S (\nnn)}) (x^{-1} \gamma x) \, dx
\]
in terms of $\inorm(\nnn)$. By the dominated convergence theorem, for fixed $T$ the integral approaches zero as $\inorm(\nnn) \to \infty$.
We make this quantitative as follows.

\begin{lemma} \label{truncatedlemma}
Let $\Omega_S \subset G (F_S)^1$ be a compact set. Then
\begin{equation} \label{truncatedintegral}
\int_{G(F) \bs G (\A)^1} F (x,T) \sum_{\gamma \in G (F), \, \gamma \ne 1}
\abs{(h_S \otimes \one_{\K^S (\nnn)}) (x^{-1} \gamma x)} \, dx \ll_{\Omega_S} \frac{\sup \abs{h_S}}{\inorm(\nnn)} (1 + \norm{T})^{d_0}.
\end{equation}
for all bounded measurable functions $h_S$ on $G (F_S)^1$ with support contained in $\Omega_S$
and all $T$ with $d(T) \ge d_{\Omega_S} = d_{\Omega_S \K^S}$.
\end{lemma}

The proof of this estimate is based on an elementary estimate for a lattice-point counting problem which we will prove below.
We first recall the following standard result from algebraic number theory (cf.~\cite[p. 102, Theorem 0]{MR1282723}
for a more precise result).

\begin{lemma} \label{numfieldlemma}
Let $\Lambda$ be a fractional ideal of $F$ and $D \subset F_\infty$ a compact set. Then
for any $a>0$ and a non-zero (integral) ideal $\nnn$ of $\mathfrak{o}_F$ we have
\[
\card{ a D \cap (\nnn \Lambda - \{ 0 \})} \ll_{D,\Lambda} \frac{a^{[F:\Q]}}{\inorm(\nnn)}.
\]
\end{lemma}

\begin{lemma} \label{countinglemma}
Let $P=M\ltimes U_P$ be a standard parabolic subgroup of $G$, $\uuu_P$ the Lie algebra of $U_P$,
$\Lambda \subset \uuu_P (F)$ an $\mathfrak{o}_F$-lattice and
$D \subset \uuu_P (F_\infty)$ a compact set. Then
for all $a \in A_P (T_1)$ and non-zero integral ideals $\nnn$ we have
\[
\card{\Ad (a) D \cap (\nnn \Lambda - \{ 0 \})} \ll_{D,\Lambda,T_1} \frac{\modulus_P (a)}{\inorm(\nnn)} .
\]
\end{lemma}

\begin{proof}
Let $e_1, \dots, e_n$ be an $F$-basis of $\uuu_P (F)$ consisting of eigenvectors with respect to $\spltrs_M$ and let $\alpha_1, \dots, \alpha_n$
be the associated eigencharacters (i.e., the roots of $\spltrs_M$ on $U_P$). By passing to a larger $\Lambda$ and $D$, if necessary, we can assume that
$\Lambda = \sum_i \Lambda_i e_i$ with fractional ideals $\Lambda_1,\dots,\Lambda_n \subset F$ and
$D = \sum_i D_i e_i$ with compact sets $D_1,\dots,D_n \subset F_\infty$.
Since a non-zero vector has at least one non-zero coordinate, we can estimate
\[
\card{\Ad (a) D \cap (\nnn \Lambda - \{ 0 \})} \le \sum_{i=1}^n \card{\alpha_i (a) D_i \cap (\nnn \Lambda_i - \{ 0 \})}
\prod_{j \neq i} \card{ \alpha_j (a) D_j \cap \nnn \Lambda_j }.
\]
We now use the estimate of Lemma \ref{numfieldlemma} for $\card{ \alpha_i (a) D_i \cap (\nnn \Lambda_i - \{ 0 \}) }$, while for the other coordinates
we use the trivial estimate $\card{ \alpha_j (a) D_j \cap \nnn \Lambda_j } \ll_{D_j,\Lambda_j} \alpha_j (a)^{[F:\Q]} + 1$.
This gives the desired result, since $\modulus_P(a)=\prod_{j=1}^n\alpha_j(a)^{[F:\Q]}$ and
the values $\alpha_j(a)$ are bounded away from $0$ (in terms of $T_1$).
\end{proof}

\begin{proof}[Proof of Lemma \ref{truncatedlemma}]
By Arthur's discussion in \cite[\S 5]{MR828844}, we can bound the left-hand side of \eqref{truncatedintegral} by the product of
$(1 + \norm{T})^{d_0}$ and
\begin{equation} \label{eq: Arthur1step}
\sup_{a \in A_{P_0} (T_1)} \modulus_{P_0} (a)^{-1} \sum_{\gamma \in \mathcal{U} (F), \, \gamma \ne 1} \phi (a^{-1} \gamma a),
\end{equation}
where
\[
\phi (x) = \sup_{y\in\Gamma} \abs{(h_S \otimes \one_{\K^S (\nnn)})(y^{-1}xy)}
\]
for a compact set $\Gamma \subset G (\A)^1$ depending only on $G$, $P_0$ and $\K$.
Of course, we can assume that $\Gamma = \prod_v \Gamma_v$
with $\Gamma_v = \K_v$ for all $v \notin S'$, where $S' \supset S_\infty$ is a finite set of places of $F$.
In a second step, Arthur bounds \eqref{eq: Arthur1step} by
\[
\sum_{P \supset P_0} \sum_{\mu \in M_P (F)} \sup_{a_1 \in A_{P} (T_1)}
\modulus_P (a_1)^{-1} \sum_{\nu \in U_P (F): \, \mu\nu \neq 1} \phi_\mu (a_1^{-1} \nu a_1),
\]
where $M_P$ is the Levi part of $P$ containing $M_0$ and
\[
\phi_\mu (u) = \sup_{b \in B} \modulus_{P_0} (b)^{-1} \phi (b^{-1} \mu u b), \quad u \in U_P (\A),
\]
for a fixed compact set $B \subset A_0$.

Here, for a given $P$ we need to sum only over all $\mu$ belonging to the intersection of $M_P(F)$
with a compact set that
depends only on $\Omega_S$, or equivalently over a finite subset of $M_P (F)$ that depends only on $\Omega_S$.
Considering each possibility for
$\mu$ separately, we see that
for all but at most finitely many $\nnn$ (depending on $\Omega_S$) only $\mu=1$ will contribute.
Furthermore, from the definition of $\phi_1$ we can estimate
\[
\phi_1 (u) \ll_{\Omega_S} \sup \abs{h_S} \one_{\Omega'_S \Ad (\Gamma^S) (\K^S (\nnn)) \cap U_P (\A)} (u),
\quad u \in U_P (\A),
\]
for a compact set $\Omega'_S \subset G (F_S)^1$ which depends only on $\Omega_S$.
Let $\primi_v$ be the prime ideal of the ring of integers of $F_v$.
There exist exponents $e_v \ge 0$ for $v \notin S$, with $e_v = 0$ for $v \notin S'$, such that
\[
\Ad (\Gamma_v) (\K_v (\primi_v^{f})) \subset \K_v (\primi_v^{f-e_v})
\]
for $f \ge e_v$.
Write $\nnn = \prod_{v \notin S} \primi_v^{f_v}$.
We conclude that $\Ad (\Gamma^S) (\K^S (\nnn)) \subset \prod_{v \notin S} L_{v,f_v}$, where
for $v \notin S$ and $f \ge 0$ we set
$L_{v,f} = \Ad (\Gamma_v) (\K_v)$ in case $f < e_v$ (which implies $v \in S'$) and
$L_{v,f} = \K_v (\primi_v^{f-e_v})$, otherwise.
Identify the unipotent radical $U_{P}$ with its Lie algebra $\uuu_P$ via the
exponential map. Then the lemma reduces to an application of Lemma \ref{countinglemma} (with $\nnn$ replaced by
$\nnn' = \prod_{v: \, f_v \ge e_v} \primi_v^{f_v-e_v}$).
\end{proof}

To finish the argument we follow Arthur's interpolation argument in \cite[pp. 1252-1254]{MR828844}.

\begin{proof}[Proof of Proposition \ref{unipproposition}]
Let $\Omega_S \subset G (F_S)^1$ be a compact set.
Given a bi-$\K_{S - S_\infty} (\nnn_S)$-invariant function $h_S \in C^\infty_{\Omega_S}(G (F_S)^1)$,
we combine Theorem \ref{TheoremArthur} and Lemma \ref{truncatedlemma} to obtain
\[
\abs{J_{\unip - \{ 1 \}}^T (h_S \otimes \one_{\K^S (\nnn)})} \le c_{\Omega_S}\norm{h_S}_{k}
\left( \inorm (\nnn_S \nnn)^m e^{-\epsilon d(T)}+ \inorm (\nnn)^{-1}\right) (1 + \norm{T})^{d_0}
\]
for all $T \in \aaa_0$ with $\norm{T} \ge d_{\Omega_S}$.
This implies
\[
\abs{J_{\unip - \{ 1 \}}^T (h_S \otimes \one_{\K^S (\nnn)})} \le 2 c_{\Omega_S} \norm{h_S}_{k} \inorm (\nnn)^{-1} (1 + \norm{T})^{d_0}
\]
for all $T \in \aaa_0$ with $d(T) \ge \max (d_{\Omega_S}, \frac{m+1}{\epsilon}\log \inorm (\nnn_S \nnn))$. Applying \cite[Lemma 5.2]{MR681737} to
the polynomial $J_{\unip - \{ 1 \}}^T (h_S \otimes \one_{\K^S (\nnn)})$ and the point $T_0 \in \aaa_0$, we obtain the assertion.
\end{proof}

We note that an alternative proof of Corollary \ref{corgeometriclimit} might be given by replacing $\mathcal{U} (F)$ by $G(F)$ in the arguments above
and using \cite[p. 267, Theorem 1]{MR546601}. We will give a detailed account of this proof in a future paper, which will treat a somewhat more general situation.

\subsection{A refinement}
We conclude this section with a refinement of Proposition \ref{unipproposition} that is close to optimal in its dependence on $\inorm (\nnn)$. We will use this refinement to give a certain
quantitative refinement of the limit multiplicity property in Theorem \ref{theoremquantitative}.
We may assume that $G$ is isotropic for otherwise $J_{\unip - \{ 1 \}} (h) = 0$.

Define
\begin{equation} \label{equationdmin}
\dmin = \min_{\tilde\alpha} \sprod{\rho}{\tilde\alpha^\vee},
\end{equation}
where $\rho$ is as usual half the sum of the positive roots of $G$ with respect to $P_0$ (counting multiplicities) and $\tilde\alpha$ ranges over the highest roots of
the irreducible components of the root system $\Phi (\spltrs_0,G)$.
We note that the minimal dimension of a non-trivial geometric unipotent orbit of $G$ containing an element of $G(F)$ is $2\dmin$.
For split groups this follows from \cite[Lemma 4.3.5]{MR1251060} and \cite{MR1610801}. In general,
note that a non-zero element of the $\tilde\alpha$-eigenspace in $\mathfrak{g}$ forms together with the co-root $\tilde\alpha^\vee$ and a suitable element of
the eigenspace of $-\tilde\alpha$ an $\mathfrak{sl} (2)$-triplet. We may then apply the dimension formula \cite[Lemma 4.1.3]{MR1251060} to compute the
dimension of the associated nilpotent orbit as $2 \sprod{\rho}{\tilde\alpha^\vee}$. The minimality of this orbit follows from the argument of [ibid., Theorem 4.3.3].


\begin{proposition} \label{unippropositionrefined}
There exists an integer $k\ge0$ such that for any compact subset $\Omega_S\subset G(F_S)^1$
and any integral ideals $\nnn_S$ and $\nnn$ of $\mathfrak{o}_F$, where $\nnn_S$ is a product of prime ideals of places in $S$ and $\nnn$ is prime to $S$, we have
\[
\abs{J_{\unip - \{ 1 \}} (h_S \otimes \one_{\K^S (\nnn)})} \ll_{\Omega_S}
\frac{(1 + \log \inorm (\nnn_S \nnn))^{d_0}}{\inorm (\nnn)^\dmin} \norm{h_S}_k
\]
for any bi-$\K_{S-S_\infty} (\nnn_S)$-invariant function $h_S \in C^\infty_{\Omega_S} (G (F_S)^1)$.
\end{proposition}

We remark that by Arthur's fine expansion for the unipotent contribution \cite{MR828844}, the exponent $\dmin$ in
the estimate of Proposition \ref{unippropositionrefined} is optimal.
For instance, consider the case where $\nnn = \mathfrak{p}^e$ for a prime ideal $\mathfrak{p}$ at a place $v$ of $F$ not contained in $S$.
Then [ibid.] expresses $J_{\unip - \{ 1 \}} (h_S \otimes \one_{\K^S (\nnn)})$ in terms of certain weighted orbital integrals
(including the invariant orbital integral)
of $h_S \otimes \one_{\K_v (\mathfrak{p}_v^e)}$ over the non-trivial unipotent orbits in $G(F_{S \cup \{ v\}})$ containing elements of $G(F)$.
Now it is easy to see that the invariant orbital integral of $\one_{\K_v (\mathfrak{p}_v^{2e})}$ over a unipotent orbit
$\mathfrak{o}$ in $G(F_v)$ is equal to a constant multiple of $q_v^{-de}$ for $e$ large 
where $d$ is the dimension of the geometric orbit associated to $\mathfrak{o}$. Indeed, embed an element
of the nilpotent orbit $\log \mathfrak{o} \subset \mathfrak{g} \otimes F_v$ into an $\mathfrak{sl} (2)$-triplet and
let $\mathfrak{g}_i \subset \mathfrak{g} \otimes F_v$, $i \in \Z$, be the associated eigenspaces.
Then the explicit formula for the invariant orbital integral $j_{\log \mathfrak{o}}$ of the nilpotent orbit $\log \mathfrak{o}$ \cite[Theorem 1]{MR0320232} implies the
homogeneity relation $j_{\log \mathfrak{o}} (f (\lambda^2 \cdot)) = \abs{\lambda}^{-r} j_{\log \mathfrak{o}} (f)$
for all $f\in C_c^\infty(\mathfrak{g}\otimes F_v)$ and $\lambda \in F_v^\times$, where
$r = \dim \mathfrak{g}_1 + 2\sum_{i \ge 2} \dim \mathfrak{g}_i = d$ (using \cite[Lemma 4.1.3]{MR1251060}
for the last equality). This immediately implies the assertion.

A glance at the proof of Proposition \ref{unipproposition} given above shows that
Proposition \ref{unippropositionrefined} will follow from the following improvement of Lemma \ref{countinglemma}.

\begin{lemma} \label{countinglemmarefined}
Let $P=M\ltimes U_P$ be a standard parabolic subgroup of $G$, $\uuu_P$ the Lie algebra of $U_P$,
$\Lambda \subset \uuu_P (F)$ an $\mathfrak{o}_F$-lattice and
$D \subset \uuu_P (F_\infty)$ a compact set. Then
for all $a \in A_P (T_1)$ and non-zero integral ideals $\nnn$ we have
\[
\card{\Ad (a) D \cap (\nnn \Lambda - \{ 0 \})} \ll_{D,\Lambda,T_1} \frac{\modulus_P (a)}{\inorm(\nnn)^\dmin}.
\]
\end{lemma}

For the proof of this lemma we first need a generalization of Lemma \ref{numfieldlemma} to vector spaces.

\begin{lemma} \label{numfieldlemmavectorspace}
Let $V$ be a finite-dimensional $F$-vector space, $\Lambda$ an $\mathfrak{o}_F$-lattice in $V$ and $D \subset V \otimes F_\infty$ a compact set.
Then for all $a > 0$ and non-zero integral ideals $\nnn$ we have
\[
\card{ a D \cap (\nnn \Lambda - \{ 0 \})} \ll_{D,\Lambda} \big(\frac{a^{[F:\Q]}}{\inorm(\nnn)}\big)^{\dim_F V}.
\]
\end{lemma}

\begin{proof} Choose an $F$-basis of $V$ and use Lemma \ref{numfieldlemma} for the coordinates, taking into account that
$a D \cap (\nnn \Lambda - \{ 0 \}) \neq \emptyset$ implies that $a^{[F:\Q]} \gg \inorm (\nnn)$.
\end{proof}

The core of the argument is contained in the following lemma on root systems.

\begin{lemma} \label{rootlemma}
Let $\Phi$ be a (possibly non-reduced) root system and $\Phi^+$ a set of positive roots for $\Phi$. Furthermore, let $m_\alpha \ge 0$ for $\alpha \in \Phi$ be
given and assume that $m_\alpha$ is invariant under the action of the Weyl group $W_{\Phi}$ on $\Phi$. Set $\rho = \frac{1}{2} \sum_{\alpha \in \Phi^+} m_\alpha \alpha$.
Then for all $\beta \in \Phi^+$ and all subsets $S \subset \Phi^+$ we have
\[
\sum_{\alpha \notin S} m_\alpha \alpha + \sum_{\alpha \in S} m_\alpha \beta \in \dmin\beta + \sum_{\alpha \in \Phi^+} \R^{\ge 0} \alpha,
\]
where $\dmin = \min_{\tilde\alpha} \sprod{\rho}{\tilde\alpha^\vee}$, $\tilde\alpha$ ranging over the highest roots of the irreducible components of $\Phi$.
\end{lemma}

\begin{proof}
It is clearly enough to consider the case where $\Phi$ is irreducible.
Let $C = \sum_{\alpha \in \Phi^+} \R^{\ge 0} \alpha$ be the closed cone spanned by the positive roots. Let $\tilde\alpha$ be the highest root of $\Phi^+$.
Note that $\sprod{\alpha}{\tilde\alpha^\vee}$ for $\alpha \in \Phi^+$ takes the values $0$, $1$ and $2$, and the last only
for $\alpha = \tilde\alpha$. Let $R$ be the set of all $\alpha \in \Phi^+$ with $\sprod{\alpha}{\tilde\alpha^\vee} = 1$.

Inserting the definition of $\rho$ and multiplying the inequality by two, we have to show that
\[
\sum_{\alpha \notin S} 2 m_\alpha \alpha + \sum_{\alpha \in S} 2 m_\alpha \beta \in \left( 2 m_{\tilde\alpha} + \sum_{\alpha \in R} m_\alpha \right) \beta + C.
\]
It is evidently sufficient to show the modified statement where
on the left hand side we restrict the sums to roots $\alpha \in R \cup \{ \tilde\alpha \}$. The contribution from $\alpha = \tilde\alpha$ is either $2 m_{\tilde\alpha} \tilde\alpha$ or
$2 m_{\tilde\alpha} \beta$, and therefore lies in $2 m_{\tilde\alpha} \beta + C$ in both cases.
It therefore remains to show that
\[
\sum_{\alpha \in R - S} 2 m_\alpha \alpha + \sum_{\alpha \in R \cap S} 2 m_\alpha \beta \in \sum_{\alpha \in R} m_\alpha \beta + C.
\]
Note that $\alpha \mapsto - w_{\tilde\alpha} (\alpha) = \tilde\alpha - \alpha$ defines an involution of $R$. We may therefore rewrite the last statement as
\[
\sum_{\alpha \in R - S} m_\alpha \alpha + \sum_{\alpha \in R \cap S} m_\alpha \beta
+ \sum_{\alpha \in R : \tilde\alpha - \alpha \notin S} m_\alpha (\tilde\alpha - \alpha)
+ \sum_{\alpha \in R : \tilde\alpha - \alpha \in S} m_\alpha \beta
\in \sum_{\alpha \in R} m_\alpha \beta + C.
\]
The only case where $\alpha \in R$ does not contribute a summand $m_\alpha \beta$ to the left hand side is when $\alpha \notin S$ and $\tilde\alpha - \alpha \notin S$, in which case
the total contribution is $m_\alpha \alpha + m_\alpha (\tilde\alpha - \alpha) = m_\alpha \tilde\alpha \in m_\alpha \beta + C$.
This finishes the proof of the lemma.
\end{proof}

\begin{proof}[Proof of Lemma \ref{countinglemmarefined}]
First note that we may assume that $\sprod{\alpha}{T_1} \le 0$ for all $\alpha \in \Delta_0$. By extending the lattice $\Lambda$ to a lattice
in $\uuu_{P_0} (F)$, while keeping $D$ and $a$ fixed,
we may reduce to the case where $P$ is the minimal parabolic subgroup $P_0$.
Let $\Phi = \Phi (\spltrs_0, G)$ and $\Phi^+$ the set of positive roots corresponding to $P_0$.
For $\alpha \in \Phi$ let $\uuu_\alpha$ be the $\alpha$-eigenspace in $\mathfrak{g}$ and $m_\alpha = \dim_F \uuu_\alpha$.
For $X \in \uuu_{P_0}$ let $X_\alpha$ be its projection to $\uuu_\alpha$ and write $S (X) = \{ \alpha \in \Phi^+ : X_\alpha \neq 0 \}$.
Evidently, it is enough to show that for each non-empty subset $S \subset \Phi^+$ we can bound
\begin{equation} \label{eqncard}
\frac{\inorm(\nnn)^\dmin}{\modulus_{P_0} (a)}\card{ X \in \Ad (a) D \cap \nnn \Lambda \, : \, S (X) = S}
\end{equation}
as $a$ ranges over $A_{P_0} (T_1)$ and $\nnn$ over the non-zero ideals of $\mathfrak{o}_F$.
For this, we may apply Lemma \ref{numfieldlemmavectorspace} to obtain for \eqref{eqncard} the bound
\[
\mbox{} \ll \frac{\inorm(\nnn)^\dmin}{\modulus_{P_0} (a)} \prod_{\alpha \in S} \left( \frac{\alpha (a)^{[F:\Q]}}{\inorm (\nnn)} \right)^{m_\alpha}
= \big(\prod_{\alpha \notin S} \alpha (a)^{-[F:\Q] m_\alpha}\big) {\inorm (\nnn)}^{\dmin - \sum_{\alpha \in S} m_\alpha}.
\]
This is clearly bounded in case $\sum_{\alpha \in S} m_\alpha \ge \dmin$, and we may therefore assume that
$\sum_{\alpha \in S} m_\alpha \le \dmin$.
Note that the existence of a vector $X \in \Ad (a) D \cap \nnn \Lambda$ with $S(X) = S$ implies that $\beta (a)^{[F:\Q]} \gg \inorm (\nnn)$ for arbitrary $\beta \in S$.
Therefore, \eqref{eqncard} is
\[
\mbox{} \ll \big(\prod_{\alpha \notin S} \alpha (a)^{-[F:\Q] m_\alpha}\big) \beta (a)^{[F:\Q] (\dmin - \sum_{\alpha \in S} m_\alpha)}.
\]
Applying Lemma \ref{rootlemma} we can write this as a product $\prod_{\alpha \in \Phi^+} \alpha (a)^{-\lambda_\alpha}$ with $\lambda_\alpha \ge 0$, and
it is therefore bounded for $a \in A_{P_0} (T_1)$.
\end{proof}

\section{Review of the spectral side of the trace formula} \label{SectionTraceFormula}

We turn to the spectral side of Arthur's trace formula and recall the results of \cite{MR2811597},
which are based on \cite{MR681737, MR681738}.

\subsection{Notation} \label{subsecnotation}

Let $\theta$ be the Cartan involution of $G (F_\infty)$ defining $\K_\infty$. It induces
a Cartan decomposition $\mathfrak{g} = \Lie G (F_\infty) = \mathfrak{p} \oplus \mathfrak{k}$ with $\mathfrak{k} = \Lie \K_\infty$.
We fix an invariant bilinear form $B$ on $\mathfrak{g}$ which is positive definite on $\mathfrak{p}$ and negative definite on $\mathfrak{k}$.
This choice defines a Casimir operator $\Omega$ on $G(F_\infty)$,
and we denote the Casimir eigenvalue of any $\pi \in \Pi (G (F_\infty))$ by $\lambda_\pi$.
Similarly, we obtain
a Casimir operator $\Omega_{\K_\infty}$ on $\K_\infty$ and write $\lambda_\tau$ for the Casimir eigenvalue of a
representation $\tau \in \Pi (\K_\infty)$ (cf.~\cite[\S 2.3]{MR701563}).
The form $B$ induces a Euclidean scalar product $(X,Y) = - B (X,\theta(Y))$ on $\mathfrak{g}$ and all its subspaces.

We write $\levis$ for the (finite) set of Levi subgroups containing $M_0$, i.e., the set of centralizers of subtori of $\spltrs_0$.
Let $W_0=N_{G(F)}(\spltrs_0)/M_0=N_{G(F)}(M_0)/M_0$ be the Weyl group of $(G,\spltrs_0)$,
where $N_{G(F)}(H)$ is the normalizer of $H$ in $G(F)$.
For any $s\in W_0$ we choose a representative $w_s\in N_{G(F)}(\spltrs_0)=N_{G(F)}(M_0)$.
Note that $W_0$ acts on $\levis$ by $sM=w_sMw_s^{-1}$.

Let now $M\in\levis$.
Let $W(M)=N_{G(F)}(M)/M$, which can be identified with a subgroup of $W_0$.
Denote by $\aaa_M^*$ the $\R$-vector space spanned by the lattice $X^*(M)$ of $F$-rational characters of $M$
and let $\aaa_{M,\C}^*=\aaa_M^*\otimes_{\R}\C$ be its complexification.
We write $\aaa_M$ for the dual space of $\aaa_M^*$, which is spanned by the co-characters of $\spltrs_M$.
It can also be identified with the Lie algebra of the torus $\spltrs_M$.
Let $\Ht_M:M(\A)\rightarrow\aaa_M$ be the homomorphism given by
\[
e^{\sprod{\chi}{\Ht_M(m)}}=\abs{\chi (m)}_{\A^\times} = \prod_v\abs{\chi(m_v)}_v
\]
for any $\chi\in X^*(M)$ and
denote by $M(\A)^1 \subset M (\A)$ the kernel of $\Ht_M$.
Let $\levis(M)$ be the set of Levi subgroups containing $M$ and
$\PPP(M)$ the set of parabolic subgroups of $G$ with Levi part $M$.
We also write $\FFF(M)=\FFF^G(M)=\coprod_{L\in\levis(M)}\PPP(L)$ for the (finite) set of parabolic subgroups of $G$ containing $M$.
Note that $W(M)$ acts on $\PPP(M)$ and $\FFF(M)$ by $sP=w_sPw_s^{-1}$.
Denote by $\rts_M$ the set of reduced roots of $\spltrs_M$ on the Lie algebra of $G$.
For any $\alpha\in\rts_M$ we denote by $\alpha^\vee\in\aaa_M$ the corresponding co-root.
Let $L^2_{\disc}(\AAA_MM(F)\bs M(\A))$ be the discrete part of $L^2(\AAA_MM(F)\bs M(\A))$, i.e., the
closure of the sum of all irreducible subrepresentations of the regular representation of $M(\A)$.
We denote by $\Pi_{\disc}(M(\A))$ the countable set of equivalence classes of irreducible unitary
representations of $M(\A)$ which occur in the decomposition of $L^2_{\disc}(\AAA_MM(F)\bs M(\A))$
into irreducibles.

For any $L\in\levis(M)$ we identify $\aaa_L^*$ with a subspace of $\aaa_M^*$.
We denote by $\aaa_M^L$ the annihilator of $\aaa_L^*$ in $\aaa_M$.
We set
\[
\levis_1(M)=\{L\in\levis(M):\dim\aaa_M^L=1\}
\]
(i.e., the set of Levi subgroups containing $M$ as a maximal Levi subgroup) and
\[
\FFF_1(M)=\bigcup_{L\in\levis_1(M)}\PPP(L).
\]
Note that the restriction of the scalar product $(\cdot,\cdot)$ on $\mathfrak{g}$ defined above endows
$\aaa_0=\aaa_{M_0}\subset\mathfrak{g}$ with the structure of a Euclidean space.
In particular, this fixes Haar measures on the spaces $\aaa^L_M$ and their duals
$(\aaa^L_M)^*$. We follow Arthur in the corresponding normalization of Haar measures on the groups $M(\A)$
\cite[\S 1]{MR518111}.

\subsection{Intertwining operators}

Now let $P\in\PPP(M)$. We write $\aaa_P=\aaa_M$.
Let $U_P$ be the unipotent radical of $P$.
Denote by $\rts_P\subset\aaa_P^*$ the set of reduced roots of $\spltrs_M$ on the Lie algebra $\mathfrak{u}_P$ of $U_P$.
Let $\srts_P$ be the subset of simple roots of $P$, which is a basis for $(\aaa_P^G)^*$.
Write $\aaa_{P,+}^*$ for the closure of the Weyl chamber of $P$, i.e.
\[
\aaa_{P,+}^*=\{\lambda\in\aaa_M^*:\sprod{\lambda}{\alpha^\vee}\ge0\text{ for all }\alpha\in\rts_P\}
=\{\lambda\in\aaa_M^*:\sprod{\lambda}{\alpha^\vee}\ge0\text{ for all }\alpha\in\srts_P\}.
\]
Denote by $\modulus_P$ the modulus function of $P(\A)$.
Let $\bar\AF^2(P)$ be the Hilbert space completion of
\[
\{\phi\in C^\infty(M(F)U_P(\A)\bs G(\A)):\modulus_P^{-\frac12}\phi(\cdot x)\in
L^2_{\disc}(A_MM(F)\bs M(\A))\ \forall x\in G(\A)\}
\]
with respect to the inner product
\[
(\phi_1,\phi_2)=\int_{\AAA_MM(F)U_P(\A)\bs G(\A)}\phi_1(g)\overline{\phi_2(g)}\ dg.
\]

Let $\alpha\in\rts_M$.
We say that two parabolic subgroups $P,Q\in\PPP(M)$ are \emph{adjacent} along $\alpha$,
and write $P|^\alpha Q$, if $\rts_P\cap-\rts_Q=\{\alpha\}$.
Alternatively, $P$ and $Q$ are adjacent if the group $\gen{P,Q}$ generated by $P$ and $Q$ belongs to $\FFF_1(M)$.
Any $R\in\FFF_1(M)$ is of the form $\gen{P,Q}$ where $P,Q$ are the elements of $\PPP(M)$ contained in $R$;
we have $P|^\alpha Q$ with $\alpha^\vee\in\rts_P^\vee \cap\aaa^R_M$.
Interchanging $P$ and $Q$ switches $\alpha$ to $-\alpha$.

For any $P\in\PPP(M)$ let $\Ht_P:G(\A)\rightarrow\aaa_P$ be the extension of $\Ht_M$ to a left $U_P(\A)$-
and right $\K$-invariant map.
Denote by $\AF^2(P)$ the dense subspace of $\bar\AF^2(P)$ consisting of its $\K$- and $\zzz$-finite vectors,
where $\zzz$ is the center of the universal enveloping algebra of $\mathfrak{g} \otimes \C$.
That is, $\AF^2(P)$ is the space of automorphic forms $\phi$ on $U_P(\A)M(F)\bs G(\A)$ such that
$\modulus_P^{-\frac12}\phi(\cdot k)$ is a square-integrable automorphic form on
$\AAA_MM(F)\bs M(\A)$ for all $k\in\K$.
Let $\rho(P,\lambda)$, $\lambda\in\aaa_{M,\C}^*$, be the induced
representation of $G(\A)$ on $\bar\AF^2(P)$ given by
\[
(\rho(P,\lambda,y)\phi)(x)=\phi(xy)e^{\sprod{\lambda}{\Ht_P(xy)-\Ht_P(x)}}.
\]
It is isomorphic to $\Ind_{P(\A)}^{G(\A)}\left(L^2_{\disc}(\AAA_M M(F)\bs M(\A))\otimes e^{\sprod{\lambda}{\Ht_M(\cdot)}}\right)$.

For $P,Q\in\PPP(M)$ let
\[
M_{Q|P}(\lambda):\AF^2(P)\to\AF^2(Q),\quad\lambda\in\aaa_{M,\C}^*,
\]
be the standard \emph{intertwining operator} \cite[\S 1]{MR681738}, which is the meromorphic
continuation in $\lambda$ of the integral
\[
[M_{Q|P}(\lambda)\phi](x)=\int_{U_Q(\A)\cap U_P(\A)\bs U_Q(\A)}\phi(nx)
e^{\sprod{\lambda}{\Ht_P(nx)-\Ht_Q(x)}}\ dn, \quad \phi\in\AF^2(P), \ x\in G(\A).
\]
These operators satisfy the following properties.
\begin{enumerate}
\item $M_{P|P}(\lambda)\equiv\Id$ for all $P\in\PPP(M)$ and $\lambda\in\aaa_{M,\C}^*$.
\item For any $P,Q,R\in\PPP(M)$ we have
$M_{R|P}(\lambda)=M_{R|Q}(\lambda)\circ M_{Q|P}(\lambda)$ for all $\lambda\in\aaa_{M,\C}^*$.
In particular, $M_{Q|P}(\lambda)^{-1}=M_{P|Q}(\lambda)$.
\item $M_{Q|P}(\lambda)^*=M_{P|Q}(-\overline{\lambda})$ for any $P,Q\in\PPP(M)$
and $\lambda\in\aaa_{M,\C}^*$. In particular, $M_{Q|P}(\lambda)$
is unitary for $\lambda\in\iii\aaa_M^*$.
\item If $P|^\alpha Q$ then $M_{Q|P}(\lambda)$ depends only on $\sprod{\lambda}{\alpha^\vee}$.
\end{enumerate}

Let $\pi \in \Pi_{\disc} (M (\A))$ and let
$\AF^2_\pi(P)$ be the space of all $\phi \in \AF^2 (P)$ for which the functions $\modulus_P^{-\frac12}\phi(\cdot g)$, $g \in G (\A)$,
belong to the $\pi$-isotypic subspace of $L^2 (A_M M (F) \bs M (\A))$.
For any $P\in\PPP(M)$ we have a canonical isomorphism of $G(\A_f)\times(\LieG_{\C},K_\infty)$-modules
\[
j_P:\Hom(\pi,L^2(\AAA_MM(F)\bs M(\A)))\otimes\Ind_{P(\A)}^{G(\A)}(\pi)\rightarrow\AF^2_\pi(P).
\]
If we fix a unitary structure on $\pi$ and endow $\Hom(\pi,L^2(\AAA_MM(F)\bs M(\A)))$
with the inner product $(A,B)=B^*A$ (which is a scalar operator on the space of $\pi$), the isomorphism $j_P$ becomes an isometry.

Suppose that $P|^\alpha Q$.
The operator $M_{Q|P}(\pi,s) := M_{Q|P} (s \varpi)|_{\AF^2_\pi(P)}$,
where $\varpi \in \aaa_M^*$ is
such that $\sprod{\varpi}{\alpha^\vee} = 1$,
admits a normalization by a global factor
$n_\alpha(\pi,s)$ which is a meromorphic function in $s$. We may write
\begin{equation} \label{eq: normalization}
M_{Q|P}(\pi,s)\circ j_P=n_\alpha(\pi,s)\cdot j_Q\circ(\Id\otimes R_{Q|P}(\pi,s))
\end{equation}
where $R_{Q|P}(\pi,s)=\otimes_v R_{Q|P}(\pi_v,s)$ is the product
of the locally defined normalized intertwining operators and $\pi=\otimes_v\pi_v$
(\cite[\S 6]{MR681738}, cf.~\cite[(2.17)]{MR1935546}).

\subsection{The spectral side}
We now turn to the spectral expansion of Arthur's distribution $J$ on $C^\infty_c (G (\A)^1)$.
Let $L \supset M$ be Levi subgroups in $\levis$, $P \in \PPP (M)$, and
let $m=\dim\aaa_L^G$ be the co-rank of $L$ in $G$.
Denote by $\bases_{P,L}$ the set of $m$-tuples $\bss=(\beta_1^\vee,\dots,\beta_m^\vee)$
of elements of $\rts_P^\vee$ whose projections to $\aaa_L$ form a basis for $\aaa_L^G$.
For any $\bss=(\beta_1^\vee,\dots,\beta_m^\vee)\in\bases_{P,L}$ let
$\vol(\bss)$ be the co-volume in $\aaa_L^G$ of the lattice spanned by the projection of $\bss$ to $\aaa_L$ and let
\begin{align*}
\Xi_L(\bss)&=\{(Q_1,\dots,Q_m)\in\FFF_1(M)^m: \ \ \beta_i^{\vee}\in\aaa_M^{Q_i}, \, i = 1, \dots, m\}\\&=
\{(\gen{P_1,P_1'},\dots,\gen{P_m,P_m'}): \ \ P_i|^{\beta_i}P_i', \, i = 1, \dots, m\}.
\end{align*}

For any smooth function $f$ on $\aaa_M^*$ and $\mu\in\aaa_M^*$ denote by $D_\mu f$ the directional derivative of $f$ along $\mu\in\aaa_M^*$.
For a pair $P_1|^\alpha P_2$ of adjacent parabolic subgroups in $\PPP(M)$ write
\[
\delta_{P_1|P_2}(\lambda)=M_{P_2|P_1}(\lambda)D_\varpi M_{P_1|P_2}(\lambda):\AF^2(P_2)\rightarrow\AF^2(P_2),
\]
where $\varpi\in\aaa_M^*$ is such that $\sprod{\varpi}{\alpha^\vee}=1$.\footnote{Note that this definition differs slightly from the definition of
$\delta_{P_1|P_2}$ in \cite{MR2811597}.} Equivalently,
\[
\delta_{P_1|P_2}(\lambda)=\Phi(\sprod{\lambda}{\alpha^\vee})^{-1}\Phi'(\sprod{\lambda}{\alpha^\vee})
\]
where $\Phi$ is the meromorphic function on $\C$ such that
$M_{P_1|P_2}(\lambda)=\Phi(\sprod{\lambda}{\alpha^\vee})$.

For any $m$-tuple $\dtup=(\gen{P_1,P_1'},\dots,\gen{P_m,P_m'})\in\Xi_L(\bss)$
with $P_i|^{\beta_i}P_i'$, denote by $\Delta_{\dtup}(P,\lambda)$
the expression
\[
\frac{\vol(\bss)}{m!}M_{P_1'|P}(\lambda)^{-1}\delta_{P_1|P_1'}(\lambda)M_{P_1'|P_2'}(\lambda) \cdots
\delta_{P_{m-1}|P_{m-1}'}(\lambda)M_{P_{m-1}'|P_m'}(\lambda)\delta_{P_m|P_m'}(\lambda)M_{P_m'|P}(\lambda).
\]
In \cite[pp. 179-180]{MR2811597} we defined a (purely combinatorial) map $\dtup_L: \bases_{P,L} \to \FFF_1 (M)^m$ with the property that
$\dtup_L(\bss) \in \Xi_L (\bss)$ for all $\bss \in \bases_{P,L}$.\footnote{The map $\dtup_L$ depends in fact on the additional choice of
a vector $\underline{\mu} \in (\aaa^*_M)^m$ which lies outside a prescribed finite
set of hyperplanes. For our purposes, the precise definition of $\dtup_L$ is immaterial.}

For any $s\in W(M)$ let $L_s$ be the smallest Levi subgroup in $\levis(M)$
containing $w_s$. We recall that $\aaa_{L_s}=\{H\in\aaa_M\mid sH=H\}$.
Set
\[
\iota_s=\abs{\det(s-1)_{\aaa^{L_s}_M}}^{-1},
\]
a constant which we will not worry much about.
For $P\in\PPP(M)$ and $s\in W(M)$ let
$s: \AF^2 (P) \to \AF^2 (sP)$ be left translation by $w_s^{-1}$ and
$M(P,s) = M_{P|sP} (0) s:\AF^2(P)\to\AF^2(P)$ as in \cite[p.~1309]{MR681738}.
$M(P,s)$ is a unitary operator which intertwines $\rho(P,\lambda)$ with itself for $\lambda\in\iii\aaa_{L_s}^*$.
Finally, we can state the refined spectral expansion.

\begin{theorem}[\cite{MR2811597}] \label{thm: specexpand}
The spectral side of Arthur's trace formula is given by
\[
J (h) = \sum_{[M]} J_{\spec,M} (h),\ \ \ h\in C_c^\infty(G(\A)^1),
\]
$M$ ranging over the conjugacy classes of Levi subgroups of $G$ (represented by members of $\levis$),
where
\[
J_{\spec,M} (h) =
\frac1{\card{W(M)}}\sum_{s\in W(M)}\iota_s
\sum_{\bss\in\bases_{P,L_s}}\int_{\iii(\aaa^G_{L_s})^*}
\tr(\Delta_{\dtup_{L_s}(\bss)}(P,\lambda)
M(P,s)\rho(P,\lambda,h))\ d\lambda
\]
with $P \in \PPP(M)$ arbitrary.
The operators are of trace class and the integrals are absolutely convergent.
\end{theorem}

Note that here the term corresponding to $M=G$ is simply $J_{\spec,G} (h) = \tr R_{\disc}(h)$.

\section{Bounds on intertwining operators} \label{SectionBounds}

We now introduce the key global and local properties required for the proof of the spectral limit property, and verify
them for the groups $\GL (n)$ and $\SL(n)$. These properties will be used in \S \ref{sec: spectralside} to provide estimates for the
contribution of the continuous spectrum to the spectral side of the trace formula.
In view of planned future applications, our estimates are somewhat more precise in their dependence on the relevant parameters than
it is strictly necessary for the purposes of this paper.

\subsection{The level of a compact open subgroup} \label{SubsectionLevel}
We begin by introducing the notion of level for open subgroups $K$ of $\K_{\fin}$ (or, more generally, open compact subgroups $K \subset G (\A_{\fin})$), as well as some generalizations.
Recall the definition of the principal congruence subgroups $\K (\nnn)$ from \S \ref{subsecnotationgeom}.
Let $\nnn_K$ be the largest ideal $\nnn$ of $\mathfrak{o}_F$ such that
$\K (\nnn) \subset K$, and define the \emph{level} of $K$ as $\level_G(K) = \level(K)=\inorm (\nnn_K)$.
Analogously, we define $\level (K_v)$ for open subgroups $K_v \subset \K_v$.
For a smooth representation $\pi$ of $G(\A)$ let $\nnn_\pi$ be the largest ideal $\nnn$ such that
$\pi^{\K(\nnn)}\ne0$. Let $\level_G(\pi)=\level(\pi)=\inorm(\nnn)$. Thus
\[
\level_G(\pi)=\level(\pi)=\min\{\level(K):\pi^K\ne0\},
\]
where $K$ ranges over the open subgroups of $\K_{\fin}$.

For any Levi subgroup $M \in \levis$ and an open subgroup $K \subset \K_{M,\fin} = \K_{\fin} \cap M (\A_{\fin})$
we define $\level_M (K)$ analogously using the principal congruence subgroups $\K (\nnn) \cap M (\A_{\fin}) = \K_M (\nnn)$ of $\K_{M,\fin}$.
Equivalently, $\level_M (K)$ is the level of $K \cap M (\A_{\fin})$ with respect to the restriction to $M$ of the faithful $G$-representation $\rho$ fixed in \S \ref{subsecnotationgeom}.

More generally, suppose that $M \in \levis$ and that $\mathcal{H}$ is a closed factorizable subgroup of $G(\A_{\fin})$.
For any open compact subgroup $K \subset M (\A_{\fin})$ let
$\nnn_{K;\mathcal{H}}$ be the largest ideal $\nnn$ of $\mathfrak{o}_F$ such that
$\K (\nnn)\cap M (\A_{\fin}) \cap \mathcal{H} \subset K$, and define the \emph{relative level} of $K$ with respect to $\mathcal{H}$
as $\level_M (K;\mathcal{H}) = \inorm (\nnn_{K;H})$.
Clearly, $\level_M (K;\mathcal{H}_1)\le\level_M (K;\mathcal{H}_2)$ if $\mathcal{H}_1\subset \mathcal{H}_2$
and $\level_M (K;G(\A_{\fin}))=\level_M(K)$.
In particular, $\level_M (K;\mathcal{H})\le\level_M(K)$ for all $\mathcal{H}$.

As before, for an irreducible admissible representation $\pi$ of $M(\A)$ we write
\[
\level_M(\pi;\mathcal{H})=\min\{\level(K;\mathcal{H}):\pi^K\ne0\}
\]
where $K$ ranges over the compact subgroups of $\K_{M, \fin}$. Equivalently,
$\level_M(\pi;\mathcal{H})=\inorm(\nnn)$ where $\nnn$ is the largest ideal such that 
$\pi^{\K(\nnn)\cap M (\A_{\fin}) \cap \mathcal{H}}\ne0$.
Because $\mathcal{H}$ is assumed to be factorizable, it is clear how to define $\level_M (K_v; \mathcal{H})$ and $\level_M (\pi_v;\mathcal{H})$, where
$v$ is a finite place of $F$, $K_v$ an open subgroup of $\K_{M,v}$ and $\pi_v$ a smooth representation of $M (F_v)$.

When $H$ is an algebraic subgroup of $G$ defined over $F$, we write simply $\level_M (K;H)$ for $\level_M (K;H (\A_{\fin}))$. However,
more general closed subgroups $\mathcal{H} \subset G (\A_{\fin})$ also appear naturally.
For any reductive group $H$ let
$H^{\SC}$ be the simply connected cover of the derived group
$H^{\der}$ of $H$ and $\scprj:H^{\SC}\rightarrow H$ the natural homomorphism.
We write $H(F_v)^+ \subset H^{\der} (F_v)$ for the image of $H^{\SC}(F_v)$ under $\scprj$, and similarly for $H (\A)^+ \subset H^{\der} (\A)$.
Note that $H (\A)^+$ contains the derived group of $H(\A)$ as an open subgroup, and that it is equal to
it if all the $F$-simple factors of $H^{\SC}$ are $F$-isotropic, in which case it is also the closed subgroup
of $H(\A)$ generated by its unipotent elements (cf.~\cite[Theorem 7.1, Proposition 7.6, Theorem 7.6]{MR1278263}
for the equality and [ibid., Proposition 3.5, Proposition 3.17] for the openness statement).
Also note that for any normal reductive subgroup $H$ of $G$, the group $H(\A)^+$ is a normal subgroup of $G(\A)$, since $G$ acts on $H^{\SC}$ by conjugation.
Finally, note that the indices $[H^{\der} (F_v):H (F_v)^+]$ are bounded in terms of $H$ only, as $v$ ranges over the places of $F$.
For simplicity we will often write by abuse of notation $\level_M (K;H^+)$ instead of $\level_M (K;H(\A_{\fin})^+)$.

Although in this paper our main examples are reductive groups $H$ whose derived group is simply connected,
which implies that $H^{\SC}=H^{\der}$ and $H(\A)^+=H^{\der}(\A)$,
in the general case it is advantageous to work with $H(\A)^+$ instead of $H^{\der}(\A)$.

\begin{lemma} \label{lem: twist}
Let $G$ be a reductive group defined over a non-archimedean local field $F$ and $M$ a Levi subgroup of $G$ defined over $F$.
Then for any smooth representation $\pi$ of $M (F)$ we have
\[
\level_M(\pi;G (F)^+)=\min_\chi \level_M(\pi\otimes\chi),
\]
where $\chi$ ranges over the characters of $G (F) /G (F)^+$
(viewed as characters of $M (F)$ by restriction).
\end{lemma}

\begin{proof}
It is clear that $\level_M(\pi\otimes\chi;G^+)=\level_M(\pi;G^+)$
for any character $\chi$ of $G(F)$ trivial on $G(F)^+$
and that $\level_M(\pi;G^+)\le\level_M(\pi)$. Thus
$\level_M(\pi;G^+)\le\min\level_M(\pi\otimes\chi)$.

In the other direction, suppose that $K \subset \K_{M,v}$ is an open subgroup such that $V:=\pi^{K\cap G^+}\ne0$.
Then the compact abelian group $A:=K/K\cap G (F)^+$ acts on $V$. Let $\chi$ be a character of $A$ such that $V^{A,\chi^{-1}}\ne0$.
We can extend $\chi$ to a character of $G (F)/G (F)^+$. Then $(\pi\otimes\chi)^K\ne0$.
\end{proof}

We also have the following archimedean analog of $\level_M (\pi;H^+)$.
To define it,
recall first Vogan's definition of a norm $\norm{\cdot}$ on $\Pi (K_\infty)$ (cf. \cite[\S 2.2]{MR759263}), where
$K_\infty$ is a compact Lie group satisfying the conditions of \cite[0.1.2]{MR632407}. Namely,
let $K^0_\infty$ be the connected component of the identity of $K_\infty$. Then
$\norm{\mu} = \norm{\chi_\mu + 2 \rho}^2$, where $\chi_\mu$ denotes the highest weight of
an arbitrary irreducible constituent of $\mu|_{K^0_\infty}$ with respect to a maximal torus of $K^0_\infty$ (and the choice
of a system of positive roots), and as usual $\rho$ is half of the sum of all positive roots with multiplicities.
A lowest $K_\infty$-type of a representation is then a $K_\infty$-type minimizing $\norm{\cdot}$.

Let $H \subset G$ be a reductive algebraic subgroup normalized by $M$.
For an irreducible representation $\pi$ of $M(F_\infty)$ we will write
\[
\param_M(\pi;H)=1+(\lambda^H_\pi)^2+\norm{\tau}^2,
\]
where $\lambda^H_\pi$ is the eigenvalue of the Casimir operator of $M (F_\infty) \cap H (F_\infty)$
(which is an element of the center of the universal enveloping algebra of $\mathfrak{g}$)
and $\tau$ is a lowest $\K_\infty\cap M (F_\infty) \cap H (F_\infty)$-type of $\pi$.

For $M \subset H$ we simply write $\param_M(\pi)$ for $\param_M (\pi;H)$.
In \cite{MR1935546}, the parameter
$\Lambda_\pi= \min_\tau \sqrt{\lambda_\pi^2+\lambda_\tau^2}$ was used,
where $\tau$ ranges over the lowest $\K_\infty$-types of the induced representation $\Ind^{G(F_\infty)}_{P(F_\infty)}(\pi)$
and $\lambda_\tau$ denotes the Casimir eigenvalue of $\tau$.
By \cite[(5.15)]{MR1616155} we have the estimate
\begin{equation} \label{EqnParamMUpper}
1 \le \param_M (\pi;H) \le \param_M (\pi) \ll_G 1 + \Lambda_{\pi}^2
\end{equation}
for any $H$.
Using the explicit description of lowest $\K_\infty$-types of irreducible and parabolically induced representations (\cite[Theorem 7.16]{MR519352}, \cite[6.5.9]{MR632407}),
one can also show that
\begin{equation} \label{EqnParamMLower}
\Lambda_{\pi}^2 \ll_G \param_M (\pi).
\end{equation}
We will not give any details here, since we will only use this estimate in a side remark (the first part of Remark \ref{rem: TWN} below).

\subsection{Bounds for global normalizing factors}
For $M \in\levis$, $\alpha \in \Sigma_M$ and $\pi\in\Pi_{\disc}(M(\A))$ let
$n_\alpha (\pi,s)$ be the global normalizing factor defined by \eqref{eq: normalization}.

Let $U_\alpha$ be the unipotent subgroup of $G$ corresponding to $\alpha$
(so that the eigenvalues of $\spltrs_M$ on the Lie algebra of $U_\alpha$ are positive integer multiples of $\alpha$).
Let $M_\alpha\in\levis_1(M)$ be the group generated by $M$ and $U_{\pm\alpha}$.
Let $\prpr{M_\alpha}$ be the group generated by $U_{\pm\alpha}$.
It is a connected normal subgroup of $M_\alpha$ defined over $F$ \cite[Proposition 4.11]{MR0207712}.
Moreover, since $M$ has co-rank one in $M_\alpha$, precisely
one simple root $\beta$ of $M_\alpha$ restricts to $\alpha$, which implies that the root system of $\prpr{M_\alpha}$ is the irreducible
component of the root system of $M_\alpha$ containing $\beta$.
The group $\prpr{M_\alpha}$ is therefore $F$-simple. It is also clearly $F$-isotropic, and therefore
$\prpr{M_\alpha} (\A)^+$ is the closed subgroup of $G (\A)$ generated by $U_{\pm\alpha} (\A)$ \cite[Proposition 6.2]{MR0316587}.
It is also the derived group of $\prpr{M_\alpha} (\A)$.

For any subset $\types \subset \Pi (\K_{M,\infty})$ we denote by $\Pi_{\disc}(M(\A))^{\types}$ the set of
all $\pi = \pi_\infty \otimes \pi_{\fin} \in\Pi_{\disc}(M(\A))$ for which $\pi_\infty$ contains a $\K_{M,\infty}$-type in $\types$.
If in addition an open subgroup $K_M \subset \K_{M,\fin}$ is given, we define
$\Pi_{\disc}(M(\A))^{\types, K_M}$ as the set of all $\pi \in\Pi_{\disc}(M(\A))^{\types}$ with $\pi_{\fin}^{K_M} \neq 0$.

\begin{definition} \label{DefinitionTWN}
We say that the group $G$ satisfies property \TWN\ (tempered winding numbers) if
for any $M\in\levis$, $M\ne G$, and any finite subset $\types \subset \Pi (\K_{M,\infty})$
there exists $k > 1$ such that for any $\alpha \in \Sigma_M$ and any $\epsilon>0$ we have
\[
\int_{\iii\R}\abs {\frac{n'_\alpha (\pi,s)}{n_\alpha(\pi,s)}}  (1+\abs{s})^{-k}\ ds
\ll_{\types,\epsilon} \param_M(\pi_\infty;\prpr M_\alpha)^k \level_M(\pi;\prpr M_\alpha^+)^{\epsilon}
\]
for any $\pi\in\Pi_{\disc}(M(\A))^{\types}$. (We recall that $\abs{n_\alpha(\pi,s)}=1$ for $s\in\iii\R$.)
\end{definition}

Note that $n_\alpha(\pi,s)$ is not changed if we replace $G$ by $M_\alpha$ or any other Levi subgroup containing it.
Therefore, property \TWN\ is hereditary for Levi subgroups.

\begin{remark} \label{rem: TWN}
\begin{enumerate}
\item If we fix an open compact subgroup $K_M$ then the corresponding bound
\[
\int_{\iii\R}\abs {\frac{n'_\alpha (\pi,s)}{n_\alpha(\pi,s)}} (1+\abs{s})^{-k}\ ds\ll_{K_M} \param_M(\pi_\infty;\prpr M_\alpha)^k
\]
for any $\pi\in\Pi_{\disc}(M(\A))^{\types,K_M}$ and a suitable $k > 1$ depending only on $G$ can be deduced (by invoking \eqref{EqnParamMLower}) from \cite[Theorem 5.3]{MR1935546}, applied to the groups $M \cap \prpr M_\alpha \subset \prpr M_\alpha$.
So, the point of \TWN\ lies in the dependence of the bound on the level of $\pi$.
\item In fact, we expect that
\begin{equation}
\label{eq: logbnd}
\int_{T}^{T+1}\abs{\frac{n'_\alpha(\pi,\iii t)}{n_\alpha(\pi,\iii t)}} \ dt\ll
\log(\abs{T}+\param_M(\pi_\infty;\prpr M_\alpha)+\level_M(\pi;\prpr M_\alpha^+))
\end{equation}
for all $T\in\R$ and $\pi\in\Pi_{\disc}(M(\A))$.

This would give the following strengthening  of \TWN:
\[
\int_{\iii\R}\abs{\frac{n'_\alpha (\pi,s)}{n_\alpha(\pi,s)}}(1+\abs{s})^{-k}\ ds\ll_k
\log(\param_M(\pi_\infty;\prpr M_\alpha)+\level_M(\pi;\prpr M_\alpha^+))
\]
for any $\pi\in\Pi_{\disc}(M(\A))$ and $k>1$.
\end{enumerate}
\end{remark}

\begin{lemma} \label{lem: trivred}
Suppose that $\tilde G$ is a connected reductive group over $F$ which satisfies \TWN.
Let $G$ be a connected subgroup of $\tilde G$ containing $\tilde G^{\der}$. Then $G$ also satisfies \TWN.
The analogous statement holds for the bound \eqref{eq: logbnd}.
\end{lemma}

\begin{proof}
The map $\tilde M\mapsto \tilde M\cap G$ defines a one-to-one correspondence between the Levi subgroups of $\tilde G$ and those of $G$
and we have $\tilde M\cap G\supset \tilde M \cap \tilde G^{\der} \supset \tilde M^{\der}$.
Suppose that $M=\tilde M\cap G$ and let $\tilde\pi\in\Pi_{\disc}(\tilde M(\A))$ be realized automorphically on a subspace
$V_{\tilde\pi} \subset L^2 (A_{\tilde M} {\tilde M} (F) \bs {\tilde M} (\A))$.
Then the space $V^M_{\tilde\pi} =\{\varphi\rest_{M(\A)}:\varphi\in V_{\tilde\pi}\}$ is non-trivial
and if $\pi$ is an irreducible constituent of the $M(\A)$-representation $V^M_{\tilde\pi}$, then for every place $v$, $\pi_v$ is a constituent of the
restriction of $\tilde\pi_v$ to $M(F_v)$.
By \cite[\S 3]{MR999488} for the archimedean and \cite[Theorem 3.3.4]{ChaoLi} for the non-archimedean case,
we have $n_\alpha(\pi_v,s)=n_\alpha(\tilde\pi_v,s)$ for all $v$. Hence $n_\alpha(\pi,s)=n_\alpha(\tilde\pi,s)$.
Moveover, $\prpr M_\alpha = \prpr{{\tilde M}_\alpha}$ and $\prpr M_\alpha \cap M = \prpr{{\tilde M}_\alpha} \cap \tilde M$,
which implies $\level_{\tilde M}(\tilde\pi;\prpr{{\tilde M}_\alpha}^+)\le\level_M(\pi;\prpr M_\alpha^+)$
and $\param_{\tilde M}(\tilde\pi_\infty;\prpr{{\tilde M}_\alpha}) \le \param_M(\pi_\infty;\prpr M_\alpha)$.
Therefore the lemma follows from the fact that every $\pi\in\Pi_{\disc}(M(\A))$
is equivalent to a constituent of $V^M_{\tilde\pi}$ for some $\tilde\pi\in\Pi_{\disc}(\tilde M(\A))$ with respect
to \emph{some} automorphic realization $V_{\tilde\pi}$ (\cite[Theorem 4.13 and Remark 4.23]{MR2918491} applied to $M\subset \tilde M$).
\end{proof}

\begin{proposition} \label{prop: mainglobal}
The estimate \eqref{eq: logbnd} holds for $G=\GL(n)$ or $\SL(n)$ with an implied constant depending only on $n$ and $F$.
In particular, the groups $\GL(n)$ and $\SL(n)$ satisfy property \TWN.
\end{proposition}

We first need the following two lemmas which are a direct consequence of the results of Bushnell--Henniart \cite{MR1462836}
and Jacquet--Piatetski-Shapiro--Shalika (cf.~\cite{MR3001803, 1201.5506}).\footnote{We define levels for $\GL(n)$
and its subgroups in terms of the identity representation $\rho$ of $\GL (n)$ and the standard lattice $\mathfrak{o}_F^n$.}

\begin{lemma} \label{LemmaConductor}
Let $F$ be a local non-archimedean field with residue field $\F_q$, let $\pi$ be an irreducible smooth representation of $\GL (n,F)$
and $f (\pi)$ be the exponent of the conductor of $\pi$ as defined in \cite{MR620708}. Then
\[
f (\pi) \le \log_q \level (\pi).
\]
\end{lemma}

\begin{proof}
If $\pi$ is generic, this immediately follows from the characterization of $f(\pi)$ in \cite[Th\'eor\`eme 5.1]{MR620708}.
In general, write $\pi$ as the Langlands quotient of the parabolic induction of
$\sigma_1 \abs{\det}^{t_1} \otimes \cdots \otimes \sigma_r \abs{\det}^{t_r}$ to $\GL (n, F)$,
where $n = n_1 + \cdots + n_r$, $t_1, \ldots, t_r \in \R$, and each $\sigma_i$ is an irreducible tempered representation of $\GL (n_i,F)$.
We then have $f (\pi) = f (\sigma_1) + \cdots + f(\sigma_r)$ \cite[Theorem 3.4]{MR546609}.
Let $\pi_0$ be the parabolic induction of $\sigma_1\otimes\dots\otimes\sigma_r$ to $\GL (n,F)$. Then $\pi_0$ is irreducible and generic and we have
$f(\pi)=f(\pi_0)\le\log_q\level(\pi_0)\le\log_q\level(\pi)$.
\end{proof}

\begin{lemma} \label{lem: BushHen}
Let $F$ be a local non-archimedean field with residue field $\F_q$.
Consider $G=\GL(n)$, $G^{\der}=\SL(n)$, a maximal Levi subgroup $M=\GL(n_1)\times\GL(n_2)$ of $G$, and an irreducible smooth representation $\pi=\pi_1\otimes\pi_2$
of $M(F)$. 
Let $f(\pi_1\times\tilde\pi_2)$ be the exponent of the conductor of $\pi_1 \times \tilde\pi_2$ (cf.~\cite{MR1462836}). Then
\[
f(\pi_1\times\tilde\pi_2)\le n \log_q\level_M(\pi;G^{\der}).
\]
\end{lemma}

\begin{proof}
Since $f(\pi_1\times\tilde\pi_2)$ is not affected by twisting both $\pi_1$ and $\pi_2$ by the same
character, by Lemma \ref{lem: twist} it suffices to prove that
\[
f(\pi_1\times\tilde\pi_2)\le n\log_q\level_M(\pi).
\]
The results of \cite{MR1462836} give
\[
f(\pi_1\times\tilde\pi_2)\le n_1f(\pi_2)+n_2f(\pi_1),
\]
where $f(\pi_i)$ is the exponent of the conductor of $\pi_i$.
By Lemma \ref{LemmaConductor}, we have $f (\pi_i) \le \log_q \level_{\GL(n_i)}\pi_i= \log_q \level_M(\pi;\GL(n_i))\le\log_q\level_M(\pi)$.
The lemma follows.
\end{proof}

\begin{proof}[Proof of Proposition \ref{prop: mainglobal}]
By Lemma \ref{lem: trivred}, it suffices to consider the case of $G=\GL(n)$.
The global normalizing factors $n_\alpha$ can be expressed in terms
of Rankin--Selberg $L$-functions whose properties are summarized and analyzed in \cite[\S4 and \S5]{MR2276771}.
More precisely, write $M \simeq \prod_{i=1}^r \GL (n_i)$, where the root $\alpha$ is trivial on $\prod_{i \ge 3} \GL (n_i)$,
and let $\pi \simeq \otimes \pi_i$ with representations $\pi_i \in \Pi_{\disc}(\GL(n_i,\A))$.
Note that then $\prpr{M_\alpha}=\SL(n_1+n_2)$.
Let $L(s,\pi_1\times\tilde\pi_2)$ be the completed Rankin--Selberg $L$-function associated to $\pi_1$ and $\pi_2$. It satisfies the functional equation
\[
L(s,\pi_1\times\tilde\pi_2)=\eps(\frac12,\pi_1\times\tilde\pi_2)N(\pi_1\times\tilde\pi_2)^{\frac12-s}L(1-s,\tilde\pi_1\times\pi_2),
\]
where $\abs{\eps(\frac12,\pi_1\times\tilde\pi_2)}=1$ and
\[
N(\pi_1\times\tilde\pi_2) = d_F^{n_1 n_2} \prod_v q_v^{f_v (\pi_{1,v} \times \tilde\pi_{2,v})}
\]
is the conductor.
Here, the local exponents $f_v (\pi_{1,v} \times \tilde\pi_{2,v})$ are as in Lemma \ref{lem: BushHen} above, and
$d_F$ is the absolute value of the discriminant of $F$.
We can then write
\[
n_\alpha (\pi,s) = \frac{L(s,\pi_1\times\tilde\pi_2)}{\eps(\frac12,\pi_1\times\tilde\pi_2)N(\pi_1\times\tilde\pi_2)^{\frac12-s}L(s+1,\pi_1\times\tilde\pi_2)}.
\]
The proof of Propositions 4.5 and 5.1 in \cite{MR2276771} gives
\[
\int_T^{T+1}\abs{ \frac{n'_\alpha(\pi,\iii t)}{n_\alpha(\pi,\iii t)} } \ dt\ll \log (\abs{T} + \nu (\pi_1 \times \tilde\pi_2))
\]
with
\[
\nu (\pi_1 \times \tilde\pi_2) = N (\pi_1 \times \tilde\pi_2) ( 2 + c (\pi_1 \times \tilde\pi_2))
\]
and the quantity $c (\pi_1 \times \tilde\pi_2) \ge 0$ of [ibid., (4.21)], which depends only on the archimedean factors of $\pi_1$ and $\pi_2$.
Moreover, $c (\pi_1 \times \tilde \pi_2)$ is by definition invariant under twisting by characters of $\GL (n_1+n_2, F_\infty)$
(viewed as characters of $\GL(n_1,F_\infty) \times \GL (n_2,F_\infty)$ by restriction).
Combining [ibid., Lemma 4.2] and [ibid., Lemma 5.4]\footnote{Note that in [ibid., Lemma 4.2] it is not necessary to
assume that the $\pi_i$ are generic, it suffices that they are unitary. Also, [ibid., Lemma 5.4] deals explicitly only with the real places $v$ of $F$, but
complex places can be dealt with in the same way.}
shows that $\log ( 2 + c (\pi_1 \times \tilde\pi_2)) \ll \min_{\chi} \log \param_M(\pi_\infty \otimes \chi;\GL (n_1+n_2))$, where
$\chi$ ranges over the characters of $\GL (n_1+n_2, F_\infty)$.

Since the central character $\omega (\pi_\infty \otimes \chi)$ of the representation $\pi_\infty \otimes \chi$ is simply $\omega(\pi_\infty) \chi^{n_1+n_2}$,
by an appropriate choice of $\chi$ we can achieve that $\omega (\pi_\infty \otimes \chi)$ belongs to a finite set of characters that depends only on $n_1+n_2$ and $F_\infty$.
We can then clearly bound the contribution of $\Lie (Z (\GL (n_1+n_2,F_\infty)))$ to the Casimir eigenvalue of $\pi_\infty \otimes \chi$ in terms of $n_1+n_2$ and $[F:\Q]$,
and similarly for the central contribution to $\lambda_\tau$, where $\tau$ is a lowest $\K_\infty$-type of $\pi_\infty \otimes \chi$.
Altogether we obtain the estimate
$\log ( 2 + c (\pi_1 \times \tilde\pi_2)) \ll \log \param_M(\pi_\infty;\prpr{M_\alpha})$.
On the other hand, by Lemma \ref{lem: BushHen} we have
$\log N(\pi_1\times\tilde\pi_2)\ll1+\log\level_M(\pi;\prpr M_\alpha)$.
The proposition follows.
\end{proof}

\begin{remark} \label{RemarkGeneralG}
For a general group $G$ the normalizing factors are given, at least up to local factors,
by quotients of automorphic $L$-functions
associated to the irreducible constituents of the adjoint representation of the $L$-group $^LM$ of $M$ on the Lie algebra of the unipotent radical
of the corresponding parabolic subgroup of $^LG$ \cite{MR0419366}.
To argue as above, we would need to know that
these $L$-functions have finitely many poles and satisfy a functional equation with the associated conductor
bounded polynomially in $\level(\pi)$ for any $\pi\in\Pi_{\disc}(M(\A))$.
Unfortunately, the finiteness of the number of poles and the expected functional equation are not known in general
(although they are known in some important cases, e.g. \cite{MR1800349}).
It is possible that for classical groups these properties are within reach using the work of Arthur \cite{MR3135650}
and M\oe glin \cite{MR2925174}, and the same may apply to the exceptional group $G_2$ \cite{MR1726704}.
However, this may require further work and we will not pursue this matter any further.
\end{remark}

\subsection{Degrees of normalized local intertwining operators}
Let $M\in\levis$ and $P,Q\in\PPP(M)$ adjacent along $\alpha\in\rts_M$.
Recall that in the $p$-adic case, for any irreducible representation $\pi_v$ of $M(F_v)$ and any open subgroup
$K_v$ of $G(F_v)$ the restriction $R_{Q|P}(\pi_v,s)^{K_v}$ of the operator $R_{Q|P} (\pi_v,s)$ to the finite dimensional space of $K_v$-fixed vectors
in $I_P^G (\pi_v)$ is a rational function in $q^{-s}$.
More precisely, there exists a polynomial $f$, depending only on $\pi_v$, and whose degree is bounded in terms of $G$
only, such that $f(q^{-s})R_{Q|P}(\pi_v,s)^{K_v}$ is polynomial in $q^{-s}$.
We have $U_Q\cap U_{\bar P}=U_\alpha$, where $\{\alpha\}=\rts_Q\cap\rts_{\bar P}$.
Note that the operator $R_{Q|P}$ is obtained by induction from the intertwining operator $R^{M_\alpha}_{Q \cap M_\alpha|P \cap M_\alpha}$ between representations
of the group $M_\alpha (F_v)$. Let $P_\alpha$ be the maximal parabolic subgroup of $M_\alpha$ with unipotent radical $U_\alpha$ (and Levi subgroup $M$), and
$\bar{P}_\alpha$ the corresponding opposite parabolic subgroup.

\begin{definition}
We say that $G$ satisfies property \BD\ (bounded degree) if there exists a constant $c$ (depending only on $G$ and $\rho$),
such that for any
\begin{itemize}
\item Levi subgroup $M \in\levis$, $M \neq G$,
\item $\alpha \in \Sigma_M$,
\item finite place $v$ of $F$,
\item open subgroup $K_v\subset\K_{M_\alpha,v}$,
\item smooth irreducible representation $\pi_v$ of $M(F_v)$,
\end{itemize}
the degrees of the numerators of the linear operators $R^{M_\alpha}_{\bar{P}_\alpha|P_\alpha}(\pi_v,s)^{K_v}$
are bounded by
\[
\begin{cases}c\log_{q_v} \level_{M_\alpha} (K_v;\prpr{M}_\alpha^+),&\text{if $\K_v$ is hyperspecial},\\
c (1+\log_{q_v} \level_{M_\alpha} (K_v;\prpr{M}_\alpha^+)),&\text{otherwise.}
\end{cases}
\]
\end{definition}

\begin{remark} \label{RemarkBD}
\begin{enumerate}
\item Property \BD\ is by definition hereditary for Levi subgroups.

\item Property \BD\ is discussed in more detail in \cite{MR3001800}. Conjectures 1 and 2 of [ibid.] for all Levi subgroups of $G$ amount to the slightly
weaker statement obtained by replacing $\level_{M_\alpha} (K_v;\prpr{M}_\alpha^+)$
by $\level_{M_\alpha} (K_v; M_\alpha^{\der})$, but the above formulation is more natural (and we intend to use it
in a future paper).
Note that Conjectures 1 and 2 of [ibid.] for the $F$-simple factors of the groups $L^{\SC}$, $L \in \levis$, imply property \BD\ in its current formulation.
\item Note also that by \cite[Lemma 1.6]{MR1312501} (cf.~also \cite[Proposition 3]{MR3001800} and the remark following it)
we may replace $\level_{M_\alpha} (K_v;\prpr{M}_\alpha^+)$ by the inverse of the volume of
$(\scprj_{\prpr{M}_\alpha})^{-1}(K_v)$ in $\prpr{M}_\alpha^{\SC}(F_v)$.
\end{enumerate}
\end{remark}

Let now $\fctr M$ be the subgroup of $G$ generated by the unipotent subgroups $U_\alpha$, $\alpha\in\rts_M$.
In other words, if $P$ and $\bar P$ are (arbitrary) opposite parabolic subgroups in $\PPP(M)$, then $\fctr M$ is the
subgroup of $G$ generated by $U$ and $\bar U$.
By \cite[Proposition 4.11]{MR0207712}, $\fctr M$ is a connected semisimple normal subgroup of $G$ defined over $F$. Clearly, all non-central normal
subgroups of $\fctr M$ are $F$-isotropic, and therefore $\fctr M (\A)^+$ is the closed subgroup of $G (\A)$
generated by the groups $U_\alpha (\A)$, $\alpha \in \Sigma_M$. It is also the derived group of $\fctr M (\A)$.

Property \BD\ has the following consequence for the operators $R_{Q|P} (\pi_v,s)$, which is
a slight strengthening of \cite[Proposition 6]{MR3001800}.\footnote{We take the opportunity to correct
an inaccuracy in the proof (not the statement) of \cite[Proposition 5]{MR3001800}. Using the notation of [ibid.] freely,
in the last sentence of the proof it is stated that the matrix coefficients of $M_{Q_{i+1}|Q_i} (\sigma,\sprod{\lambda}{\alpha_i^\vee})^K$
are given by those of $M_{\overline{Q'}|Q'} (\sigma,\sprod{\lambda}{\alpha_i^\vee})^{K \cap M_R}$. Correctly, they are given by
the matrix coefficients of
$M_{\overline{Q'}|Q'} (\sigma,\sprod{\lambda}{\alpha_i^\vee})^{\gamma K \gamma^{-1} \cap M_R}$, where $\gamma$ ranges over $K_0$.
However, we are free to consider instead of $K$ the largest principal congruence subgroup $K'_n$ of $K'_0$ contained in $K$, which is
a normal subgroup of $K_0$. The same correction applies to the proof of [ibid., Lemma 20].}

\begin{lemma} \label{LemmaBD}
Let $G$ satisfy property \BD. Then for any
\begin{itemize}
\item Levi subgroup $M \in\levis$, $M \neq G$,
\item any adjacent parabolic subgroups $P, Q \in \PPP (M)$,
\item finite place $v$ of $F$,
\item open subgroup $K_v\subset\K_v$,
\item smooth irreducible representation $\pi_v$ of $M(F_v)$,
\end{itemize}
the degrees of the numerators of the linear operators $R_{Q|P}(\pi_v,s)^{K_v}$
are bounded by
\[
\begin{cases}c\log_{q_v} \level(K_v;\fctr M^+),&\text{if $\K_v$ is hyperspecial},\\
c (1+\log_{q_v} \level(K_v;\fctr M^+)),&\text{otherwise.}
\end{cases}
\]
\end{lemma}

\begin{proof}
The matrix coefficients of $R_{Q|P}(\pi_v,s)^{K_v}$ are linear combinations of
the matrix coefficients of $R_{Q\cap M_\alpha|P\cap M_\alpha}^{M_\alpha}(\pi_v,s)^{\gamma K_v\gamma^{-1}\cap M_\alpha(F_v)}$, $\gamma\in\K_v$.
Property \BD\ implies that the degrees of the latter coefficients are bounded by $c \log_{q_v} \level (\gamma K_v\gamma^{-1}\cap M_\alpha(F_v); \prpr{M}_\alpha^+)$
in the hyperspecial case and by
$c (1+\log_{q_v} \level (\gamma K_v\gamma^{-1}\cap M_\alpha(F_v); \prpr{M}_\alpha^+))$, otherwise.
Since $\fctr M(\A_{\fin})^+$ is normal in $G(\A_{\fin})$ and
$\fctr M(\A_{\fin})^+\supset\prpr M_\alpha(\A_{\fin})^+$, we obtain the required estimate.
\end{proof}

The proof of Lemma \ref{LemmaBD} yields in addition the following estimate.
\begin{lemma} \label{lem: cmprGm}
Let $\alpha\in\rts_M$. Then $\level_M(\pi;\prpr M_\alpha^+)\le\level_G(I_P^G(\pi);\fctr M^+)$.
In other words, if $K$ is an open subgroup of $\K_{\fin}$ such that
$I_P^G(\pi)^K\ne0$ then $\level_M(\pi;\prpr M_\alpha^+)\le\level(K;\fctr M^+)$.
\end{lemma}

\begin{proof}
Suppose that $I_P^G(\pi)^K\ne0$.
Since the finite part of $I_P^G(\pi)$ is $I_{P (\A_{\fin}) \cap\K_{\fin}}^{\K_{\fin}}\pi\rest_{\K_{M,\fin}}$
as a $\K_{\fin}$-module, $\pi$ must have a non-zero vector fixed under $\proj_M(\gamma K\gamma^{-1}\cap P(\A_{\fin}))$
for some $\gamma\in\K_{\fin}$, where $\proj_M:P\rightarrow M$ denotes the canonical projection.
Suppose that $K\supset\K(\nnn)\cap\fctr M(\A_{\fin})^+$.
Since $\fctr M(\A_{\fin})^+$ is normal in $G(\A)$ and $\fctr M(\A_{\fin})^+\supset\prpr M_\alpha(\A_{\fin})^+$
we get
\[
\proj_M(\gamma K\gamma^{-1}\cap P(\A_{\fin}))\supset\proj_M(\K(\nnn)\cap \fctr M(\A_{\fin})^+\cap P(\A_{\fin}))
\supset\K_M(\nnn)\cap\prpr M_\alpha(\A_{\fin})^+.
\]
The lemma follows.
\end{proof}

\begin{remark} \label{RemarkUnramifiedBD}
Note that for $\K_v$ hyperspecial and $I^{G(F_v)}_{P(F_v)} (\pi_v)^{\K_v \cap \fctr M (F_v)^+} \neq 0$
the operator $R_{Q|P} (\pi_v,s)^{\K_v \cap \fctr M (F_v)^+}$ is independent of $s$.
This implies that the assertion of Lemma \ref{LemmaBD} is satisfied unconditionally if $\K_v$ is hyperspecial and $\level (K_v;\fctr M^+)=1$.
(In particular, the special case of property \BD\ where $\K_v$ is
hyperspecial and $\level_{M_\alpha} (K_v;\prpr{M}_\alpha^+) = 1$ is always true.)
To see this, note that via $p^{\SC}|_{G_M^{\SC}}$ we can pull back $\K_v$ to a hyperspecial maximal compact subgroup of $G_M^{\SC} (F_v)$,
$M$ to a proper Levi subgroup $\tilde{M}$ of $G_M^{\SC}$,
$P$ and $Q$ to adjacent parabolic subgroups $\tilde{P}$ and $\tilde{Q}$ with Levi subgroup $\tilde{M}$, and $\pi$
to a smooth representation $\tilde{\pi}$ of $\tilde{M} (F_v)$, which is a finite direct sum of smooth irreducible representations.
Then we can reduce to the property that $R_{\tilde{Q}|\tilde{P}}$ acts as a constant
on the one-dimensional space of unramified vectors in the induction to $G_M^{\SC} (F_v)$ of
any unramified irreducible summand of $\tilde{\pi}$.
\end{remark}

We have the following analog of Lemma \ref{lem: trivred}.
\begin{lemma} \label{LemmaGderBD}
Suppose that $\tilde G$ satisfies property \BD\ and $G$ is a subgroup of $\tilde G$ containing $\tilde G^{\der}$.
Then $G$ also satisfies \BD.
\end{lemma}

\begin{proof}
Let $\tilde M\in\levis^{\tilde G}$ and let $\tilde P,\tilde Q\in\PPP^{\tilde G}(\tilde M)$ be adjacent along $\alpha$.
Then $M=\tilde M\cap G\in\levis^G$ and $P=\tilde P\cap G\in\PPP^G(M)$ and $Q=\tilde Q\cap G\in\PPP^G(M)$ are adjacent along $\alpha$
and if $\tilde\pi_v$ is an irreducible representation of $\tilde M(F_v)$ then we can identify
the restriction of $I_{\tilde P}^{\tilde G}(\tilde\pi_v)$ to $G(F_v)$ with $I_P^G(\tilde\pi_v\big|_{M(F_v)})$.
Moreover, if $\pi_v$ is a constituent of $\tilde\pi_v\big|_{M(F_v)}$, then $R_{Q|P}(\pi_v,s)$ is the restriction
of $R_{\tilde Q|\tilde P}(\tilde\pi_v,s)$. The lemma follows immediately.
\end{proof}

Finally, the main result of \cite{MR3001800} can be phrased as follows.

\begin{theorem} \label{TheoremLocalGLn}
The groups $\GL(n)$ and $\SL(n)$ satisfy \BD.
\end{theorem}

\begin{proof} The $\GL(n)$ case follows directly from \cite[Theorem 1]{MR3001800}, taking into account that Levi subgroups
of $\GL(n)$ are products of groups $\GL(n_i)$.
For $\SL(n)$ we may use Lemma \ref{LemmaGderBD}.
\end{proof}

\subsection{Logarithmic derivatives of normalized local intertwining operators}
The relevance of property \BD\ to the trace formula lies in the following consequence, which we will prove in the remainder of this section.

\begin{proposition} \label{prop: mainlocal}
Suppose that $G$ satisfies \BD. Let $M\in\levis$ and $P$, $Q \in \PPP (M)$ be adjacent parabolic subgroups. Then
for all open compact subgroups $K$ of $G(\A_{\fin})$ and all $\tau \in \Pi (\K_\infty\cap\fctr M)$
we have
\[
\int_{\iii\R}\norm{R_{Q|P}(\pi,s)^{-1}R'_{Q|P}(\pi,s)\big|_{I^G_P(\pi)^{\tau,K}}}(1+\abs{s}^2)^{-1}\ ds\ll
1+\log(\norm{\tau}+\level(K;\fctr M^+)).
\]
\end{proposition}

We remark that the dependence of the bound on $\tau$ is not essential for the limit multiplicity problem, but it is relevant for other
asymptotic questions.

The key ingredient for the proof of Proposition \ref{prop: mainlocal} is the following generalization of the classical Bernstein inequality
due to Borwein and Erd\'elyi.
\begin{proposition}[\cite{MR1433285}] \label{prop: Bor-Erd}
Let $\ucirc$ be the unit circle in $\C$ and let $z_1,\dots,z_m\notin\ucirc$.
Suppose that $f(z)$ is a rational function on $\C$ such that $\sup_{z\in\ucirc}\abs{f(z)}\le1$ and
$(z-z_1)\dots (z-z_m)f(z)$ is a polynomial of degree $n$. Then
\[
\abs{f'(z)}\le\max(\max(n-m,0)+\sum_{j:\abs{z_j}>1}\frac{\abs{z_j}^2-1}{\abs{z_j-z}^2},
\sum_{j:\abs{z_j}<1}\frac{1-\abs{z_j}^2}{\abs{z_j-z}^2}),\quad  z\in\ucirc.
\]
\end{proposition}

Note that in [ibid., Theorem 1] this inequality is stated explicitly only for $n\le m$. However, as explained on [ibid., p. 418],
the case $n>m$ follows from the previous case by passing to
\[
\frac{f(z)(\abs{w_1}-1)\dots(\abs{w_{n-m}}-1)}{(z-w_1)\dots (z-w_{n-m})},
\]
where $w_1,\dots,w_{n-m} \in \C$ are auxiliary parameters with $\abs{w_i} > 1$, and then letting $w_i \to \infty$.

We will need a vector-valued version of Proposition \ref{prop: Bor-Erd} which is a direct consequence.
For the next two lemmas let $V$ be a normed space over $\C$.

\begin{corollary} \label{cor: derbnd}
Let $z_1,\dots,z_m \in \C\setminus\ucirc$.
Suppose that $A:\C\setminus\{z_1,\dots,z_m\}\rightarrow V$ is such that $(z-z_1)\dots (z-z_m)A(z)$
is a polynomial in $z\in\C$ of degree $n$ with coefficients in $V$.
Assume that $\norm{A(z)}\le 1$ for all $z\in\ucirc$.
Then
\[
\norm{A'(z)}\le\max(\max(n-m,0)+\sum_{j:\abs{z_j}>1}\frac{\abs{z_j}^2-1}{\abs{z_j-z}^2},
\sum_{j:\abs{z_j}<1}\frac{1-\abs{z_j}^2}{\abs{z_j-z}^2}),\quad z\in\ucirc.
\]
Consequently,
\begin{equation} \label{eq: normAint}
\int_{\ucirc}\norm{A'(z)}\ \abs{dz}\le 2\pi\max(m,n).
\end{equation}
\end{corollary}

This follows from Proposition \ref{prop: Bor-Erd} by applying it to $f(z) = (A'(z),w)$ for any linear form $w$ on $V$
with $\norm{w}\le1$.

We remark that when $V=\C$, i.e., when $A(z)$ is scalar valued, the bound \eqref{eq: normAint},
at least with $2\pi$ replaced by $8$, can be easily proved directly
without appealing to Proposition \ref{prop: Bor-Erd} (see \cite[Lemma 1]{MR2811597}).
However, we do not know a direct proof of \eqref{eq: normAint} (even with $2\pi$ replaced by an arbitrary constant)
in the general case.

Analogously, we have:

\begin{lemma} \label{lem: derbnd2}
Let $z_j=u_j+\iii v_j\in\C$, $j=1,\dots,m$ and $b(z)=(z-z_1)\dots (z-z_m)$.
Suppose that $A:\C\setminus\{z_1,\dots,z_m\}\rightarrow V$ is such that $\norm{A(z)}\le 1$ for all $z\in\iii\R$ and $b(z)A(z)$
is a polynomial in $z\in\C$ (necessarily of degree $\le m$) with coefficients in $V$.
Then
\[
\int_{\iii\R}\norm{A'(z)}\frac{\abs{dz}}{1+\abs{z}^2}\le 2\pi\sum_{j=1}^m\frac{\abs{u_j}+1}{(\abs{u_j}+1)^2+v_j^2}\le2\pi m.
\]
\end{lemma}

\begin{proof}
For any $w\in\C$ let $\phi_w(z)=\frac{z+\bar w}{z-w}$ and set
\[
\phi_\gtrless(z)=\prod_{j:\Re{z_j}\gtrless0}\phi_{z_j}(z).
\]
Applying \cite[Theorem 4]{MR1433285},\footnote{This result is misstated on \cite[p. 190]{MR2811597}.} we conclude as before that
\[
\norm{A'(z)}\le\max(\abs{\phi_>'(z)},\abs{\phi_<'(z)})\le\abs{\phi_>'(z)}+\abs{\phi_<'(z)}, \ \ z\in\iii\R.
\]

It remains to observe that for any $w=u+\iii v\in\C\setminus\iii\R$ we have
\[
\int_{\iii\R}\abs{\phi_w'(z)}\frac{\abs{dz}}{1+\abs{z}^2}=2\pi\frac{\abs{u}+1}{(\abs{u}+1)^2+v^2}.
\]
Indeed, we have $\abs{\phi_w'(z)}=\frac{2\abs{u}}{\abs{z-w}^2}=\frac{2\abs{u}}{u^2+(t-v)^2}$ for $z=\iii t$, $t\in\R$, so that
\[
\int_{\iii\R}\abs{\phi_w'(z)}\frac{\abs{dz}}{1+\abs{z}^2}=\int_{\R}\frac{2\abs{u}}{(u^2+(t-v)^2)(1+t^2)}\ dt.
\]
By the residue theorem this is equal to
\begin{align*}
2\pi\left(\frac{\abs{u}}{u^2+(\iii-v)^2}+\frac1{1+(v+\iii\abs{u})^2}\right)&=
\frac{2\pi}{v+\iii(\abs{u}-1)}\left(\frac{\abs{u}}{v-\iii(\abs{u}+1)}+\frac1{v+\iii(\abs{u}+1)}\right)\\&=
\frac{2\pi(\abs{u}+1)}{v^2+(\abs{u}+1)^2},
\end{align*}
as claimed.
\end{proof}

\begin{proof}[Proof of Proposition \ref{prop: mainlocal}]
Replacing $K$ by its largest factorizable subgroup does not change $\level (K;\fctr M^+)$.
We may therefore assume
that $K$ is factorizable. Write $K=\prod K_v$ and set $K^v=\prod_{w\ne v}K_w$.
Also set $N_v=\level_v(K_v;\fctr M^+)$ so that $\prod N_v=\level(K;\fctr M^+)$.
Let $S'$ (resp., $S''$) be the finite set of finite places such that $\K_v$ is not hyperspecial (resp., $N_v\ne1$).
Of course $S'$ depends only on $\K_{\fin}$.
Note that by Remark \ref{RemarkUnramifiedBD}, $R_v(\pi_v,s)^{K_v}$ is independent of $s$ if $v$ is finite and $v\notin S'\cup S''$.
We have
\begin{align*}
R(\pi,s)^{-1}R'(\pi,s)\big|_{I(\pi)^{\tau,K}} & = R_\infty(\pi_\infty,s)^{-1}R_\infty'(\pi_\infty,s)\big|_{I(\pi_\infty)^\tau}\otimes\Id_{I(\pi^\infty)^K} \\
&+\sum_{v\in S'\cup S''}R_v(\pi_v,s)^{-1}R_v'(\pi_v,s)\big|_{I(\pi_v)^{K_v}}\otimes\Id_{I(\pi^v)^{\tau,K^v}}.
\end{align*}
Recall that the operators $R_v(\pi_v,s)$ are unitary for $\Re s = 0$.

Consider first the case where $v\in S'\cup S''$.
Write $R_v(\pi_v,s)\big|_{I(\pi_v)^{K_v}}=A_v(q_v^{-s})$. Then
\begin{multline} \label{eq: int123}
\int_{\iii\R}\norm{R_v(\pi_v,s)^{-1}R_v'(\pi_v,s)\big|_{I(\pi_v)^{K_v}}}\frac{ds}{1+\abs{s}^2}
=\int_{\iii\R}\norm{R_v'(\pi_v,s)\big|_{I(\pi_v)^{K_v}}}\frac{ds}{1+\abs{s}^2}\\
\le
2\sum_{n=0}^\infty \big( 1+\frac{4\pi^2n^2}{(\log q_v)^2} \big)^{-1}
\int_0^{\frac{2\pi\iii}{\log q_v}}\norm{R_v'(\pi_v,s)\big|_{I(\pi_v)^{K_v}}}\ ds\\
\le
2 \big(1+\frac14\log q_v \big) \int_0^{\frac{2\pi\iii}{\log q_v}}\norm{R_v'(\pi_v,s)\big|_{I(\pi_v)^{K_v}}}\ ds=
\big( 2+\frac12\log q_v \big) \int_{\ucirc}\norm{A_v'(z)}\ \abs{dz}.
\end{multline}
By Lemma \ref{LemmaBD}, property \BD\ for $G$ implies that
$A_v$ satisfies the conditions of Corollary \ref{cor: derbnd} (with respect to the operator norm) with
$m$ bounded in terms of $G$ only and
\[
n \ll\begin{cases}1+\log_{q_v}N_v,&\text{if }v\in S',\\\log_{q_v}N_v,&\text{otherwise.}\end{cases}
\]
By Corollary \ref{cor: derbnd}, \eqref{eq: int123} is therefore $\ll (\log q_v)(\log_{q_v}N_v)=\log N_v$, if $v\notin S'$
and $\ll 1 + \log N_v$, otherwise.

Regarding the archimedean contribution, it follows from \cite[Proposition A.2]{MR2053600} that the operator
$R_\infty(\pi_\infty,s)\big|_{I(\pi_\infty)^{\tau}}$ satisfies the conditions of Lemma \ref{lem: derbnd2}
with $b(s)=\prod_{j=1}^r\prod_{k=1}^m(s-\rho_j+ck)$, where
\begin{itemize}
\item $c>0$ depends only on $M$,
\item $r$ is bounded in terms of $G$ only,
\item $m\ll 1+\norm{\tau}$.
\end{itemize}
(In addition, the real parts $\Re\rho_j$ are bounded from above in terms of $G$ only, but we will not need to use this fact.)
Note that although [ibid., Proposition A.2] gives the bound $m\ll 1+\norm{\tilde\tau}$ on the $\tilde\tau$-isotypic subspace, where
$\tilde\tau \in \Pi (\K_\infty)$, on [ibid., p.~88] it is explicitly stated that we may in fact consider the isotypic subspace for
a representation of $\K_\infty \cap G^{\der} (F_\infty)$, and it is clear from the definition of $R_\infty (\pi_\infty,s)$
that we may even replace $G^{\der}$ by $\fctr M$.

Write $\rho_j=u_j+\iii v_j$. By Lemma \ref{lem: derbnd2} we infer that
\begin{multline*}
\int_{\iii\R}\norm{R_\infty(\pi_\infty,s)^{-1}R_\infty'(\pi_\infty,s)\big|_{I(\pi_\infty)^{\tau}}}(1+\abs{s}^2)^{-1}\ ds=
\int_{\iii\R}\norm{R_\infty'(\pi_\infty,s)\big|_{I(\pi_\infty)^{\tau}}}(1+\abs{s}^2)^{-1}\ ds\\\ll
\sum_{j=1}^r\sum_{k=1}^m\frac{\abs{u_j-ck}+1}{(\abs{u_j-ck}+1)^2+v_j^2}\ll 1+\log(1+\norm{\tau}).
\end{multline*}

Altogether,
\begin{align*}
\int_{\iii\R}\norm{R(\pi,s)^{-1}R'(\pi,s)\big|_{I(\pi)^{\tau,K}}}(1+\abs{s})^{-2}\ ds &\ll
1+\log(1+\norm{\tau})+\sum_{v\text{ finite}}\log N_v\\
&= 1+\log(1+\norm{\tau})+\log \level(K;\fctr M^+),
\end{align*}
as required.
\end{proof}


\section{Polynomially bounded collections of measures} \label{SectionPB}

As a preparation for our proof of the spectral limit property in \S \ref{sec: spectralside}, we prove in this section a result on real reductive Lie groups
(Proposition \ref{Delormeprop} below) which extends an argument of Delorme in \cite{MR860667}. Let temporarily
$G_\infty$ be the group of real points of a connected reductive group defined over $\R$,
or, slightly more generally, the quotient of such a group by a connected subgroup of its center
(like the group $G(F_\infty)^1$ to which we will apply our results in \S \ref{sec: spectralside}).
Let $K_\infty$ be a maximal compact subgroup of $G_\infty$ and $\theta$ the associated Cartan involution.
We will consider Levi subgroups $M$ and parabolic subgroups $P$ defined over $\R$.
All Levi subgroups are implicitly assumed to be $\theta$-stable.
We factorize any Levi subgroup $M$ as a direct product $M = A_M\times M^1$, where
$A_M$ is the largest central subgroup of $M$ isomorphic to a power of $\R^{>0}$, and let $\aaa_M = \Lie A_M$.
We identify representations of $M^1$ with representations of $M$ on which $A_M$ acts trivially.
Fix a minimal $\theta$-stable Levi subgroup $M_0$.
As in \S \ref{subsecnotation}, we fix an invariant bilinear form $B$ on $\Lie G_\infty$, which
induces Euclidean norms on all its subspaces and
therefore Hermitian norms on the spaces $\aaa_{M,\C}^*$. For each $r > 0$ and each finite set
$\types \subset \Pi (K_\infty)$
we define
$\Hecke (G_\infty)_{r,\types}$ as the space of all smooth functions $f$ on $G_\infty$ with support contained
in the compact set $K_\infty \exp (\{ x \in \aaa_0 \, : \, \norm{x} \le r \}) K_\infty$ whose
translates $f (k_1 \cdot k_2)$, $k_1$, $k_2 \in K_\infty$,
span a finite dimensional space that decomposes under the action
of $K_\infty \times K_\infty$ as a sum of representations $\tau_1 \otimes \tau_2$ with $\tau_1$, $\tau_2 \in \types$.
We let $\Hecke (G_\infty)_r$ be the union of the spaces $\Hecke (G_\infty)_{r,\types}$ over all finite sets $\types \subset \Pi (K_\infty)$. The union
of the spaces $\Hecke (G_\infty)_r$ for all $r>0$ is then the space $\Hecke (G_\infty)$ introduced in \S \ref{SectionIntro}.

As before, for $f \in C^\infty (G_\infty)$ and $k\ge0$ let
\[
\norm{f}_k
=\sum_{X\in\mathcal{B}_k}\norm{X\star f}_{L^1(G_\infty)}.
\]
These norms endow $\Hecke (G_\infty)_{r,\types}$ with the structure of a Fr\'{e}chet space.

Let $\Irr (G_\infty)$ be the set of all irreducible admissible representations of $G_\infty$ up to infinitesimal
equivalence.
The unitary dual $\Pi (G_\infty)$ can be viewed as a subset of $\Irr (G_\infty)$ in a natural way.
For $\pi \in \Irr (G_\infty)$ denote its infinitesimal character by $\chi_\pi$
and its Casimir eigenvalue (which depends only on $\chi_\pi$) by $\lambda_\pi$.
For any $\mu\in \Pi (K_\infty)$ let $\Irr(G_\infty)_\mu$ be the set of irreducible
representations containing $\mu$ as a $K_\infty$-type.
More generally, for any subset $\types$ of $\Pi( K_\infty)$ we write $\Irr(G_\infty)_\types=
\cup_{\tau\in\types}\Irr(G_\infty)_\tau$.

We write $\data$ for the set of all conjugacy classes of pairs $(M,\delta)$ consisting of
a Levi subgroup $M$ of $G_\infty$ and a discrete series representation $\delta$ of $M^1$.
For any $\underline{\delta} \in \data$ let
$\Irr(G_\infty)_{\underline{\delta}}$ be the set of all irreducible representations
which arise by the Langlands quotient construction
from the (tempered) irreducible constituents of $I_{M}^{L}(\delta)$ for Levi subgroups $L \supset M$.
Here, $I_{M}^L$ denotes (unitary) induction from an arbitrary parabolic subgroup of $L$
with Levi subgroup $M$ to $L$.
We then have a disjoint decomposition
\[
\Irr(G_\infty)=\coprod_{\underline{\delta}\in\data}\Irr(G_\infty)_{\underline{\delta}}
\]
and consequently
\[
\Pi(G_\infty)=\coprod_{\underline{\delta}\in\data}\Pi(G_\infty)_{\underline{\delta}}.
\]

Recall the definition of the norm $\norm{\cdot}$ on $\Pi (K_\infty)$ given in \S \ref{SubsectionLevel}.
We call a finite subset $\types \subset \Pi (K_\infty)$
saturated if for each $\mu \in \types$, all $\mu' \in \Pi (K_\infty)$ with
$\norm{\mu'} \le \norm{\mu}$ are also contained in $\types$.

For $\pi \in \Irr (G_\infty)$ we write $\underline{\delta} (\pi)$ for the unique element $\underline{\delta} \in \data$
with $\pi \in \Irr(G_\infty)_{\underline{\delta}}$.
We introduce a partial order on $\data$ as in \cite[\S 2.3]{MR759263}, using the lowest $K_\infty$-types of $I^G_M (\delta)$:
$\underline{\delta} \prec \underline{\delta}'$ if and only if $\norm{\mu} < \norm{\mu'}$ for lowest $K_\infty$-types $\mu$ and $\mu'$ of $I^G_M (\delta)$ and
$I^G_{M'} (\delta')$, respectively.

When $\types \subset \Pi( K_\infty)$ is finite and saturated, Vogan's theory of lowest $K_\infty$-types implies that
\begin{equation} \label{eq: decomtypes}
\Irr(G_\infty)_\types=\cup_{\underline{\delta}\in\data_\types}\Irr(G_\infty)_{\underline{\delta}}
\end{equation}
for a finite subset $\data_\types \subset \data$ (cf. \cite[Proposition D.1]{MR1046496}).

For ${\underline{\delta}} \in \data$ we let $\types ({\underline{\delta}})$ be the
finite saturated set of all $\mu' \in \Pi (K_\infty)$ with
$\norm{\mu'} \le \norm{\mu}$ for a lowest $K_\infty$-type $\mu$ of $I^G_M (\delta)$.

\begin{proposition} \label{Delormeprop}
Let $\mathfrak{M}$ be a set of Borel measures on $\Pi (G_\infty)$. Then the following conditions on $\mathfrak{M}$ are equivalent:
\begin{enumerate}
\item For all $\underline{\delta} \in \mathcal{D}$ there exists $N_{\underline{\delta}}>0$ such that
\[
\nu (\{\pi \in \Pi (G_\infty)_{\underline{\delta}} : \abs{\lambda_\pi} \le R \}) \ll_{\underline{\delta}} (1+R)^{N_{\underline{\delta}}}
\]
for all $\nu \in \mathfrak{M}$ and $R > 0$.
\item There exists $r > 0$ such that
$\sup_{\nu \in \mathfrak{M}} \abs{\nu(\hat{f})}$ is a continuous seminorm
on $\Hecke (G_\infty)_{r,\mathcal{F}}$ for any finite set $\mathcal{F} \subset \Pi (K_\infty)$.
\item $\sup_{\nu \in \mathfrak{M}} \abs{\nu(\hat{f})}$ is a continuous seminorm
on $\Hecke (G_\infty)_{r,\mathcal{F}}$ for any $r > 0$ and
any finite set $\mathcal{F} \subset \Pi (K_\infty)$.
\item For each finite set $\mathcal{F} \subset \Pi (K_\infty)$
there exists an integer $k = k (\mathcal{F})$ such that $\sup_{\nu \in \mathfrak{M}} \nu (g_{k,\mathcal{F}}) < \infty$,
where $g_{k,\mathcal{F}}$ is the non-negative function on $\Pi (G_\infty)$ defined by
\[
g_{k,\mathcal{F}} (\pi) =\begin{cases}(1+ \abs{\lambda_\pi})^{-k},&\text{if }\pi \in \Pi (G_\infty)_{\mathcal{F}},\\
0,&\text{otherwise.}\end{cases}
\]
\end{enumerate}
\end{proposition}

\begin{definition} \label{DefinitionBounded}
We call a collection $\mathfrak{M}$ of measures satisfying the equivalent conditions of Proposition \ref{Delormeprop} \emph{\PB}.
\end{definition}

We begin the proof of Proposition \ref{Delormeprop}.
Let $\mathfrak{M}$ be a collection of Borel measures on $\Pi (G_\infty)$.
Evidently the third condition of the proposition implies the second one.
Note that if $z_{G_\infty}$ is the Casimir element in the center of $\univ(\Lie G_\infty \otimes \C)$,
then $\widehat{z_{G_\infty} f}(\pi)=\lambda_\pi\hat f(\pi)$. Since also $\abs{\hat f(\pi)}\le\norm{f}_{L^1(G_\infty)}$,
it follows that for any $k\ge0$ we have
\[
\hat{f}(\pi)\ll_k \norm{f}_{2k} g_{k,\types} (\pi)
\]
for all $f \in \Hecke (G_\infty)_\types$ and $\pi \in \Pi (G_\infty)$.
We infer that the fourth condition of the proposition implies the third one.

For $k\ge 0$ and $\underline{\delta} \in \mathcal{D}$ let $g_{k,{\underline{\delta}}} = (1 + \abs{\lambda_\pi})^{-k}$
for $\pi \in \Pi (G_\infty)_{\underline{\delta}}$, and extend this function by zero to all of $\Pi (G_\infty)$.

For a given $\underline{\delta} \in \mathcal{D}$, consider the following two statements:
\begin{subequations}
\begin{multline} \label{eq: statement1}
\text{ There exists $N_{\underline{\delta}}>0$ such that}\\
\nu (\{\pi \in \Pi (G_\infty)_{\underline{\delta}} : \abs{\lambda_\pi} \le R \}) \ll_{\underline{\delta}} (1+R)^{N_{\underline{\delta}}}
\ \ \text{for all $\nu \in \mathfrak{M}$ and $R \ge 0$.}
\end{multline}
\begin{equation} \label{eq: statement2}
\text{There exists an integer }k=k_{\underline{\delta}} > 0\text{ such that }
\sup_{\nu \in \mathfrak{M}} \nu (g_{k,{\underline{\delta}}}) < \infty.
\end{equation}
\end{subequations}

It is easy to see that these statements are equivalent: if \eqref{eq: statement1} is satisfied, then we can bound
\[
\nu (g_{k,{\underline{\delta}}}) \le \sum_{m \ge 0}
\nu (\{\pi \in \Pi (G_\infty)_{\underline{\delta}} : m \le \abs{\lambda_\pi} \le m+1 \}) (m+1)^{-k}
\ll_{\underline{\delta}} \sum_{m \ge 0} \frac{(m+2)^{N_{\underline{\delta}}}}{(m+1)^k},
\]
which is bounded independently of $\nu \in \mathfrak{M}$ for $k \ge N_{\underline{\delta}}+2$. On the other hand, we clearly
have
\[
(1+R)^{-k} \nu (\{\pi \in \Pi (G_\infty)_{\underline{\delta}} : \abs{\lambda_\pi} \le R \}) \le \nu (g_{k,{\underline{\delta}}}),
\]
which gives the other implication.

Observe now that the first condition of the proposition is just \eqref{eq: statement1} for all $\underline{\delta}$.
Moreover, by \eqref{eq: decomtypes} the fourth condition is equivalent to \eqref{eq: statement2} for all $\underline{\delta}$.
Therefore, the first and fourth conditions of the proposition are equivalent.

It remains to show that the second condition implies the first (or the fourth) one.
This step is somewhat more difficult and requires some preliminary results, namely
the classification of tempered and admissible representations of $G_\infty$ and the Paley--Wiener theorem.
We first recall Vogan's classification of irreducible admissible representations.
For $(M, \delta)$ as above,
and $\lambda \in \aaa^*_{M,\C}$, consider the induced representation $\pi_{\delta,\lambda}$ (with respect to any
parabolic subgroup containing $M$ as a Levi subgroup). Its semi-simplification depends only on
the $K_\infty$-conjugacy class of the triple $(M, \delta,\lambda)$. Vogan defines the $R$-group $R_\delta$ of $\delta$,
a finite group of exponent two,
as well as its subgroup $R_{\delta,\lambda}$. The dual group $\hat{R}_\delta$ acts simply transitively
on the set $A (\delta)$ of lowest $K_\infty$-types of $\pi_{\delta, \lambda}$.
We then have a decomposition of the representation
$\pi_{\delta,\lambda}$ as a direct sum of $\abs{R_{\delta,\lambda}}$ many representations $\pi_{\delta,\lambda} (\mu)$, where
$\mu$ is an orbit of $R^\perp_{\delta,\lambda}$ in $A(\delta)$
\cite[6.5.10, 6.5.11]{MR632407}:
\[
\pi_{\delta,\lambda} = \bigoplus_{\mu \in A (\delta) / R^\perp_{\delta,\lambda}} \pi_{\delta,\lambda} (\mu).
\]
We call the $\pi_{\delta,\lambda} (\mu)$'s \emph{basic representations}.
Each basic representation $\pi_{\delta,\lambda} (\mu)$ has a unique irreducible subquotient $\bar{\pi}_{\delta,\lambda} (\mu)$ containing
a $K_\infty$-type in the orbit $\mu$. Alternatively, this subquotient can also be constructed as a Langlands quotient \cite[6.6.14, 6.6.15]{MR632407}.
This construction sets up a bijection $\bar{\pi}_{\delta,\lambda} (\mu) = \pi \mapsto \sigma_\pi = \pi_{\delta,\lambda} (\mu)$ between
infinitesimal equivalence classes of irreducible admissible representations $\pi$ and
basic representations $\sigma_\pi$, where the latter are interpreted as elements of the Grothendieck group
of admissible representations \cite[6.5.13]{MR632407}. By definition, the parametrization is compatible with the disjoint decomposition of
$\Irr (G_\infty)$ according to the elements of $\data$.

The distributions $\tr \sigma_\pi$ for $\pi \in \Irr (G_\infty)$ form a basis of the Grothendieck group of admissible representations. More precisely,
we have the following relations expressing the characters of irreducible representations $\pi \in \Irr (G_\infty)$ in terms of the
characters of basic representations:
\begin{equation} \label{grothendieck}
\tr \pi (\phi) = \tr \sigma_\pi (\phi) + \sum_{\pi': {\underline{\delta}}(\pi) \prec {\underline{\delta}} (\pi'), \,
\chi_{\pi} = \chi_{\pi'}} n (\pi,\pi') \tr \sigma_{\pi'} (\phi)
\end{equation}
with certain integers $n (\pi,\pi')$ \cite[6.6.7]{MR632407}. Note that here the sum on the right-hand side is finite.
For our purposes, all we need to know about the integers $n(\pi,\pi')$ is the following uniform boundedness property
\cite[Proposition 2.2]{MR860667}.

\begin{lemma}[Vogan] \label{lem: vogan}
We have
\[
\sum_{\pi': {\underline{\delta}}(\pi) \prec {\underline{\delta}} (\pi'), \,
\chi_{\pi} = \chi_{\pi'}} \abs{n(\pi,\pi')} \ll 1
\]
for all $\pi \in \Irr (G_\infty)$.
\end{lemma}

For the Paley--Wiener theorem, we need to group the basic representations into series of induced representations, which gives
a slightly different parametrization. We use the concept of a non-degenerate limit of discrete series introduced in
\cite{MR672840,MR678478}.
Let $\underline{\delta} \in \data$ with representative $(M,\delta)$.
Whenever $L$ is a Levi subgroup containing $M$ and the irreducible constituents $\delta'$ of $I^{L}_{M} (\delta)$
are non-degenerate limits of discrete series of the group $L^1$, we call the resulting pairs $(L,\delta')$ \emph{affiliated} with the class $\underline{\delta}$
\cite[D\'{e}finition 2]{MR1046496}.
These representations are precisely those irreducible constituents of the representations $I^{L}_{M} (\delta)$ for $L \supset M$,
which are not themselves irreducibly induced from any smaller Levi subgroup. For fixed $(M, \delta)$, the Levi subgroups $L \supset M$ appearing in this construction
are precisely those for which
$\aaa^*_L$ is the fixed space of $\aaa^*_M$ under one of the subgroups $R_{\delta,\lambda} \subset R_\delta$ (where
we regard $R_\delta$ as a subgroup of $W (A_M)_{\delta} \subset W(A_M) = N_K (A_M) / C_K (A_M)$ as in \cite[\S 2.1]{MR1046496}).

We can then rewrite any representation $\pi_{\delta,\lambda} (\mu)$ in the form
$\pi_{\delta',\lambda}$, where $L \supset M$ is a Levi subgroup with
$\lambda \in \aaa^*_{L,\C} \subset \aaa^*_{M,\C}$ and
$(L,\delta')$ is affiliated with $\data$,
such that the intermediate induction to the largest Levi subgroup $L_{\Re \lambda}$ with $\Re \lambda \in \aaa^*_{L_{\Re \lambda}}$
is irreducible (and tempered). (To see this, combine \cite[6.6.14, 6.6.15]{MR632407} with \cite[(2.1), (2.2)]{MR1046496}.)
Note that the tempered dual of $G_\infty$ can be parametrized as either the set of all basic representations
$\pi_{\delta,\lambda} (\mu)$ with $\Re \lambda = 0$ (which are always irreducible),
or as the set of all \emph{irreducible} induced representations $\pi_{\delta',\lambda}$, $\Re \lambda = 0$,
where $\delta'$ is a non-degenerate limit of discrete series.

Recall the definition of the Paley--Wiener space $\PW (\aaa)_r$, $r > 0$,
of a Euclidean vector space $\aaa$. It is the space of all entire functions $F$ on the complexified dual $\aaa^*_\C$ such that the
Paley--Wiener norms
\[
\norm{F}_{r,n} = \sup_{\lambda \in \aaa^*_\C} (1 + \norm{\lambda})^n e^{-r\norm{\Re \lambda}} \abs{F(\lambda)}, \quad n \ge 0,
\]
are finite. The Paley--Wiener norms endow $\PW (\aaa)_r$ with the structure of a Fr\'{e}chet space and the Fourier transform is
a topological isomorphism between $\PW (\aaa)_r$ and the Fr\'{e}chet space of smooth functions on $\aaa$ supported on the ball of radius $r$
around $0$.

Now let $\underline{\delta} \in \mathcal{D}$.
Consider the finite set $\mathcal{D}' (\underline{\delta})$ of all pairs $(M,\delta)$,
where $M$ is a standard Levi subgroup of $G_\infty$, $\delta \in \Pi (M^1)$ a non-degenerate limit of discrete series,
and $(M,\delta)$ is affiliated with $\underline{\delta}$.
The Paley--Wiener space $\PW_{r,\underline{\delta}}$ is then defined as the space
of all elements $F = (F_{(M,\delta)}) \in \prod_{(M,\delta) \in \mathcal{D}' (\underline{\delta})} \PW (\aaa_M)_r$ fulfilling the following conditions:
\begin{enumerate}
\item Whenever the triples $(M,\delta,\lambda)$ and $(M',\delta',\lambda')$ are conjugate by an element of $K_\infty$, we have
$F_{(M',\delta')} (\lambda') = F_{(M,\delta)} (\lambda)$.
\item Whenever for $M \subset M'$ we have a decomposition
\[
I^{M'}_{M} (\delta_M) = \bigoplus_{i=1}^m \delta^{(i)}_{M'}
\]
with $(M,\delta_M)$, $(M', \delta^{(i)}_{M'})\in \mathcal{D}' (\underline{\delta})$,
the corresponding identity
\[
F_{(M,\delta_M)} (\lambda) = \sum_{i=1}^m F_{(M',\delta^{(i)}_{M'})} (\lambda), \quad \lambda \in \aaa^*_{M',\C} \subset \aaa^*_{M,\C},
\]
holds.
\end{enumerate}
For any finite saturated set $\types \subset \Pi (K_\infty)$ the space $\PW_{r,\types}$ is defined as
$\prod_{\underline{\delta} \in \data_\types} \PW_{r,\underline{\delta}}$.
These Paley--Wiener spaces have in a natural way the structure of Fr\'{e}chet spaces, and we define for each $n \ge 0$ the
Paley--Wiener norm $\norm{F}_{r,n}$ of $F \in \PW_{r,\types}$ to be the maximum of the norms $\norm{F_{(M,\delta)}}_{r,n}$, where
$(M,\delta) \in \mathcal{D}' (\underline{\delta})$, $\underline{\delta} \in \data_\types$.
(Cf. \cite[Appendice C]{MR1046496} for
a concrete combinatorial description of these spaces.)

We can now state the Paley--Wiener theorem of Clozel--Delorme \cite[Th\'{e}or\`{e}me 1, Th\'{e}or\`{e}me 1']{MR1046496}.

\begin{theorem}[Clozel--Delorme] \label{PWTheorem}
For any finite saturated set $\types \subset \Pi (K_\infty)$ and any $r > 0$
the natural continuous map of Fr\'{e}chet spaces $T_{r,\types}: \Hecke (G_\infty)_{r,\types} \to \PW_{r,\types}$
given by $f \mapsto (\tr \pi_{\delta,\lambda} (f))$
is surjective.
\end{theorem}

\begin{remark} \label{PWRemark}
By the open mapping theorem, a continuous surjection of Fr\'{e}chet spaces is automatically open.
For the surjections $T_{r,\mathcal{F}}$ of Theorem \ref{PWTheorem}
this means concretely that for every integer $k\ge0$ there exists an integer $n$ (depending on $k$, $r$ and $\types$)
with the following property: for any $F \in \PW_{r,\types}$
there exists $\phi \in \Hecke (G_\infty)_{r,\types}$ such that $T_{r,\types} (\phi) = F$ and
$\norm{\phi}_k \ll_{r, \types, k} \norm{F}_{r,n}$.
\end{remark}

We now turn to the proof of Proposition \ref{Delormeprop}, which is an extension of
an argument of Delorme (cf.~the proof of \cite[Proposition 3.3]{MR860667}).
As in [ibid.], the proof is based on the existence of certain test functions on $G_\infty$, however,
in comparison to Delorme's argument
we also need to bound
the seminorms of these functions. We therefore recall the construction in some detail.
The first elementary
lemma \cite[Lemma 6.3]{MR532745}
asserts the existence of functions with certain properties of the Fourier transform.

\begin{lemma}[Duistermaat--Kolk--Varadarajan] \label{DKVLemma}
Let $\aaa$ be a Euclidean vector space, $W$ a finite group acting on $\aaa$ and $r > 0$.
Then for any $t \ge 1$ there exists a function 
$\hat{h} (t, \cdot) \in \PW (\aaa)^W_r$ with the following properties.
\begin{enumerate}
\item $\hat{h} (t, \lambda) \in \R^{\ge0}$ for all $\lambda \in \aaa^*_\C$ for which there
exists an element $w \in W$ with $w (\lambda) = - \bar{\lambda}$.
\item $\abs{\hat{h} (t, \lambda)} \ge 1$ for all $\lambda \in \aaa^*_\C$ with $\norm{\lambda} \le t$.
\item For all $m \ge 0$ we have $\norm{\hat{h} (t, \cdot)}_{r,m}\ll_{r,m} t^m$. In particular, for all $a > 0$ and $m \ge 0$ we have
\[
\hat{h} (t, \lambda) \ll_{r,m,a} \frac{t^m}{(1 + \norm{\Im \lambda})^m}
\]
for all $\lambda \in \aaa^*_\C$ with $\norm{\Re \lambda} \le a$.
\end{enumerate}
\end{lemma}

The following lemma is a strengthening of \cite[Proposition 3.2]{MR860667} (we have added the third assertion).

\begin{lemma} \label{phitlemma}
Let $r > 0$,
let $\underline{\delta} \in \mathcal{D}$
with representative $(M,\delta)$, and let $k \ge 0$ be an integer.
Then there exist an integer $r_{k,\underline{\delta}} \ge 0$ (depending on $k$ and $\underline{\delta}$) and for each $t \ge 1$ a function
$\phi^{t,k}_{\underline{\delta}} \in \Hecke (G_\infty)_{r,\types ({\underline{\delta}})}$
with the following properties.
\begin{enumerate}
\item $\tr \sigma_\pi (\phi^{t,k}_{\underline{\delta}}) = 0$ for all $\pi \in \Irr (G_\infty)$ with $\underline{\delta} (\pi) \neq \underline{\delta}$.
\item $\tr \sigma (\phi^{t,k}_{\underline{\delta}}) = [R_\delta : R_{\delta,\lambda}] \hat{h} (t, \lambda)$
for all basic representations $\sigma = \pi_{\delta,\lambda} (\mu)$, $\lambda \in \aaa_{M,\C}^*$, $\mu \in A (\delta)$.
\item $\norm{\phi^{t,k}_{\underline{\delta}}}_k \ll_{k,\underline{\delta}} t^{r_{k,\underline{\delta}}}$.
\end{enumerate}
\end{lemma}

\begin{proof}
We apply the Paley--Wiener theorem to the following element $F^t = F^t_{(M',\delta')}$ of
$\PW_{r,\types}$, where $\types = \types (\underline{\delta})$.
Apply Lemma \ref{DKVLemma} to the vector space $\aaa_M$ and the group $W(A_M) = N_K (A_M) / C_K (A_M)$
to obtain functions $\hat{h} (t, \cdot) \in \PW (\aaa_M)^{W(A_M)}_r$.
Set $F^t_{(M',\delta')} (\lambda) = 0$ whenever $(M',\delta')$ is not affiliated with $\underline{\delta}$.
On the other hand, if  $(M',\delta')$ is affiliated with $\underline{\delta}$, then we
have a decomposition $I^{M'}_{M} (\delta) = \oplus_{i=1}^S \delta'_i$, where $\delta'_1 = \delta'$ and $S$ is a divisor of $|R_\delta|$.
We then set $F^t_{(M',\delta')} (\lambda) = \frac{|R_\delta|}{S} \hat{h} (t, \lambda)$.
By the $W(A_M)$-invariance of $\hat{h} (t,\cdot)$ and the transitivity of induction, this defines an element $F^t$ of $\PW_{r,\types}$.
(Note that the commutativity of the group $R_\delta$ implies that if we have a decomposition $I^{M''}_M (\delta) = \oplus_{i=1}^{S''} \delta''_i$ with
$(M'',\delta''_i)$ affiliated with $\underline{\delta}$ and $M'' \supset M'$, then $S''$ is a multiple of $S$ and each $I^{M''}_{M'} (\delta'_i)$ splits into $\frac{S''}{S}$ many
irreducible constituents $\delta''_j$, cf. \cite[pp. 204-205]{MR1046496}.)
Moreover, by the third assertion of Lemma \ref{DKVLemma}, the Paley--Wiener norms of $F^t$ satisfy $\norm{F^t}_{r,n} \ll_{r,\underline{\delta},n} t^n$ for every $n \ge 0$.
The Paley--Wiener theorem (Theorem \ref{PWTheorem}) and Remark \ref{PWRemark} provide for every $k \ge 0$ an integer $r_{k,\underline{\delta}} \ge 0$ and a preimage
$\phi^{t,k}_{\underline{\delta}} \in \Hecke (G_\infty)_{r,\types}$ of $F^t$ under $T_{r,\types}$ with
$\norm{\phi^{t,k}_{\underline{\delta}}}_k \ll_{r, \underline{\delta}, k} \norm{F^t}_{r,r_{k,\underline{\delta}}} \ll_{r,\underline{\delta},k} t^{r_{k,\underline{\delta}}}$. Therefore,
$\phi^{t,k}_{\underline{\delta}}$ satisfies the third property of the lemma. The first property is clear by construction.
Finally, write $\sigma = \pi_{\delta,\lambda} (\mu) \simeq \pi_{\delta',\lambda}$, where $\delta'$ is a non-degenerate limit of discrete series of $(M')^1$,
$M' \supset M$ is a Levi subgroup with
$\lambda \in \aaa^*_{M',\C} \subset \aaa^*_{M,\C}$ and
$(M',\delta')$ is affiliated with $\data$. The number $S$ of irreducible constituents of $I^{M'}_M (\delta)$ is then equal to $|R_{\delta,\lambda}|$. Therefore
$\tr \sigma (\phi^{t,k}_{\underline{\delta}}) = F^t_{(M',\delta')} (\lambda) = [R_\delta : R_{\delta,\lambda}] \hat{h} (t, \lambda)$, which establishes the second property and finishes the proof.
\end{proof}

\begin{corollary} \label{phitcorollary}
The test functions $\phi^{t,k}_{\underline{\delta}} \in \Hecke (G_\infty)_{r,\types ({\underline{\delta}})}$ above satisfy the following additional properties:
\begin{enumerate}
\item $\tr \sigma_\pi (\phi^{t,k}_{\underline{\delta}}) \ge 1$ for all
$\pi \in \Pi (G_\infty)_{\underline{\delta}}$ with $\abs{\lambda_\pi} \le t^2 - c_{\underline{\delta}}$,
where $c_{\underline{\delta}} \ge 0$ is a constant depending only on ${\underline{\delta}}$.
\item  For all $m \ge 0$ we have
\[
0 \le \tr \sigma_\pi (\phi^{t,k}_{\underline{\delta}}) \ll_{{\underline{\delta}},m} \frac{t^{2m}}{(1 + \abs{\lambda_\pi})^m}
\]
for all $\pi \in \Pi (G_\infty)_{\underline{\delta}}$.
\end{enumerate}
\end{corollary}

\begin{proof}
Let $\bar{\pi}_{\delta,\lambda}(\mu) \simeq \pi \in \Pi (G_\infty)$ and $\sigma_\pi = \pi_{\delta,\lambda} (\mu)$. Since $\pi$ is unitary, we need to have
$w (\lambda) = - \bar{\lambda}$ for an element $w \in W (A_M)_\delta$ \cite[(2.1)]{MR860667}.
By Lemma \ref{phitlemma}, the trace $\tr \sigma_\pi (\phi^{t,k}_{\underline{\delta}})$ is an integer multiple of $\hat{h} (t, \lambda)$.
By the first property of Lemma \ref{DKVLemma}, it is therefore nonnegative real.

Furthermore, the Casimir eigenvalue of $\pi$ can be computed
as $\lambda_\pi = - \norm{\Im \lambda}^2 + \norm{\Re \lambda}^2 + \norm{\chi_\delta}^2 - c_M$ for a constant $c_M$
(cf.~\cite[\S 3.2, (2)]{MR701563}). Again by
unitarity, we have $\norm{\Re \lambda} \le \norm{\rho_P}$, where $\rho_P$ is half the sum of the positive roots of a
parabolic $P$ with Levi subgroup $M$ \cite[(2.2)]{MR860667}. Therefore, we obtain $\abs{\lambda_\pi} \ge \norm{\lambda}^2 - c_{\underline{\delta}}$
for a constant $c_{\underline{\delta}}$, which we may take to be nonnegative.

To show the first assertion, $\abs{\lambda_\pi} \le t^2 - c_{\underline{\delta}}$ implies
that $\norm{\lambda} \le t$, and by the second property of Lemma \ref{DKVLemma} we obtain
$\hat{h} (t, \lambda) \ge 1$.
For the second assertion, we use the last property of Lemma \ref{DKVLemma} (with $2m$ instead of $m$) and the boundedness of $\norm{\Re \lambda}$
together with the obvious fact that $[R_\delta:R_{\delta,\lambda}]$ is bounded by $\abs{R_\delta}$.
\end{proof}

\begin{proof}[End of proof of Proposition \ref{Delormeprop}]
It remains to prove the equivalent statements \eqref{eq: statement1} or \eqref{eq: statement2} for any $\underline{\delta}$ assuming the second condition of the proposition.
We will prove them by induction on $\underline{\delta}$, i.e., for a given $\underline{\delta}$ we
assume that \eqref{eq: statement2} is satisfied for all $\underline{\delta}' \prec \underline{\delta}$ and are going to prove
\eqref{eq: statement1} for $\underline{\delta}$. For this,
consider the test functions $\phi^{t,k}_{\underline{\delta}}$ constructed in Lemma \ref{phitlemma} for
$t = (1 + c_{\underline{\delta}} + R)^{1/2} \ge 1$, where $R > 0$ is a parameter.
By assumption, for a suitable $r > 0$, for each finite set $\types$ the supremum $\sup_{\nu \in \mathfrak{M}} \abs{\nu (\hat{f})}$ is a continuous seminorm
on $\Hecke (G_\infty)_{r,\types}$.
This means that for a suitable value of $k$ (depending on $\types$) we have
$\nu (\hat{f}) \ll_{\types} \norm{f}_k$ for all $\nu \in \mathfrak{M}$ and $f \in \Hecke (G_\infty)_{r,\types}$.
Taking $f = \hat{\phi}^{t,k}_{\underline{\delta}}$ and using the third assertion of Lemma \ref{phitlemma}, we obtain that
there exists an integer $m_{\underline{\delta}} \ge 0$ such that
\begin{equation} \label{eq: bndphitk}
\nu (\hat{\phi}^{t,k}_{\underline{\delta}}) \ll_{\underline{\delta}} (1+R)^{m_{\underline{\delta}}},\ \ \nu \in \mathfrak{M}.
\end{equation}
Write
\[
\nu (\hat{\phi}^{t,k}_{\underline{\delta}}) = \int \tr \pi (\phi^{t,k}_{\underline{\delta}}) d \nu (\pi).
\]
Inserting \eqref{grothendieck} into this equation, we obtain
\[
\nu (\hat{\phi}^{t,k}_{\underline{\delta}})
 = \int (\tr \sigma_\pi) (\phi^{t,k}_{\underline{\delta}}) d \nu (\pi) + \int \left[ \sum_{\pi': {\underline{\delta}}(\pi) \prec {\underline{\delta}} (\pi'), \,
\chi_{\pi} = \chi_{\pi'}} n (\pi,\pi') \tr \sigma_{\pi'} (\phi^{t,k}_{\underline{\delta}}) \right] \, d \nu (\pi).
\]
By the first assertion of Corollary \ref{phitcorollary}, the first integral provides an upper bound for the measure
of the set $\{\pi \in \Pi (G_\infty)_{\underline{\delta}} : \abs{\lambda_\pi} \le R \}$:
\[
\nu (\{\pi \in \Pi (G_\infty)_{\underline{\delta}} : \abs{\lambda_\pi} \le R \}) \le \int (\tr \sigma_\pi)
(\phi^{t,k}_{\underline{\delta}}) d \nu (\pi).
\]
Regarding the second integral,
only $\pi'$ with ${\underline{\delta}} (\pi') = {\underline{\delta}}$ can
contribute, and we can estimate their contribution using
the second assertion of Corollary \ref{phitcorollary}:
\[
0 \le \tr \sigma_{\pi'} (\phi^{t,k}_{\underline{\delta}}) \ll_{{\underline{\delta}},m} \frac{t^{2m}}{(1 + \abs{\lambda_{\pi'}})^m}
= \frac{(1+c_{\underline{\delta}}+R)^{m}}{(1 + \abs{\lambda_{\pi}})^m},
\]
since $\pi$ and $\pi'$ have the same infinitesimal character.
Combining this inequality with \eqref{eq: bndphitk} and using Lemma \ref{lem: vogan} we obtain
\begin{align*}
\nu (\{\pi \in \Pi (G_\infty)_{\underline{\delta}} : \abs{\lambda_\pi} \le R \}) &\le
\nu (\hat{\phi}^{t,k}_{\underline{\delta}}) -
\int \left[ \sum_{\pi': {\underline{\delta}}(\pi) \prec {\underline{\delta}} (\pi'), \,
\chi_{\pi} = \chi_{\pi'}} n (\pi,\pi') \tr \sigma_{\pi'} (\phi^{t,k}_{\underline{\delta}}) \right] \, d \nu (\pi) \\
&\ll_{\underline{\delta},m}
(1+R)^{m_{\underline{\delta}}} + (1+R)^{m}
\sum_{{\underline{\delta}}' \prec {\underline{\delta}}} \nu (g_{m,{\underline{\delta}}'})
\end{align*}
for all $\nu \in \mathfrak{M}$.
For suitable $m$ the sum $\sum_{{\underline{\delta}}' \prec {\underline{\delta}}} \nu (g_{m,{\underline{\delta}}'})$
is bounded independently of $\nu$
by the induction hypothesis. We conclude that \eqref{eq: statement1} holds for ${\underline{\delta}}$, which finishes the induction and thereby the proof of the proposition.
\end{proof}

We remark that the proof simplifies for the groups ${\rm GL} (n)$, since in this case
the tempered basic representations $\pi_{\delta,\lambda}$, $\Re \lambda = 0$, are always irreducible, the $R$-groups are trivial and the
sets $\mathcal{D}' (\underline{\delta})$ are therefore singletons. The Paley--Wiener space $\PW_{r,\underline{\delta}}$ is then just the space of
$W_\delta$-invariant functions in $\PW (\aaa_M)_r$, where $(M,\delta)$ is
a representative of $\underline{\delta}$ and $W_\delta$ denotes the stabilizer of $\delta$ inside the Weyl group $W(A_M)$.

\section{The spectral limit property} \label{sec: spectralside}

We now come back to the situation of \S\S \ref{SectionIntro} -- \ref{SectionBounds}.
Before treating the spectral limit property \eqref{eq: mainspectral}, we consider first the question whether the
collection of measures $\{ \mu^{G,S_\infty}_K \}_{K \in \mathcal{K}}$ on $G (F_\infty)^1$ associated
to a set $\mathcal{K}$ of open subgroups $K$ of $\K_{\fin}$ is \PB.
We conjecture that this is true for the set of all open subgroups of $\K_{\fin}$ (or even for the set of all open compact subgroups of $G (\A_{\fin})$).
Note that each finite set $\mathcal{K}$ is known to have this property \cite{MR1025165}. So, as in the case of property \TWN\
above, the issue is to control the dependence on $K$.

\begin{remark}
Deitmar and Hoffmann \cite{MR1697144} have shown unconditionally that for any $G$ the collection of measures $\{ \mu^{G, S_\infty}_{\K (\nnn),\cusp} \}$, where
$\mu^{G, S_\infty}_{\K (\nnn),\cusp}$ is the analog of $\mu^{G, S_\infty}_{\K (\nnn)}$ for the \emph{cuspidal} spectrum, is \PB.
(In fact, they obtain a more precise statement.) However, for our argument we need to know the corresponding statement for the full discrete spectrum.
\end{remark}

Our results in this direction are Lemmas \ref{AnisotropicLemma} and \ref{BoundedLemma} below, which we will
use for a conditional proof of the spectral limit property for principal congruence subgroups in Corollary \ref{corspectrallimit}, thereby finishing our argument.
Recall the spectral expansion of Theorem \ref{thm: specexpand},
which expresses Arthur's distribution $J (h)$ as a sum of contributions $J_{\spec,M} (h)$ associated to the conjugacy classes of Levi subgroups $M$ of $G$.
Also recall properties \TWN\ and \BD\ from \S \ref{SectionBounds}. They are hereditary for Levi subgroups.




Fix $M \in\levis$, $M \neq G$.
The technical heart of our argument is contained in the following lemma and its corollary. We freely use the notation introduced in \S \ref{SectionTraceFormula}.
We denote by $\norm{\cdot}_{1,\mathfrak{H}}$ the trace norm of an operator on a Hilbert space $\mathfrak{H}$.
Extending the notation of \S \ref{SectionPB},
for a finite set $\types\subset\Pi (\K_\infty)$ and an open subgroup $K_S$ of $\K_{S - S_\infty}$
let $\Hecke(G(F_S)^1)_{\types,K_S}$ be the space of all bi-$K_S$-invariant functions
$f \in \Hecke(G(F_S)^1)$ whose
translates $f (k_1 \cdot k_2)$, $k_1$, $k_2 \in \K_\infty$,
span a space that decomposes under the action
of $\K_\infty \times \K_\infty$ as a sum of representations $\tau_1 \otimes \tau_2$ with $\tau_1$, $\tau_2 \in \types$.
Recall the norms $\norm{\cdot}_k$ on $C^\infty_c (G(F_S)^1)$ introduced in \S \ref{subsecnotationgeom}.

\begin{lemma} \label{lem: mainestimate}
Suppose that $G$ satisfies properties \TWN\ and \BD. Let $M \in \levis$ and $P\in\PPP(M)$.
Furthermore, let $S \supset S_\infty$ be a finite set of places of $F$.
Then for any finite set $\types\subset\Pi (\K_\infty)$ and any sufficiently large $N>0$ there exists an integer $k \ge 0$ such that
for any
\begin{itemize}
\item open subgroup $K_S$ of $\K_{S - S_\infty}$,
\item open compact subgroup $K$ of $G (\A^S)$,
\item $\epsilon>0$,
\item $s\in N_G(M)/M$,
\item $\bss\in\bases_{P,L_s}$,
\item $\dtup\in\Xi_{L_s}(\bss)$,
\end{itemize}
we have
\begin{multline} \label{eq: bnd11}
\int_{\iii(\aaa^G_{L_s})^*}\norm{\Delta_{\dtup}(P,\lambda) M (P,s)
\rho(P,\lambda,h\otimes\one_K)}_{1,\bar\AF^2(P)}\ d\lambda\ll_{\types,N,\epsilon}\\
\vol(K)\level(K_SK;\fctr M^+)^\epsilon\norm{h}_k\sum_{\substack{\tau\in\types,\\\pi\in\Pi_{\disc}(M(\A))}}
\param_M(\pi_\infty)^{-N}\,\dim\AF^2_\pi (P)^{K_SK,\tau}
\end{multline}
for all $h\in\Hecke(G(F_S)^1)_{\types,K_S}$. Consequently,
\begin{multline} \label{eq: JspecMbound}
J_{\spec,M}(h\otimes\one_K)\\\ll_{\types,N,\epsilon}
\vol(K)\level(K_SK;\fctr M^+)^\epsilon\norm{h}_k\sum_{\substack{\tau\in\types,\\\pi\in\Pi_{\disc}(M(\A))}}
\param_M(\pi_\infty)^{-N}\,\dim\AF^2_\pi (P)^{K_SK,\tau}.
\end{multline}
\end{lemma}

Note here that for all $N \ge N_0$, where $N_0$ depends only on $G$, the right hand sides of \eqref{eq: bnd11} and \eqref{eq: JspecMbound} are finite \cite{MR1025165}.

\begin{proof}
We argue as in \cite[\S5]{MR2811597} (cf.~also \cite{MR1935546}). First note that we may omit $M (P,s)$
on the left-hand side of \eqref{eq: bnd11}, since it is a unitary operator which commutes with $\rho(P,\lambda,h\otimes\one_K)$,
and hence does not affect the trace norm.
Let $\Delta$ be the operator $\Id-\Casimir+2\Casimir_{\K_\infty}$,
where $\Casimir$ (resp.~$\Casimir_{\K_\infty}$) is the Casimir operator of $G(F_\infty)$
(resp.~$\K_\infty$).
For any $k>0$ we bound the left-hand side of \eqref{eq: bnd11} by
\begin{multline*}
\int_{\iii(\aaa^G_{L_s})^*}\norm{\Delta_{\dtup}(P,\lambda)
\rho(P,\lambda,\Delta)^{-2k}}_{1,\bar\AF^2(P)^{K_SK,\types}}\norm{\rho(P,\lambda,\Delta^{2k}\star h\otimes\one_K)}\ d\lambda\\\le
\vol(K)\norm{\Delta^{2k}\star h}_{L^1(G(F_S)^1)}\int_{\iii(\aaa^G_{L_s})^*}\norm{\Delta_{\dtup}(P,\lambda)
\rho(P,\lambda,\Delta)^{-2k}}_{1,\bar\AF^2(P)^{K_SK,\types}}\ d\lambda.
\end{multline*}
Consider the integral on the right-hand side.\footnote{In the corresponding formula \cite[(5.1)]{MR2811597} the restriction to the $K_0$-fixed part was mistakenly omitted.}
For any $\pi\in\Pi_{\disc}(M(\A))$ and $\tau\in\Pi (\K_\infty)$,
the operator $\rho(P,\lambda,\Delta)$ acts by the scalar $\mu(\pi,\lambda,\tau)=1+\norm{\lambda}^2-\lambda_\pi+2\lambda_\tau - e_P$
on $\AF^2_\pi(P)^\tau$, where $\lambda_\pi$ and $\lambda_\tau$ are the Casimir eigenvalues of $\pi$ and $\tau$, respectively,
and $e_P$ is a constant depending only on $P$ (cf.~\cite[\S 3.2, (2)]{MR701563}).
Since it is easy to see that $e_P \le 0$, using \eqref{EqnParamMUpper} we get
\[
\mu(\pi,\lambda,\tau)^2\ge\frac14(1+\norm{\lambda}^2+\lambda_\pi^2+\lambda_\tau^2) \gg_G \norm{\lambda}^2+\param_M(\pi_\infty).
\]
Therefore,
\begin{multline*}
\int_{\iii(\aaa^G_{L_s})^*}\norm{\Delta_{\dtup}(P,\lambda)
\rho(P,\lambda,\Delta)^{-2k}}_{1,\bar\AF^2(P)^{K_SK,\types}}\ d\lambda\le\\
\sum_{\tau\in\types}\sum_{\pi\in\Pi_{\disc}(M(\A))}\int_{\iii(\aaa^G_{L_s})^*}
\norm{\Delta_{\dtup}(P,\lambda)}_{1,\AF^2_\pi(P)^{K_SK,\tau}}
\mu(\pi,\lambda,\tau)^{-2k}\ d\lambda.
\end{multline*}
Estimating $\norm{A}_1\le\dim V\norm{A}$ for any linear operator $A$ on a finite-dimensional Hilbert space $V$,
we bound the previous expression by
\[
\sum_{\tau\in\types}\sum_{\pi\in\Pi_{\disc}(M(\A))}\dim\AF^2_\pi (P)^{K_SK,\tau}
\int_{\iii(\aaa^G_{L_s})^*}\norm{\Delta_{\dtup}(P,\lambda)}_{\AF^2_\pi(P)^{K_SK,\tau}}
\mu(\pi,\lambda,\tau)^{-2k}\ d\lambda.
\]
Using the definition of $\Delta_{\dtup}(P,\lambda)$, we can bound the above by a constant multiple of
\begin{equation} \label{eq: bnd33}
\sum_{\substack{\tau\in\types,\\\pi\in\Pi_{\disc}(M(\A))}}
\param_M(\pi_\infty)^{-k/2}\,\dim\AF^2_\pi (P)^{K_SK,\tau}
\int_{\iii(\aaa^G_{L_s})^*}(1+\norm{\lambda})^{-k}\prod_{i=1}^m
\norm{\delta_{P_i|P_i'}(\lambda)\big|_{\AF^2_\pi(P_i')^{K_SK,\tau}}}\ d\lambda.
\end{equation}

We estimate the integral over $\iii(\aaa^G_{L_s})^*$.
Let $\types_M \subset \Pi (\K_{M,\infty})$ be the finite set of all irreducible components of
restrictions of elements of $\types$ to $\K_{M,\infty}$.
Then by Frobenius reciprocity only those $\pi \in \Pi_{\disc}(M(\A))$
with $\pi_\infty \in \Pi (M (F_\infty))^{\types_M}$ can contribute to \eqref{eq: bnd33}.

Let $\bss = (\beta_1^\vee,\dots,\beta_m^\vee)$ and introduce the new coordinates $s_i = \sprod{\lambda}{\beta_i^\vee}$, $i = 1, \dots, m$,
on $(\aaa^G_{L_s,\C})^*$.
By \eqref{eq: normalization} we can write
\[
\delta_{P_i|P_i'}(\lambda)=\frac{n'_{\beta_i}(\pi,s_i)}{n_{\beta_i}(\pi,s_i)}\Id+
j_{P_i'}\circ(\Id\otimes R(\pi,s_i)^{-1}R'(\pi,s_i))\circ j_{P_i'}^{-1}.
\]
Property \TWN\ and Proposition \ref{prop: mainlocal} (which is based on property \BD),
together with Lemma \ref{lem: cmprGm}, yield the estimate
\begin{equation} \label{integralbound}
\int_{\iii(\aaa^G_{L_s})^*}(1+\norm{\lambda})^{-k}\prod_{i=1}^m
\norm{\delta_{P_i|P_i'}(\lambda)\big|_{\AF^2_\pi(P_i')^{K_SK,\tau}}}\ d\lambda
\ll_{\epsilon,\types} \param_M(\pi_\infty;\fctr M)^N \level(K_SK;\fctr M^+)^\epsilon
\end{equation}
for any $\epsilon > 0$ and sufficiently large $N$ and $k$ (depending possibly on $\tau$).
Altogether we obtain \eqref{eq: bnd11}, and using Theorem \ref{thm: specexpand} also \eqref{eq: JspecMbound}.
\end{proof}

\begin{remark}
Note that the improved estimate \eqref{eq: logbnd} yields the following improvement of \eqref{integralbound}:
\begin{multline} \label{eq: intbnd}
\int_{\iii(\aaa^G_{L_s})^*}(1+\norm{\lambda})^{-m-\epsilon}\prod_{i=1}^m
\norm{\delta_{P_i|P_i'}(\lambda)\big|_{\AF^2_\pi(P_i')^{K_SK,\tau}}}
\ d\lambda\\\ll_\epsilon 
\log(\param_M(\pi_\infty; \fctr M)+\norm{\tau}+\level(K_SK;\fctr M^+))^m
\end{multline}
where the implied constant does not depend on $\tau$.
\end{remark}

Applying Lemma \ref{lem: mainestimate} to the principal congruence subgroups $\K^S (\nnn)$, and assuming polynomial boundedness
of the collection $\{ \mu^{M, S_\infty}_{\K_M (\nnn)} \}$, we obtain the following result.

\begin{corollary} \label{cor: SpectralEstimate}
Suppose that $G$ satisfies properties \TWN\ and \BD.
Furthermore, let $M \in\levis$, $M \neq G$, and assume
that the set of measures $\{ \mu^{M, S_\infty}_{\K_M (\nnn)} \}$ is \PB.
Let $S \supset S_\infty$ be a finite set of places of $F$.
Then for any finite set $\types\subset\Pi (\K_\infty)$ there exists an integer $k \ge 1$ such that for any
open subgroup $K_S \subset \K_{S - S_\infty}$ and any $\epsilon>0$, we have
\begin{equation} \label{eq: bnd111}
J_{\spec,M}(h\otimes\one_{\K^S (\nnn)})
\ll_{K_S,\types, \epsilon}\norm{h}_k \inorm (\nnn)^{(\dim M-\dim G)/2+\epsilon}
\end{equation}
for all $h\in\Hecke(G(F_S)^1)_{\types,K_S}$ and all integral ideals $\nnn$ of $\mathfrak{o}_F$ prime to $S$.
\end{corollary}

\begin{proof}
Fix $P=M\ltimes U\in\PPP(M)$ and an ideal $\nnn_0$ such that $K_S \supset \K_{S - S_\infty} (\nnn_0)$.
We have
\begin{align*}
\dim\AF^2_\pi (P)^{\K (\nnn_0 \nnn),\tau} & = m_\pi \dim \Ind^{G(\A)}_{P(\A)}(\pi)^{\K (\nnn_0 \nnn),\tau} \\
& = m_\pi \dim \Ind^{G(F_\infty)}_{P(F_\infty)}(\pi_\infty)^{\tau} \dim \Ind^{G(\A_{\fin})}_{P (\A_{\fin})} (\pi_{\fin})^{\K (\nnn_0 \nnn)},
\end{align*}
where
\[
m_\pi=\dim\Hom(\pi,L^2_{\disc}(A_MM(F)\bs M(\A))).
\]
Note that here the factor $\dim \Ind^{G(F_\infty)}_{P(F_\infty)}(\pi_\infty)^{\tau}$ is bounded by $(\dim \tau)^2$.
On the other hand, since $\K (\nnn_0 \nnn)$ is a normal subgroup of $\K_{\fin}$, we have
\[
\dim \Ind^{G(\A_{\fin})}_{P (\A_{\fin})} (\pi_{\fin})^{\K (\nnn_0 \nnn)}
\le [\K_{\fin}:(\K_{\fin}\cap P(\A_{\fin})) \K (\nnn_0 \nnn)]\dim \pi_{\fin}^{\K_M (\nnn_0 \nnn)}.
\]
Using the factorization $\K_{\fin}\cap P(\A_{\fin}) = (\K_{\fin}\cap M(\A_{\fin})) (\K_{\fin}\cap U(\A_{\fin}))$, we can write
\begin{multline*}
[\K_{\fin}:(\K_{\fin}\cap P(\A_{\fin})) \K (\nnn_0 \nnn)]=\vol (\K_M (\nnn_0 \nnn))\vol (\K (\nnn_0 \nnn))^{-1}\\
 [\K (\nnn_0 \nnn) \cap P(\A_{\fin}) : (\K (\nnn_0 \nnn) \cap M(\A_{\fin})) (\K (\nnn_0 \nnn) \cap U(\A_{\fin}))]
 [\K_{\fin} \cap U (\A_{\fin}) : \K (\nnn_0 \nnn) \cap U(\A_{\fin}) ]^{-1}.
\end{multline*}
The index $[\K (\nnn_0 \nnn) \cap P(\A_{\fin}) : (\K (\nnn_0 \nnn) \cap M(\A_{\fin})) (\K (\nnn_0 \nnn) \cap U(\A_{\fin}))]$ is bounded
independently of $\nnn$.
Furthermore, identifying $U$ with its Lie algebra $\mathfrak{u}$ via the exponential map, which is an isomorphism of affine varieties, one sees that
\[
\inorm (\nnn_0 \nnn)^{-\dim U}\ll [\K_{\fin} \cap U (\A_{\fin}) : \K (\nnn_0 \nnn) \cap U(\A_{\fin}) ]^{-1} \ll \inorm (\nnn_0 \nnn)^{-\dim U}.
\]
(We will only need the upper bound.)
Therefore
\[
\dim \Ind^{G(\A_{\fin})}_{P (\A_{\fin})} (\pi_{\fin})^{\K (\nnn_0 \nnn)}
\ll  \inorm (\nnn_0 \nnn)^{-\dim U} \vol (\K (\nnn_0 \nnn))^{-1}  \vol (\K_M (\nnn_0 \nnn)) \dim \pi_{\fin}^{\K_M (\nnn_0 \nnn)}.
\]


Incorporating the above into Lemma \ref{lem: mainestimate}, we obtain that for sufficiently large $N>0$ there exists $k \ge 0$ such that
\begin{multline*}
J_{\spec,M}(h\otimes\one_{\K^S (\nnn)})
\ll_{\nnn_0, \types,\epsilon,N} \norm{h}_k\\ \inorm (\nnn)^{-\dim U+\epsilon} \vol (\K_M (\nnn_0 \nnn)) \sum_{\pi\in\Pi_{\disc}(M(\A))^{\types_M}}
(1 + \abs{\lambda_{\pi_\infty}})^{-N} m_\pi \dim \pi_{\fin}^{\K_M (\nnn_0 \nnn)}.
\end{multline*}
By assumption, the set of measures $\{ \mu^{M, S_\infty}_{\K_M (\nnn_0 \nnn)} \}$ is \PB.
Therefore the fourth condition of Proposition \ref{Delormeprop} yields the existence of an integer $N$, depending only on $\types_M$, such that
\[
\mu^{M, S_\infty}_{\K_M (\nnn_0 \nnn)} (g_{N,\types_M}) = \frac{\vol (\K_M (\nnn_0 \nnn))}{\vol(M(F)\bs M(\A)^1)}
\sum_{\pi\in\Pi_{\disc}(M(\A))^{\types_M}}
(1 + \abs{\lambda_{\pi_\infty}})^{-N} m_\pi \dim \pi_{\fin}^{\K_M (\nnn_0 \nnn)}
\]
is bounded independently of $\nnn$. This proves the assertion since $\dim U=(\dim G-\dim M)/2$.
\end{proof}

\begin{remark} \label{rem: logsadf}
As before, \eqref{eq: logbnd} implies a slightly improved version of \eqref{eq: bnd111}
in which the expression $\inorm (\nnn)^{(\dim M-\dim G)/2+\epsilon}$ is replaced by
$(1 + \log \inorm(\nnn))^m \inorm (\nnn)^{(\dim M-\dim G)/2}$.
\end{remark}

We are now in a position to prove that under assumptions \TWN\ and \BD\ for $G$ the collection of measures
$\{ \mu^{M, S_\infty}_{\K_M (\nnn)} \}$ on $\Pi (M(F_\infty)^1)$
is \PB\ for any $M \in \levis$.\footnote{Variants of Lemma \ref{AnisotropicLemma} have been previously
established in \cite[Proposition 3.3]{MR860667} and \cite{MR1697144}.}

\begin{lemma} \label{AnisotropicLemma}
Let $G$ be anisotropic modulo the center. Then the collection of measures
$\{ \mu^{G, S_\infty}_K \}$, where $K$ ranges over the open subgroups of $\K_{\fin}$, is \PB.
\end{lemma}

\begin{proof}
In this case, the trace formula for a test function $h \in C^\infty_c (G(F_\infty)^1)$ can be written as
\[
\vol(G(F)\bs G(\A)^1) \mu^{G, S_\infty}_{K} (\hat{h}) =
\int_{G(F)\bs G (\A)^1} \sum_{\gamma\in G(F)}(h \otimes \one_K) (g^{-1} \gamma g) dg.
\]
Since $g$ can be integrated over a compact set $C$ which is independent of $h$ and $K$,
for all $h$ supported in a fixed set $\Omega_\infty \subset G(F_\infty)^1$
we can bound the absolute value of the right-hand side
by $A_{\Omega_\infty}\sup\abs{h}$, where
$A_{\Omega_\infty}$ is the constant
\[
A_{\Omega_\infty}=\vol(G(F)\bs G(\A)^1)\sup_{g\in C}\sum_{\gamma\in G(F)}(\one_{\Omega_\infty \K_{\fin}}) (g^{-1} \gamma g).
\]
Therefore, $\sup \abs{\mu^{G, S_\infty}_{K} (\hat{h})}$ is a continuous seminorm on every space $\Hecke (G(F_\infty)^1)_{r,\types}$,
and by Proposition \ref{Delormeprop} we obtain the assertion.
\end{proof}

\begin{lemma} \label{BoundedLemma}
Suppose that $G$ satisfies \TWN\ and \BD.
Then for each $M \in \levis$ the collection of measures $\{ \mu^{M, S_\infty}_{\K_M (\nnn)} \}$, $\nnn$ ranging over the integral ideals of $\mathfrak{o}_F$, is \PB.
\end{lemma}

\begin{proof}
We use induction on the semisimple rank of $M$.
The base of the induction is Lemma \ref{AnisotropicLemma}.
For the induction step, we can assume the assertion for all groups in $\levis  \bs \{ G \}$ and have the task to establish it for $G$ itself.
Fix $r > 0$ and apply the trace formula to $h \otimes \one_{\K (\nnn)}$, where $h \in \Hecke (G(F_\infty)^1)_{r,\types}$.
By Theorem \ref{thm: specexpand}, we have
\[
\vol(G(F)\bs G(\A)^1) \mu^{G, S_\infty}_{\K (\nnn)} (\hat{h}) = J (h \otimes \one_{\K (\nnn)}) - \sum_{[M], \, M \neq G} J_{\spec,M} (h \otimes \one_{\K (\nnn)}).
\]
Now, for each single choice of $\nnn$ the absolute value $\abs{J (h \otimes \one_{\K (\nnn)})}$ is a continuous seminorm by the work of Arthur \cite{MR518111}.
Moreover, as in the proof of Corollary \ref{corgeometriclimit}, for all $\nnn$ outside of a finite set depending only on $r$ we have
$J (h \otimes \one_{\K (\nnn)}) = J_{\unip} (h \otimes \one_{\K (\nnn)})$. Therefore it follows from our analysis of $J_{\unip} (h \otimes \one_{\K (\nnn)})$
in Proposition \ref{unipproposition} that
$\sup_\nnn \abs{J (h \otimes \one_{\K (\nnn)})}$ is a continuous seminorm on $\Hecke (G(F_\infty)^1)_{r,\types}$.
By Corollary \ref{cor: SpectralEstimate} (with $S=S_\infty$),
we obtain that the spectral terms $\sup_\nnn \abs{J_{\spec,M} (h \otimes \one_{\K (\nnn)})}$ for $M \neq G$
are also continuous seminorms on $\Hecke (G(F_\infty)^1)_{r,\types}$. By Proposition \ref{Delormeprop}, we conclude
that the collection $\{ \mu^{G, S_\infty}_{\K (\nnn)} \}$ is \PB.
\end{proof}

As before, let $S$ be a finite set of places of the field $F$ containing $S_\infty$.

\begin{corollary}[Spectral limit property] \label{corspectrallimit}
Suppose that $G$ satisfies \TWN\ and \BD.
Then we have the spectral limit property for the set of subgroups $\K^S (\nnn)$, where $\nnn$ ranges over the integral ideals of $\mathfrak{o}_F$ prime to $S$.
\end{corollary}

\begin{proof}
From Lemma \ref{BoundedLemma} we get that for each $M \in \levis$, the collection of measures $\{ \mu^{M, S_\infty}_{\K_M (\nnn)} \}$ is \PB.
Therefore we can apply Corollary \ref{cor: SpectralEstimate} to conclude that for each $h \in \Hecke (G(F_S)^1)$ we have
\[
J_{\spec,M} (h \otimes \one_{\K^S (\nnn)}) \to 0
\]
for all $M \neq G$. Hence (by Theorem \ref{thm: specexpand})
\[
J (h \otimes \one_{\K^S (\nnn)}) - \tr R_{\disc} (h \otimes \one_{\K^S (\nnn)}) \to 0,
\]
which is the spectral limit property \eqref{eq: mainspectral}.
\end{proof}

\begin{theorem} \label{MainTheorem}
Suppose that $G$ satisfies \TWN\ and \BD.
Then limit multiplicity holds for the set of subgroups $\K^S (\nnn)$, where $\nnn$ ranges over the integral ideals of $\mathfrak{o}_F$ prime to $S$.
\end{theorem}

\begin{proof}
The geometric limit property \eqref{eq: maingeometric} has been established in Corollary \ref{corgeometriclimit}, and the spectral limit property \eqref{eq: mainspectral} in Corollary
\ref{corspectrallimit}.
By Theorem \ref{thm: reduction}, we obtain the result.
\end{proof}

This also finishes the proof of Theorem \ref{thm: main}, since \TWN\ (resp., \BD) have been verified for the groups $\GL (n)$
and $\SL(n)$ in Proposition \ref{prop: mainglobal} (resp., Theorem \ref{TheoremLocalGLn}).

Repeating the argument above and combining Corollary \ref{cor: SpectralEstimate} with the improved geometric estimate of Proposition \ref{unippropositionrefined} we obtain the following quantitative statement.
Recall the definition of $\dmin$ in \eqref{equationdmin}.

\begin{theorem} \label{theoremquantitative}
Suppose that $G$ satisfies \TWN\ and \BD.
Then for any $r>0$ and any finite set $\types\subset\Pi (\K_\infty)$
there exists an integer $k \ge 0$ such that
for any open subgroup $K_S \subset \K_{S - S_\infty}$ and $\epsilon>0$ we have
\begin{equation} \label{quantitative}
\abs{\mu^{G, S}_{\K^S (\nnn)} (\hat{h}) - h (1)} \ll_{r,K_S,\types,\epsilon} \inorm (\nnn)^{-\dmin +\epsilon} \norm{h}_k
\end{equation}
for all $h\in\Hecke(G(F_S)^1)_{r,\types,K_S}$ and all integral ideals $\nnn$ of $\mathfrak{o}_F$ prime to $S$.
\end{theorem}
%

Note here that $\dmin\le(\dim G-\dim M)/2$ for any proper $M\in\levis$, since $\dim G-\dim M$ is the dimension of the
Richardson orbit associated to a parabolic subgroup $P\in\PPP(M)$.
If we also assume \eqref{eq: logbnd} then we can further improve the right hand side of \eqref{quantitative} to
$\frac{(1 + \log \inorm (\nnn))^{d_0}}{\inorm (\nnn)^\dmin} \norm{h}_k$ (see Remark \ref{rem: logsadf}).

\begin{remark}
A natural problem is to deduce from Theorem \ref{theoremquantitative} an estimate for the difference $\abs{\mu^{G, S}_{\K^S (\nnn)} (A) - \mu_{\plnch} (A)}$
for suitable subsets $A \subset \Pi (G)$. This would require a quantitative version
of the density principle (Theorem \ref{thm: density principle}). We will not discuss this aspect here.
\end{remark}

\newcommand{\etalchar}[1]{$^{#1}$}

\providecommand{\bysame}{\leavevmode\hbox to3em{\hrulefill}\thinspace}
\providecommand{\MR}{\relax\ifhmode\unskip\space\fi MR }
\providecommand{\MRhref}[2]{%
  \href{http://www.ams.org/mathscinet-getitem?mr=#1}{#2}
}
\providecommand{\href}[2]{#2}

\end{document}